\documentclass[reqno]{amsart}

\usepackage{float}
\usepackage[utf8]{inputenc}

\usepackage{capt-of}
\usepackage[usenames,dvipsnames,svgnames,table]{xcolor}
\usepackage[colorlinks=true, pdfstartview=FitV, linkcolor=blue, citecolor=blue, urlcolor=darkblue]{hyperref}

\usepackage{geometry}                % See geometry.pdf to learn the layout options. There are lots.
\usepackage{graphicx}
\usepackage{mathrsfs}
\usepackage{amssymb}
\usepackage{amsmath}
\usepackage{amsfonts}
\usepackage{mathrsfs}
\usepackage{epstopdf}
\usepackage{lscape}
\usepackage{array}
\usepackage[utf8]{inputenc}
\usepackage{tikz,caption}
\usepackage{biblatex}
\usepackage{enumitem}
\setlist[itemize]{leftmargin=2em}
\setlist[enumerate]{leftmargin=2em}

\usepackage{booktabs}
\usepackage{multirow}
\usepackage{mathtools}
\usepackage[linesnumbered,ruled]{algorithm2e}
\usepackage{tikz}
\usepackage{mathtools}
\usepackage{soul}
\usepackage{upgreek}

\newtheorem{theorem}{Theorem}[section]
\newtheorem{lemma}[theorem]{Lemma}
\newtheorem{corollary}[theorem]{Corollary}

\newtheorem{proposition}[theorem]{Proposition}

\theoremstyle{definition}
\newtheorem{definition}[theorem]{Definition}
\newtheorem{example}[theorem]{Example}

\newtheorem{remark}[theorem]{Remark}

\newcolumntype{C}[1]{>{\centering\arraybackslash}m{#1}}

\newcommand{\C}{\mathbb{C}}

\newcommand{\Q}{\mathbb{Q}}

\DeclareMathOperator{\spn}{span}
\newcommand{\rank}{\mathrm{rank}}
\newcommand{\csff}{\mathbf{X}_F}
\newcommand{\csft}{\mathbf{X}_T}
\newcommand{\stfrak}{\mathfrak{st}}
\newcommand{\lmulti }{\{\!\!\{}
\newcommand{\rmulti}{\}\!\!\}}
\newcommand{\diam}{\mathrm{diam}}
\newcommand{\sort}{\mathrm{sort}}
\newcommand{\lc}{\lambda_\mathrm{LC}}
\newcommand{\lead}{\lambda_\mathrm{lead}}

\definecolor{darkblue}{rgb}{0.0,0,0.7} % darkblue color
\definecolor{darkred}{rgb}{0.7,0,0} % darkred color
\definecolor{darkgreen}{rgb}{0, .6, 0} % darkgreen color

\newcommand{\defncolor}{\color{darkred}}
\newcommand{\defn}[1]{{\defncolor\emph{#1}}} % emphasis of a definition

\usepackage[colorinlistoftodos]{todonotes}

\newcommand{\tikzmath}[2][]{\vcenter{\hbox{\begin{tikzpicture}[#1]#2
\end{tikzpicture}}}
}

\title{The chromatic symmetric function in the star-basis}

\author[Gonzalez]{Michael Gonzalez}
\address[M. Gonzalez]{Mathematics Department, Dartmouth College, 
Hanover, NH 03755, U.S.A.}
\email{mikegonz1130@gmail.com}

\author[Orellana]{Rosa Orellana}
\address[R. Orellana]{Mathematics Department, Dartmouth College, 
Hanover, NH 03755, U.S.A.}
\email{Rosa.C.Orellana@dartmouth.edu}
\urladdr{\href{https://math.dartmouth.edu/~orellana/}{https://math.dartmouth.edu/~orellana/}}

\author[Tomba]{Mario Tomba}
\address[M. Tomba]{Mathematics Department, Dartmouth College, 
Hanover, NH 03755, U.S.A.}
\email{mario.tomba.morillo.25@dartmouth.edu}

\begin{document}

\begin{abstract}
    We study Stanley's chromatic symmetric function (CSF) for trees when expressed in the star-basis.  We use the deletion-near-contraction algorithm recently introduced in \cite{ADOZ} to compute coefficients that occur in the CSF in the star-basis. In particular, one of our main results determines the smallest partition in lexicographic order that occurs as an indexing partition in the CSF, and we also give a formula for its coefficient. In addition to describing properties of trees encoded in the coefficients of the star-basis, we give two main applications of the leading coefficient result. The first is a strengthening of the result in \cite{ADOZ} that says that proper trees of diameter less than or equal to 5 can be reconstructed from their CSFs. In this paper we show that this is true for all trees of diameter less than or equal 5. In our second application, we show that the dimension of the subspace of symmetric functions spanned by the CSF of $n$-vertex trees is $p(n)-n+1$, where $p(n)$ is the number of partitions of $n$.  
\end{abstract}
\maketitle

%%%%%%%%%%%%%%%%%%%%%%%%%%%%%%%%%%%%%%%%%%%%%%%%%%%%%%%%%%%%%%
\section{Introduction}
%%%%%%%%%%%%%%%%%%%%%%%%%%%%%%%%%%%%%%%%%%%%%%%%%%%%%%%%%%%%%%

In 1995, Stanley \cite{stanley95symmetric} introduced a symmetric function for a simple graph $G$ with vertex set $V=\{v_1, \ldots, v_n\}$. Let $x_1, x_2, \ldots$  be commuting variables, then the symmetric function associated to the graph $G$ is defined by
\[\mathbf{X}_G := \sum_{\kappa}x_{\kappa(v_1)}x_{\kappa(v_2)}\cdots x_{\kappa(v_n)}, \]
where the sum runs over all proper colorings $\kappa: V \rightarrow \mathbb{N}$. The function $\mathbf{X}_G$ is known as the \emph{chromatic symmetric function} of $G$. By setting $x_i=1$, for $1\leq i\leq k$, and $x_j=0$, for $j>k$, we recover $\chi_G(k)$, the one variable chromatic polynomial, which counts the number of proper colorings of $G$ with $k$ colors.  In his seminal paper, Stanley expressed $\mathbf{X}_G$ using the classical bases of symmetric 
functions, proved many results and made several conjectures related to $\mathbf{X}_G$. This function has attracted a lot of interest, see for example \cite{aliste2014proper, crew2020deletion, Crew-distinguishing, gasharov1999stanley, DSvW-epos, guay2013modular,  heil2019algorithm, loebl2018isomorphism, martin2008distinguishing, OrellanaScott,shareshian2016chromatic}.  

Most of the research related to $\mathbf{X}_G$ revolves around two main conjectures: The \emph{$e$-Positivity Conjecture}, \cite{stanley95symmetric} which states that if a poset is (3+1)-free then its incomparability graph is a nonnegative linear combination of elementary symmetric functions \cite{CH-epos,DSvW-epos,DFvW-epos,guay2013modular,HP-epos}; and the \emph{Tree Isomorphism Conjecture}, which states that the chromatic symmetric function distinguishes non-isomorphic trees. This conjecture is known to hold for trees with less than 30 vertices  \cite{heil2019algorithm} and it has been verified for several subclasses of trees \cite{ADOZ,aliste2014proper,loebl2018isomorphism,
martin2008distinguishing}. Another problem that has drawn attention is in finding families of graphs for which $\mathbf{X}_G$ is Schur positive, \cite{DSvW-epos,G-spos, TW-spos}, mainly because of connections to the representation theory of the general linear and symmetric groups. 

Another line of research has involved generalizations of the chromatic symmetric function which contain $\mathbf{X}_G$ as a specialization, for example the $q$-quasisymmetric function of Shareshian and Wachs \cite{shareshian2016chromatic}, non-commutative versions, \cite{GS-noncom}, a rooted version, \cite{LW-rootedversion}, and weighted versions \cite{ADOZ,crew2020deletion}. Most of these generalizations have been introduced as an approach to the two main open problems. 

In proving results about $\mathbf{X}_G$, we often use the \emph{modular relation} or \emph{triple deletion property} in \cite{guay2013modular,OrellanaScott} which is a recursive formula satisfied by $\mathbf{X}_G$. Another method involves writing $\mathbf{X}_G$ as a linear combination in the classical bases of symmetric functions. In a recent paper, Aliste-Prieto, De-Mier, Zamora and the second author \cite{ADOZ} introduced the \emph{Deletion-Near-Contraction} algorithm to efficiently compute the CSF of a graph in the star basis. The deletion-near-contraction relation was first established for chord diagrams by Chmutov, Duzhin, and Lando in \cite{chmutov}, and was later proved to be intimately related to the $W$-polynomial of Noble and Welsh and Stanley's chromatic symmetric function \cite{noble99weighted}.

Several results related to the Tree Isomorphism conjecture have been derived by relating $\mathbf{X}_G$ to other polynomials, such as the \emph{subtree polynomial} of Chaudhary and Gordon \cite{CG-graphpoly}, the W-polynomial of Noble and Welsh \cite{noble99weighted}, the Tutte polynomial, and others. For example, Martin, Morin, 
and Wagner, \cite{martin2008distinguishing}, proved a number of results related to $\csft$, for a tree $T$, using the subtree polynomial, in particular, they showed that the diameter of a tree can be recovered from $\csft$. In another example, Loebl and Sereni, \cite{loebl2018isomorphism}, showed that caterpillars are distinguished by the chromatic symmetric function using the W-polynomial.  

Cho and van Willigenburg \cite{chromBases} defined multiplicative bases of symmetric functions from any sequence of connected graphs $(G_n)_{n\geq 1}$ such that for each $n$ the graph $G_n$ has $n$ vertices. In this article we are interested in working with the basis constructed from star graphs, that is, when $G_n$ is the tree with one vertex of degree $n-1$ and all other vertices having degree $1$. We denote the star-basis of symmetric functions of homogeneous symmetric functions of degree $n$ by $\{\mathfrak{st}_\lambda\, :\, \text{$\lambda$ is a partition of $n$}\}$. Hence, 
 \[\mathbf{X}_G = \sum_{\lambda\vdash n} c_\lambda \mathfrak{st}_\lambda,\]
where $\lambda\vdash n$ denotes that $\lambda$ is a partition of $n$.  In this paper we give formulas for some of the coefficients $c_\lambda$ when $G = T$ is a tree. Our main result, Theorem \ref{thm:leading-partition}, is the identification of the smallest partition that occurs with nonzero coefficient, $\lead(\csft)$, when we order the partitions using lexicographic order. We call this partition the \emph{leading partition} of $\csft$. In addition, we give a formula for the coefficient $c_{\lead(\csft)}$ in Theorem \ref{thm:leading-coefficient}.  Further, we explain the information about $T$ encoded in $\lead(\csft)$ and its coefficient.  Our work with the star-basis indicates that the indexing partitions encode information about edge adjacencies, see Proposition \ref{prop:k-edge-adjacencies}, while the coefficients seem to encode information about vertex degrees and the number of non-leaf-edges, see Proposition \ref{prop:hook-coeff} and Theorem \ref{thm:leading-coefficient}. 

We give two applications of our main result about the leading partition. The first application is a strengthening of the result in \cite{ADOZ} which shows that proper trees of diameter less than or equal to 5 are distinguished by their chromatic symmetric function.  We show that \emph{all} trees of diameter less than or equal to 5 are distinguished by the chromatic symmetric function. In Theorem \ref{thm:diam-4-reconstruction} we show that trees of diameter four can be reconstructed from $\csft$; and in Proposition \ref{reconstruction diam 5 different orders} and Theorem \ref{thm:p-balanced diameter 5 reconstruction} we show that those of diameter five are also reconstructible from $\csft$.  Furthermore, in Corollary \ref{cor:proper-tree-distinct-parts}, we prove that proper trees whose leading partition have distinct parts can be reconstructed from $\csft$, as well as extended bi-stars, which are special cases of caterpillars that can be reconstructed directly from the leading partition in Corollaries \ref{cor:bi-stars-leading} and \ref{cor:bi-stars-distinguish}. Our proofs are algorithmic, in the sense that they describe an algorithm on how to reconstruct the tree from its chromatic symmetric function. 
 
For the other application, we define a subspace of the space of homogeneous symmetric functions of degree $n$,
\[\mathcal{V}_n = \text{Span}\{ \mathbf{X}_T\, : \, T \text{ is an $n$-vertex tree}\}.\]
We use Theorem \ref{thm:leading-partition} to show that the dimension of this subspace is $p(n) - n +1$, where $p(n)$ is the number of partitions of $n$, see Theorem \ref{thm:dimension}.  In addition, we give a construction for a basis of caterpillars.  Hence, for any tree $\csft$ is a linear combination of the chromatic symmetric functions of caterpillars. We hope that this basis will give a new approach to the Tree Isomorphism Conjecture. 

This paper is organized as follows: In Section 2, we review some definitions and results about graph theory and symmetric functions. 
In Section 3, we review the deletion-near-contraction rule and give the  algorithm for computing the chromatic symmetric function in the star-basis. In addition, we give some basic results about the coefficients that occur.  In Section 4, we present our main result about the leading partition and its coefficient. In Section 5, we prove that the Tree Isomorphism conjecture is true for all trees of diameter less than 6. Finally, in Section 6, we determine the dimension of the subspace spanned by the chromatic symmetric functions of trees and give a basis indexed by caterpillars. 

%%%%%%%%%%%%%%%%%%%%%%%%%%%%%%%%%%%%%%%%%%%%%%%%%%%%%%%%%%%%%%
\subsection*{Acknowledgements} MG was supported by the EE Just fellowship at Dartmouth College, RO was partially supported by NSF grant DMS--2153998, and MT was partially supported by the URAD office at Dartmouth College. We thank Andrew Koulogeorge for discussions and initial interest in this project. We also thank the anonymous referees for valuable suggestions that improved the paper, in particular for a shorter proof of Proposition \ref{prop:numberinternaledges-in-IntSub}.
%%%%%%%%%%%%%%%%%%%%%%%%%%%%%%%%%%%%%%%%%%%%%%%%%%%%%%%%%%%%%%
\section{Preliminaries and definitions}
%%%%%%%%%%%%%%%%%%%%%%%%%%%%%%%%%%%%%%%%%%%%%%%%%%%%%%%%%%%%%%
%%%%%%%%%%%%%%%%%%%%%%%%%%%%%%%%%%%%%%%%%%%%%%%%%%%%%%%%%%%%%%
\subsection{Graph theory}
     For background details on graph theory, we refer the reader to an introductory graph theory textbook such as \cite{westBook}.  For convenience and to set up the notation needed for the paper, we include some of the basic concepts here. 

    A \defn{graph} $G$ is a pair $(V,E)$, where $V$ is the set of \defn{vertices} and $E$ is a multiset of \defn{edges}. The \defn{order} of a graph is the number of vertices, i.e., $|V|$. An edge $e$ is an unordered pair of vertices $uv$, in which case $u$ and $v$ are its \defn{endpoints}. The \defn{degree} of a vertex $v$, denoted by $\deg(v)$, is the number of edges incident to it. A \defn{loop} is an edge with equal endpoints. The \defn{neighborhood} of a vertex $v$ is the set of vertices that are adjacent to $v$ and is denoted by $N(v)$. If $G$ does not contain loops or multiple edges between two vertices, we say that $G$ is \defn{simple}. In the remainder of the paper, we restrict our attention exclusively to simple graphs and, for that reason, we refer to them simply as graphs. 
    
    Two graphs $G, H$ are \defn{isomorphic} if there exists a bijection $f:V(G) \to V(H)$ between the vertex sets of the graphs such that $uv \in E(G)$ if and only if $f(u)f(v) \in E(H)$. A \defn{path} in a graph $G$ is a sequence of distinct vertices $v_1,v_2, \ldots,v_n$, such that $v_iv_{i+1}\in E(G)$ for all $1\leq i\leq n-1$. We call $v_1$ and  $v_n$ the endpoints of the path. The \defn{length} of a path is the number of edges in the path.  We say that $G$ is \defn{connected} if there is a path between any two vertices. A \defn{cycle} in $G$ is a sequence of distinct edges in which only the first and last vertex are equal.
    
    A \defn{tree} $T$ is an acyclic connected graph. The \defn{diameter} of a tree $T$, denoted by $\diam(T)$, is the maximum length of a path between any pair of vertices in $T$. We say that a tree is \defn{proper} if every non-leaf-vertex has at least one leaf in its neighborhood. A \defn{forest} is a graph where each connected component is a tree.

    Let $\mathbb{P}$ denote the positive integers. A \defn{proper coloring} of $G$ is a function $\kappa : V \to \mathbb{P}$ such that, for any $uv \in E$, it satisfies $\kappa(u) \neq \kappa(v)$. If $\kappa(V)$ is the image of $\kappa$, i.e., the set of colors used and $|\kappa(V)| = k$, then we say that $\kappa$ is a \defn{$k$-coloring}. The number of $k$-colorings is denoted by $\chi_G(k)$, and it is a basic fact in graph theory that $\chi_G(k)$ is a polynomial of degree $|V|$, called the \defn{chromatic polynomial}. It is well known that the chromatic polynomial satisfies the deletion-contraction formula. That is, for a non-loop edge $e$, we have:
    \[
    \chi_G(k) = \chi_{G \setminus e}(k) - \chi_{G/e}(k)~,
    \]
    where $G \setminus e$ and $G/e$ are the graphs obtained by deleting and contracting $e$, respectively. We remark that when we contract an edge in a simple graph, the resulting graph might contain multiple edges. In this case, the convention is that multiple edges resulting from the contraction in $G/e$ are replaced with a single edge. In Sections 4 through 6, we will restrict our attention to forests and trees, in which case contraction never produces multiple edges.

    We now introduce some definitions and notation around special types of edges and vertices that we will use throughout the paper.  An \defn{internal vertex} is a vertex of degree at least two. A \defn{leaf vertex}, or simply leaf, is a vertex of degree one. An \defn{internal edge} is a non-loop edge such that both its endpoints are internal vertices, and a \defn{leaf-edge} is an edge such that one of its endpoints is a leaf. We denote by $I(G)$ the set of internal edges of $G$ and by $IV(G)$ the set of internal vertices of $G$.
    \begin{figure}[ht]
        \centering
        \begin{tikzpicture}[auto=center,every node/.style={circle, fill=black, scale=0.5}]
            \node (a) at (1,0) {};
            \node (b) at (0,.5) {};
            \node (c) at (0,-.5) {};
            \node (d) at (-1,.5) {};
            \node (e) at (-1,-.5) {};
            \node (f) at (2, .5) {};
            \node (g) at (2, -.5) {};
            \node (h) at (3, .5) {};

            \draw[thick] (a)--(b);
            \draw[thick] (a)--(c);
            \draw[thick] (b)--(c);
            \draw[thick] (b)--(d);
            \draw[thick] (c)--(e);
            \draw[thick] (f)--(h);
            \draw[thick] (g)--(a)--(f)--(g);

            \node[fill=none, scale=1.75] at (-.5, .7) {$\ell$};
            \node[fill=none, scale=1.75] at (-1.2, .6) {$v$};
            \node[fill=none, scale=1.75] at (1, 0.2) {$u$};
            \node[fill=none, scale=1.75] at (2, 0.7) {$w$};
            \node[fill=none, scale=1.75] at (2.2, 0) {$e$};
        \end{tikzpicture}
        \caption{The edge $\ell$ is a leaf-edge, while $e$ is an internal edge. The vertex $v$ is a leaf vertex, and $u$ and $w$ are examples of internal vertices.}
        \label{fig:enter-label}
    \end{figure}
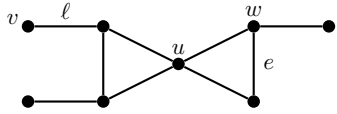

%%%%%%%%%%%%%%%%%%%%%%%%%%%%%%%%%%%%%%%%%%%%%%%%%%%%%%%%%%%%%%
\subsection{Symmetric functions}\label{subsec:sym-fun}
    In this subsection, we review the basics that we need about symmetric functions. For a deeper coverage of the vast field of symmetric functions, see \cite{macdonald1998symmetric, symSagan}.

    A \defn{partition} $\lambda$ is a sequence of positive integers $(\lambda_1, \ldots, \lambda_k)$ such that $\lambda_1 \geq \cdots \geq \lambda_k$. 
    We say that $|\lambda| \coloneqq \sum_i \lambda_i$ is the \defn{weight} of $\lambda$. If $|\lambda|=n$, we say that $\lambda$ is a partition of $n$, and write $\lambda \vdash n$. Each $\lambda_i$ is a \defn{part}, and the number of parts is the \defn{length} of $\lambda$, denoted by $\ell(\lambda)$. There is a natural total ordering for partitions called \defn{lexicographic order}. If $\lambda, \mu \vdash n$, we say that $\lambda \leq \mu$ if $\lambda = \mu$ or if $\lambda_i = \mu_i$ for $1\leq i<j$ and $\lambda_j<\mu_j$ for some $1\leq j \leq \ell(\lambda)$.

    Let $\{x_1, x_2, \ldots\}$ be a countably infinite set of commuting variables. The algebra of symmetric functions is a subalgebra of  $\Q[[x_1, x_2, \ldots]]$ and can be defined as follows. For $r\geq 1$, let $p_r$ denote the \defn{$r$-th power symmetric function} which is defined by 
    \[ p_r \coloneqq x_1^r + x_2^r + x_3^r+ \cdots \]
    and for a partition $\lambda = (\lambda_1, \lambda_2, \ldots, \lambda_\ell)$, the \defn{power symmetric function} is defined by 
    \[ p_\lambda \coloneqq p_{\lambda_1} p_{\lambda_2}\cdots p_{\lambda_\ell}~.\]
    The set $\{p_\lambda \, |\, \lambda\vdash k\}$ is linearly independent. The function $p_\lambda$ is called symmetric because it is invariant under the action of the symmetric group, $S_\infty = \bigsqcup_{n\geq 1} S_n$ (the disjoint union of all symmetric groups), which acts by permuting the indices of the variables.  

   The algebra of symmetric functions is defined as the graded algebra
    \[ \Lambda \coloneqq \Lambda^0 \oplus \Lambda^1 \oplus \Lambda^2 \oplus \cdots ~,\]
    where $\Lambda^0 \coloneqq\mathbb{Q}$, and for $k\geq 1$, we have $\Lambda^k \coloneqq \text{Span}_{\mathbb{Q}} \{ p_\lambda \, |\, \lambda \vdash k\}$.
    
    In this paper we are concerned with the \defn{chromatic symmetric function} introduced by Stanley in \cite{stanley95symmetric}. For any finite graph $G$ with vertex set $\{v_1, \ldots, v_n\}$ define
    \[
    \mathbf{X}_G = \sum_{\kappa} x_{\kappa(v_1)}x_{\kappa(v_2)}\cdots x_{\kappa(v_n)},
    \]
    where the sum runs over all proper colorings of the graph $G$.  $\mathbf{X}_G$ is a homogeneous symmetric function of degree the order of $G$, i.e., $n$. In \cite{stanley95symmetric}, Stanley described the expansion of $\mathbf{X}_G$ in terms of the power sum basis. To describe this expansion, we need some notation.  Given a subset of edges $A \subseteq E$ of the graph $G$, the \defn{spanning subgraph} induced by $A$ is the subgraph $G|_A$ with vertex set $V$ and edge set $A$. If $G$ has $n$ vertices, let $\lambda(A)$ be the partition of $n$ whose parts are the orders of the connected components of $G|_A$.

    \begin{theorem}[\cite{stanley95symmetric}, Theorem 2.5]\label{thm:stanley-power-sum-expansion}
        We have
        \[
        \mathbf{X}_G = \sum_{A \subseteq E} (-1)^{|A|}p_{\lambda(A)}
        \]
    \end{theorem}
    
    In \cite{chromBases}, Cho and van Willigenburg introduced new bases of symmetric functions using the chromatic symmetric function. They proved the following theorem.

    \begin{theorem}[\cite{chromBases}, Theorem 5] \label{thm:cho-vanW}
        For any positive integer $k$, let $G_k$ denote a connected graph with $k$ vertices and let $\{G_k\}_{k\geq 1}$ be a family of such graphs. Given a partition $\lambda \vdash n$ of length $\ell$, define $G_\lambda = G_{\lambda_1} \sqcup \cdots \sqcup G_{\lambda_\ell}$. Then, $\{\mathbf{X}_{G_{\lambda}} : \lambda \vdash n\}$ is a basis for $\Lambda^n$.
    \end{theorem}

The bases constructed in Theorem \ref{thm:cho-vanW} are called \defn{chromatic bases}. In this paper, we are concerned with the chromatic basis constructed from star graphs, where each $G_k\coloneqq St_k$ is a $k$-star. This basis was one of the examples in \cite{chromBases} and it was recently further explored in \cite{ADOZ}.  For any positive integer $k$, $St_k$ denotes the tree with $k-1$ vertices of degree 1 and one vertex of degree $k-1$, and it is called the star graph on $k$ vertices because of its shape.  For example, for $k=5$, we have
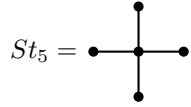
\[
St_5 = 
    \tikzmath{
    \filldraw (0,0) node {} circle (.06cm);
    \filldraw (0,-.6) node {} circle (.06cm);
    \filldraw (0,.6) node {} circle (.06cm);
    \filldraw (-.6,0) node {} circle (.06cm);
    \filldraw (.6,0) node {} circle (.06cm);

    \draw[thick] (0,0)--(0,-.6);
    \draw[thick] (0,0)--(0,.6);
    \draw[thick] (0,0)--(0.6,0);
    \draw[thick] (0,0)--(-0.6,0);
    }
\]

And for a partition $\lambda = (\lambda_1, \ldots, \lambda_\ell)$, define $St_\lambda = St_{\lambda_1} \sqcup \cdots \sqcup St_{\lambda_\ell}$ as the disjoint union of $\lambda_i$-stars. For any positive integer $k$, \[\stfrak_k \coloneqq \mathbf{X}_{St_k} \text{ and }  \stfrak_\lambda \coloneqq \stfrak_{\lambda_1}\cdots\stfrak_{\lambda_\ell} \text{ for }\lambda\vdash n~.\] 
By Theorem \ref{thm:cho-vanW}, the set $\{\stfrak_\lambda \, |\,\lambda \vdash n\}$ is a basis for $\Lambda^n$, which we call the \defn{star-basis}. There is a nice change of basis from the power sum basis to the star-basis, which we state below. Observe that \eqref{eq:st-to-power-sum} follows from Theorem \ref{thm:stanley-power-sum-expansion} and \eqref{eq:power-sum-to-star} follows from \eqref{eq:st-to-power-sum} using simple properties of binomial coefficients.

\begin{proposition} 
    We have
\begin{equation}\label{eq:st-to-power-sum}
        \stfrak_{n+1} = \sum_{r=0}^n(-1)^r \binom{n}{r} p_{(r+1,1^{n-r})}
    \end{equation}
    and 
    \begin{equation}\label{eq:power-sum-to-star}
        p_{n+1}=\sum_{r=0}^n(-1)^r \binom{n}{r} \stfrak_{(r+1,1^{n-r})}
    \end{equation}
\end{proposition}
By combining the proposition above (in particular \eqref{eq:power-sum-to-star}) with Theorem \ref{thm:stanley-power-sum-expansion}, we obtain the expansion of the chromatic symmetric function in terms of the star-basis. In \cite{ADOZ}, the authors gave an efficient algorithm for computing the star-basis expansion of $\mathbf{X}_G$ for any graph. 

%%%%%%%%%%%%%%%%%%%%%%%%%%%%%%%%%%%%%%%%%%%%%%%%%%%%%%%%%%%%%%
\section{The Deletion-Near-Contraction relation}
%%%%%%%%%%%%%%%%%%%%%%%%%%%%%%%%%%%%%%%%%%%%%%%%%%%%%%%%%%%%%%
In this section we review the \defn{deletion-near-contraction} (DNC) relation, introduced in \cite{ADOZ}. This relation is a modification of the classical deletion-contraction relation used to compute the chromatic polynomial of any graph.  The DNC relation leads to an algorithm for computing the chromatic symmetric function in the star-basis that avoids cancellation of terms. The DNC relation uses three operations on an edge: deletion, leaf-contraction, and dot-contraction.
\begin{itemize}
    \item {\bf Deletion:} this is the classical deletion of an edge in a graph. Given a graph $G$, we denote the resulting graph obtained by deleting an edge $e$ by $G \setminus e$.
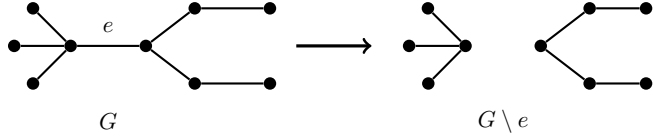
    \begin{center}
        \begin{tikzpicture}[auto=center,every node/.style={circle, fill=black, scale=0.5}]
            \node (a) at (0,0) {};
            \node (a1) at (-.5,.5) {};
            \node (a2) at (-.75, 0) {};
            \node (a3) at (-.5, -.5) {};

            \draw[thick] (a) -- (a1);
            \draw[thick] (a) -- (a2);
            \draw[thick] (a) -- (a3);

            \node (b) at (1,0) {};
            \node (b1) at (1.65, .5) {};
            \node (b2) at (1.65, -.5){};

            \node (b11) at (2.65, .5) {};
            \node (b21) at (2.65, -.5) {};

            \draw[thick] (b) -- (b1);
            \draw[thick] (b) -- (b2);

            \draw[thick] (b1) -- (b11);
            \draw[thick] (b2) -- (b21);

            \draw[thick] (a) -- node[above, fill=none, scale=1.75] {$e$} ++ (b); 

            \node[fill=none, scale=1.75] at (0.5, -1) {$G$};

            \draw[very thick, ->] (3,0) -- (4,0);

            \node (a) at (5.25,0) {};
            \node (a1) at (4.75,.5) {};
            \node (a2) at (4.5, 0) {};
            \node (a3) at (4.755, -.5) {};

            \draw[thick] (a) -- (a1);
            \draw[thick] (a) -- (a2);
            \draw[thick] (a) -- (a3);

            \node (b) at (6.25,0) {};
            \node (b1) at (6.9, .5) {};
            \node (b2) at (6.9, -.5){};

            \node (b11) at (7.65, .5) {};
            \node (b21) at (7.65, -.5) {};

            \draw[thick] (b) -- (b1);
            \draw[thick] (b) -- (b2);

            \draw[thick] (b1) -- (b11);
            \draw[thick] (b2) -- (b21);

            \node[fill=none, scale=1.75] at (5.75, -1) {$G\setminus e$};
        \end{tikzpicture}
    \end{center}
 \item \textbf{Leaf-contraction:} Given a graph $G$ and an edge $e$ in $G$, the \defn{leaf-contracted graph}, $G\odot e$, is obtained by contracting $e$ and attaching a leaf, $\ell_e$, to the vertex that results from the contraction of $e$. 
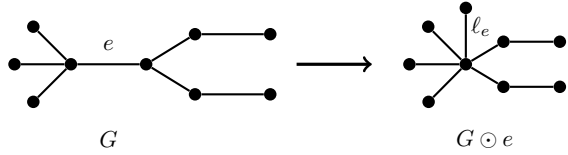
    \begin{center}
        \begin{tikzpicture}[auto=center,every node/.style={circle, fill=black, scale=0.5}]
            \node (a) at (0,0) {};
            \node (a1) at (-.5,.5) {};
            \node (a2) at (-.75, 0) {};
            \node (a3) at (-.5, -.5) {};

            \draw[thick] (a) -- (a1);
            \draw[thick] (a) -- (a2);
            \draw[thick] (a) -- (a3);

            \node (b) at (1,0) {};
            \node (b1) at (1.65, .4) {};
            \node (b2) at (1.65, -.4){};

            \node (b11) at (2.65, .4) {};
            \node (b21) at (2.65, -.4) {};

            \draw[thick] (b) -- (b1);
            \draw[thick] (b) -- (b2);

            \draw[thick] (b1) -- (b11);
            \draw[thick] (b2) -- (b21);

            \draw[thick] (a) -- node[above, fill=none, scale=1.75] {$e$} ++ (b); 

            \node[fill=none, scale=1.75] at (0.5, -1) {$G$};

            \draw[very thick, ->] (3,0) -- (4,0);

            \node (a) at (5.25,0) {};
            \node (a1) at (4.75,.5) {};
            \node (a2) at (4.5, 0) {};
            \node (a3) at (4.755, -.5) {};

            \draw[thick] (a) -- (a1);
            \draw[thick] (a) -- (a2);
            \draw[thick] (a) -- (a3);

            \node (b1) at (5.75, .3) {};
            \node (b2) at (5.75, -.3){};

            \node (l) at (5.25, 0.75) {};

            \draw[thick] (a)--(l);
            
            \node (b11) at (6.5, .3) {};
            \node (b21) at (6.5, -.3) {};

            \draw[thick] (a) -- (b1);
            \draw[thick] (a) -- (b2);

            \draw[thick] (b1) -- (b11);
            \draw[thick] (b2) -- (b21);
            \node[fill=none, scale=1.75] at (5.45, .5) {$\ell_e$};
            \node[fill=none, scale=1.75] at (5.5, -1) {$G\odot e$};
        \end{tikzpicture}
    \end{center}
 
    \item \textbf{Dot-contraction:} Given an edge $e$, the \defn{dot-contracted graph}, $(G\odot e)\setminus \ell_e$, is obtained by contracting the edge $e$ and adding an isolated vertex, $v$, to the graph. This can be formulated in terms of the leaf-contraction operation as simply removing the edge $\ell_e$, hence the notation for the dot-contraction. 
    
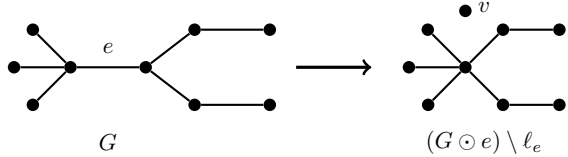
\begin{center}
        \begin{tikzpicture}[auto=center,every node/.style={circle, fill=black, scale=0.5}]
            \node (a) at (0,0) {};
            \node (a1) at (-.5,.5) {};
            \node (a2) at (-.75, 0) {};
            \node (a3) at (-.5, -.5) {};

            \draw[thick] (a) -- (a1);
            \draw[thick] (a) -- (a2);
            \draw[thick] (a) -- (a3);

            \node (b) at (1,0) {};
            \node (b1) at (1.65, .5) {};
            \node (b2) at (1.65, -.5){};

            \node (b11) at (2.65, .5) {};
            \node (b21) at (2.65, -.5) {};

            \draw[thick] (b) -- (b1);
            \draw[thick] (b) -- (b2);

            \draw[thick] (b1) -- (b11);
            \draw[thick] (b2) -- (b21);

            \draw[thick] (a) -- node[above, fill=none, scale=1.75] {$e$} ++ (b); 

            \node[fill=none, scale=1.75] at (0.5, -1) {$G$};

            \draw[very thick, ->] (3,0) -- (4,0);

            \node (a) at (5.25,0) {};
            \node (a1) at (4.75,.5) {};
            \node (a2) at (4.5, 0) {};
            \node (a3) at (4.755, -.5) {};

            \draw[thick] (a) -- (a1);
            \draw[thick] (a) -- (a2);
            \draw[thick] (a) -- (a3);

            \node (b1) at (5.75, .5) {};
            \node (b2) at (5.75, -.5){};

            \node at (5.25, 0.75) {};
            
            \node (b11) at (6.5, .5) {};
            \node (b21) at (6.5, -.5) {};

            \draw[thick] (a) -- (b1);
            \draw[thick] (a) -- (b2);

            \draw[thick] (b1) -- (b11);
            \draw[thick] (b2) -- (b21);
            \node[fill=none, scale=1.75] at (5.5, .8) {$v$};
            \node[fill=none, scale=1.75] at (5.5, -1) {$(G\odot e)\setminus \ell_e$};
        \end{tikzpicture}
        \vspace{-.5cm}
    \end{center}
\end{itemize}

In \cite{ADOZ}, the authors proved that the chromatic symmetric function satisfies the following relation, which was first introduced in \cite{chmutov} for chord diagrams and reduces to the well-known deletion-contraction relation of $\chi_G(k)$ when we make the appropriate substitution to $\mathbf{X}_G$ so that it becomes the chromatic polynomial.
\begin{proposition}[\cite{ADOZ}](The deletion-near-contraction relation or DNC relation) \label{prop:dnc-formula}
    For a simple graph $G$ and any edge $e$ in $G$, we have
    \[
    \mathbf{X}_G = \mathbf{X}_{G \setminus e} - \mathbf{X}_{(G \odot e)\setminus \ell_e} + \mathbf{X}_{G \odot e}
    \]
\end{proposition}
In \cite{ADOZ}, the authors used the DNC relation to give a recursive algorithm for computing $\mathbf{X}_G$  in the star-basis. 
\begin{remark}
\begin{enumerate}\label{DNC-properties}
\item  If $e$ is a leaf-edge in $G$, i.e., one of its endpoints has degree 1,  then $G\setminus e \cong (G \odot e)\setminus \ell_e$ and $G \cong G \odot e$. Therefore, applying the DNC relation to a leaf-edge $e$ does not simplify the computation of $\mathbf{X}_G$.

\item If $e$ is an internal edge in $G$, i.e., both endpoints have degree greater than 1, then $G\setminus e$, $G\odot e$ and $(G\odot e)\setminus \ell_e$ are graphs with fewer internal edges than $G$.  

\item The only connected, simple graph without internal edges is the star graph. 
\end{enumerate}
\end{remark}
As a consequence of Remark \ref{DNC-properties}, we can recursively apply the DNC relation on internal edges until $\mathbf{X}_G$ can be written as a linear combination of $\mathbf{X}_H$, where $H$ is a forest of star graphs. 
This process is formalized in the Star-Expansion Algorithm presented in \cite{ADOZ} which we include below. 

\begin{algorithm}
\noindent \textbf{Input:} A simple graph $G$ and an ordering of the internal edges.

\noindent \textbf{Initialization:} Let $\mathcal{T}$ be a rooted tree with root labeled by $G$ and no edges.

\noindent \textbf{Iteration:} If $H$ is a leaf of $\mathcal{T}$ labeled by a graph $H$ and $H$ has an internal edge $e$, then add three children to $H$ labeled by the graphs $H \setminus e$, $(H \odot e) \setminus \ell_e$, and $H \odot e$, and label these edges with $+$ or $-$ according to the coefficient of these graphs in the DNC relation. The algorithm terminates when all leaves in $\mathcal{T}$ have no internal edges.

\noindent \textbf{Output:} A rooted tree $\mathcal{T}(G)$ where the leaves are labeled by star forests.
    \caption{The star-expansion algorithm.}
    \label{DNC-algorithm}
\end{algorithm}
We call the output of the star-expansion algorithm a \defn{DNC-tree}. As shown by Aliste-Prieto, de Mier, Orellana and Zamora in \cite{ADOZ}, the chromatic symmetric function of a graph $G$ can be computed directly from a DNC-tree whose root is labeled by $G$.
\begin{theorem}[\cite{ADOZ}] \label{thm:dnc-tree-csf}
   For any simple graph $G$, let $\mathcal{T}(G)$ be a DNC-tree obtained from the star-expansion algorithm and let $L(\mathcal{T}(G))$ be the multiset of leaf labels of $\mathcal{T}(G)$. Then
   \[
   \mathbf{X}_G = \sum_{H \in L(\mathcal{T}(G))} (-1)^{\iota(H)-\iota(G)}\stfrak_{\lambda(H)}
   \]
   where $\iota(H)$ and $\iota(G)$ denote the number of isolated vertices in $H$ and $G$, respectively, and $\lambda(H)$ is the partition whose parts are the orders of the connected components of $H$. In addition, no cancellations occur in the computation; that is, for any partition $\lambda$, all terms $\mathfrak{st}_\lambda$ appear with the same sign.
\end{theorem}

\begin{example}\label{exa:DNC-tree}
    Figure \ref{fig:dnc-tree} shows an example of how to apply the star-expansion algorithm. In particular, it says that for the graph $T$ at the root, we have \[\csft = 
    -\mathfrak{st}_{(4,2,1)}+ \mathfrak{st}_{(4,3)}+\mathfrak{st}_{(5,1,1)}+\mathfrak{st}_{(5,2)}-2\mathfrak{st}_{(6,1)}+\mathfrak{st}_{(7)}~.\]
In Figure \ref{fig:dnc-tree} we use red to indicate the internal edge on which we are applying the DNC relation. 
\begin{figure}[hbt!]\label{Ex:DNC-tree}
    \centering
    \begin{tikzpicture}[auto=center,every node/.style={circle, fill=black, scale=0.45}, style=thick, scale=0.4] \label{tikz:dnc-tree}
    \filldraw[black] (0, 0) coordinate (A1) circle (4pt) node{};
    \filldraw[black] (1, 0) coordinate (A2) circle (4pt) node{};
    \filldraw[black] (1, 1) coordinate (A3) circle (4pt) node{};
    \filldraw[black] (1, -1) coordinate (A4) circle (4pt) node{};
    \node (A5) at (2,0) {};
    \node (A6) at (3,0) {};
    \filldraw[black] (4, 0) coordinate (A7) circle (4pt) node{};
    
    %%%%% Draw edges %%%%%
    \draw(A1) -- (A2);
    \draw(A2) -- (A3);
    \draw(A2) -- (A4);
    \draw(A2) -- (A5);
    \draw[color=red](A5) -- (A6);
    \draw(A6) -- (A7);
    \node[fill=none, scale=1.75] at (14.5, -4.7) {$e$};

        %%%%% TREE 2 %%%%%
    %%%%% Draw vertices %%%%%
    \filldraw[black] (-13, -5) coordinate (A1) circle (4pt) node{};
    \filldraw[black] (-12, -5) coordinate (A2) circle (4pt) node{};
    \filldraw[black] (-12, -4) coordinate (A3) circle (4pt) node{};
    \filldraw[black] (-12, -6) coordinate (A4) circle (4pt) node{};
    \filldraw[black] (-11, -5) coordinate (A5) circle (4pt) node{};
    \filldraw[black] (-10, -5) coordinate (A6) circle (4pt) node{};
    \filldraw[black] (-9, -5) coordinate (A7) circle (4pt) node{};
    
    %%%%% Draw edges %%%%%
    \draw(A1) -- (A2);
    \draw(A2) -- (A3);
    \draw(A2) -- (A4);
    \draw(A2) -- (A5);
    \draw(A6) -- (A7);
    
        %%%%% TREE 3 %%%%%
    %%%%% Draw vertices %%%%%
    \filldraw[black] (0.5, -5) coordinate (A1) circle (4pt) node{};
    \node (A2) at (1.5,-5) {};
    \filldraw[black] (1.5, -4) coordinate (A3) circle (4pt) node{};
    \filldraw[black] (1.5, -6) coordinate (A4) circle (4pt) node{};
    \node (A5) at (2.5,-5) {};
    \filldraw[black] (3.5, -5) coordinate (A6) circle (4pt) node{};
    \filldraw[black] (2.5, -4) coordinate (A7) circle (4pt) node{};
    
    %%%%% Draw edges %%%%%
    \draw(A1) -- (A2);
    \draw(A2) -- (A3);
    \draw(A2) -- (A4);
    \draw[red](A2) -- (A5);
    \draw(A5) -- (A6);
    \node[fill=none, scale=1.75] at (2.1, -4.7) {$e$};

        %%%%% TREE 4 %%%%%
    %%%%% Draw vertices %%%%%
    \filldraw[black] (13, -5) coordinate (A1) circle (4pt) node{};
    \node (A2) at (14,-5) {};
    \filldraw[black] (14, -4) coordinate (A3) circle (4pt) node{};
    \filldraw[black] (14, -6) coordinate (A4) circle (4pt) node{};
    \node (A5) at (15,-5) {};
    \filldraw[black] (16, -5) coordinate (A6) circle (4pt) node{};
    \filldraw[black] (15, -4) coordinate (A7) circle (4pt) node{};

    %%%%% Draw edges %%%%%
    \draw(A1) -- (A2);
    \draw(A2) -- (A3);
    \draw(A2) -- (A4);
    \draw[red](A2) -- (A5);
    \draw(A5) -- (A7);
    \draw(A5) -- (A6);
    \node[fill=none, scale=1.75] at (2.5, 0.3) {$e$};

        %%%%% TREE 5 %%%%%
    %%%%% Draw vertices %%%%%
    \filldraw[black] (-4.5, -10) coordinate (A1) circle (4pt) node{};
    \filldraw[black] (-3.5, -10) coordinate (A2) circle (4pt) node{};
    \filldraw[black] (-3.5, -9) coordinate (A3) circle (4pt) node{};
    \filldraw[black] (-3.5, -11) coordinate (A4) circle (4pt) node{};
    \filldraw[black] (-2.5, -10) coordinate (A5) circle (4pt) node{};
    \filldraw[black] (-1.5, -10) coordinate (A6) circle (4pt) node{};
    \filldraw[black] (-2.5, -9) coordinate (A7) circle (4pt) node{};
    
    %%%%% Draw edges %%%%%
    \draw(A1) -- (A2);
    \draw(A2) -- (A3);
    \draw(A2) -- (A4);
    \draw(A5) -- (A6);
    
        %%%%% TREE 6 %%%%%
    %%%%% Draw vertices %%%%%
    \filldraw[black] (0.5, -10) coordinate (A1) circle (4pt) node{};
    \filldraw[black] (1.5, -10) coordinate (A2) circle (4pt) node{};
    \filldraw[black] (1.5, -9) coordinate (A3) circle (4pt) node{};
    \filldraw[black] (1.5, -11) coordinate (A4) circle (4pt) node{};
    \filldraw[black] (2.5, -10) coordinate (A5) circle (4pt) node{};
    \filldraw[black] (2.5, -11) coordinate (A6) circle (4pt) node{};
    \filldraw[black] (2.5, -9) coordinate (A7) circle (4pt) node{};
    
    %%%%% Draw edges %%%%%
    \draw(A1) -- (A2);
    \draw(A2) -- (A3);
    \draw(A2) -- (A4);
    \draw(A2) -- (A5);
    
        %%%%% TREE 7 %%%%%
    %%%%% Draw vertices %%%%%
    \filldraw[black] (4.5, -10) coordinate (A1) circle (4pt) node{};
    \filldraw[black] (5.5, -10) coordinate (A2) circle (4pt) node{};
    \filldraw[black] (5.5, -9) coordinate (A3) circle (4pt) node{};
    \filldraw[black] (5.5, -11) coordinate (A4) circle (4pt) node{};
    \filldraw[black] (6.5, -10) coordinate (A5) circle (4pt) node{};
    \filldraw[black] (6.5, -11) coordinate (A6) circle (4pt) node{};
    \filldraw[black] (6.5, -9) coordinate (A7) circle (4pt) node{};
    
    %%%%% Draw edges %%%%%
    \draw(A1) -- (A2);
    \draw(A2) -- (A3);
    \draw(A2) -- (A4);
    \draw(A2) -- (A5);
    \draw(A2) -- (A6);
    
        %%%%% TREE 8 %%%%%
    %%%%% Draw vertices %%%%%
    \filldraw[black] (8.5, -10) coordinate (A1) circle (4pt) node{};
    \filldraw[black] (9.5, -10) coordinate (A2) circle (4pt) node{};
    \filldraw[black] (9.5, -9) coordinate (A3) circle (4pt) node{};
    \filldraw[black] (9.5, -11) coordinate (A4) circle (4pt) node{};
    \filldraw[black] (10.5, -10) coordinate (A5) circle (4pt) node{};
    \filldraw[black] (11.5, -10) coordinate (A6) circle (4pt) node{};
    \filldraw[black] (10.5, -9) coordinate (A7) circle (4pt) node{};
    
    %%%%% Draw edges %%%%%
    \draw(A1) -- (A2);
    \draw(A2) -- (A3);
    \draw(A2) -- (A4);
    \draw(A5) -- (A7);
    \draw(A5) -- (A6);
    
        %%%%% TREE 9 %%%%%
    %%%%% Draw vertices %%%%%
    \filldraw[black] (13.5, -10) coordinate (A1) circle (4pt) node{};
    \filldraw[black] (14.5, -10) coordinate (A2) circle (4pt) node{};
    \filldraw[black] (14.5, -9) coordinate (A3) circle (4pt) node{};
    \filldraw[black] (14.5, -11) coordinate (A4) circle (4pt) node{};
    \filldraw[black] (15.5, -10) coordinate (A5) circle (4pt) node{};
    \filldraw[black] (15.5, -11) coordinate (A6) circle (4pt) node{};
    \filldraw[black] (15.5, -9) coordinate (A7) circle (4pt) node{};
    
    %%%%% Draw edges %%%%%
    \draw(A1) -- (A2);
    \draw(A2) -- (A3);
    \draw(A2) -- (A4);
    \draw(A2) -- (A6);
    \draw(A2) -- (A7);
    
        %%%%% TREE 10 %%%%%
    %%%%% Draw vertices %%%%%
    \filldraw[black] (17.5, -10) coordinate (A1) circle (4pt) node{};
    \filldraw[black] (18.5, -10) coordinate (A2) circle (4pt) node{};
    \filldraw[black] (18.5, -9) coordinate (A3) circle (4pt) node{};
    \filldraw[black] (18.5, -11) coordinate (A4) circle (4pt) node{};
    \filldraw[black] (19.5, -10) coordinate (A5) circle (4pt) node{};
    \filldraw[black] (19.5, -11) coordinate (A6) circle (4pt) node{};
    \filldraw[black] (19.5, -9) coordinate (A7) circle (4pt) node{};
    
    %%%%% Draw edges %%%%%
    \draw(A1) -- (A2);
    \draw(A2) -- (A3);
    \draw(A2) -- (A4);
    \draw(A2) -- (A5);
    \draw(A2) -- (A6);
    \draw(A2) -- (A7);

    %%%%% DNC-TREE EDGES %%%%%
    \draw(2,-1.5) -- (-9,-3.5);
    
    \node[fill=none] at (-3.5, -2.15) {\huge $+$};
    \draw (2,-1.5) -- (2,-3.5);
    \node[fill=none] at (1.55, -2.35) {\huge $-$};
    \draw (2,-1.5) -- (13,-3.5);
    \node[fill=none] at (7.25, -2.15) {\huge $+$};
    
    \draw [fill=none](2,-6.5) -- (-2,-8.5);

    \node[fill=none] at (-0.5, -7.25) {\huge $+$};
    
    \draw(2,-6.5) -- (2,-8.5);
    
    \node[fill=none] at (1.55, -7.45) {\huge $-$};
    
    \draw (2,-6.5) -- (6,-8.5);

    \node[fill=none] at (4.5, -7.25) {\huge $+$};

    \draw (14.5,-6.5) -- (10.5,-8.5);
    
    \node[fill=none] at (12, -7.25) {\huge $+$};
    
    \draw (14.5,-6.5) -- (14.5,-8.5);

    \node[fill=none] at (13.95, -7.5) {\huge $-$};
    
    \draw (14.5,-6.5) -- (18.5,-8.5);

    \node[fill=none] at (17, -7.25) {\huge $+$};
    \end{tikzpicture}
    \caption{A DNC-tree $\mathcal{T}(G)$ where we apply the relation on the edge labeled $e$.}
    \label{fig:dnc-tree}
\end{figure}

\end{example}

In this paper we are interested in the coefficients that occur when $\mathbf{X}_G$ is written in the star-basis for any $n$-vertex graph $G$.  That is,
\[
\mathbf{X}_G = \sum_{\lambda \vdash n} c_\lambda \mathfrak{st}_\lambda ~. 
\]
By Theorem \ref{thm:dnc-tree-csf} each path in $\mathcal{T}(G)$ from the root $G$ to a star forest $F$ produces a summand in $\mathbf{X}_G$. Since there are no cancellations, i.e., a term obtained from one such path from $G$ to $F_1$ does not cancel another from $G$ to $F_2$, the coefficient of $\mathfrak{st}_\lambda$ can be computed by counting paths from $G$ to leaf-vertices $F$ in $\mathcal{T}(G)$ such that $\lambda(F) = \lambda$ and counting the number of times we dot-contract in the path since $\iota(F)-\iota(G)$ is precisely the number of new isolated vertices created along the path.

In addition, the paths in $\mathcal{T}(G)$ can be encoded as sequences of operations. If we let $L$ represent a deletion, $M$ a dot-contraction, and $R$ a leaf-contraction, then a path from the root $G$ to a star forest $F$ is a sequence of length at most $\#I(T)$ of $L$s, $M$s, and $R$s, see Example \ref{exa:DNC-tree-seq}.  Note that $M$ is the only operation that results in isolated vertices as components and the only operation that changes the sign of the coefficient.  We summarize these observations in the following corollary. 

\begin{corollary}\label{cor:coeff-paths-seq} Let $G$ be an $n$-vertex graph and $\mathcal{T}(G)$ a DNC-tree corresponding to $G$. If $\mathbf{X}_G=\sum_{\lambda\vdash n} c_\lambda \mathfrak{st}_\lambda$, then 
\[ c_\lambda = (-1)^{m} |\mathcal{S}_\lambda|,\]
where $\mathcal{S}_\lambda$ is the set of sequences from the root $G$ to a forest $F$ such that $\lambda(F)= \lambda$ and $m$ is the number of $M$s in each sequence.
\end{corollary}

\begin{example}\label{exa:DNC-tree-seq}
For the DNC-tree in Figure \ref{fig:dnc-tree} we have the following correspondences to sequences. 
\begin{center}
\begin{tabular}{||c | c | r||} 
 \hline
 $\lambda \vdash n$ & $\mathcal{S}_\lambda$ & $c_{\lambda}$ \\ [0.5ex] 
 \hline\hline
 (7) & \{$(R,R)$\} & $1$ \\ 
 \hline
 (6,1) & \{$(R,M), (M,R)$\} & $-2$ \\
 \hline
 (5,2) & \{$(L)$\} & $1$ \\
 \hline
 $(5,1^2)$ & \{$(M,M)$\} & $1$ \\
 \hline
 (4,3) & \{$(R,L)$\} & $1$ \\ 
 \hline
 (4,2,1) & \{$(M,L)$\} & $-1$ \\
 \hline
\end{tabular}
\end{center}
\end{example}

Our main interest is to study $\mathbf{X}_G$ when $G = T$ is a tree; therefore, in what follows we will restrict ourselves to the case when $G = F$ is a forest.  

In the proof of Lemma \ref{lem:deletions-lead}, we will see that deleting an internal edge $e$ from a forest $G$ can result in a forest $G\setminus e$ with one to three fewer internal edges than $T$. On the other hand, both leaf-contraction and dot-contraction result in a forest with only one fewer internal edge as we show in the next lemma. 

\begin{lemma}\label{lemma:internal-edge-contraction}
    If $G$ is a forest and $e$ an internal edge, then $G\odot e$ and $(G\odot e)\setminus \ell_e$ have exactly one fewer internal edge than $G$.
\end{lemma}

\begin{proof}
    Notice that in both dot-contraction and leaf-contraction operations, we contract an internal edge to a vertex, thus reducing the number of edges by one in the case of a dot-contraction and leaving the number of edges the same in the case of a leaf-contraction in which an internal edge is contracted, but we add a leaf-edge.
    
    Let $e'\neq e$ be another edge in $G$. Then if $e'$ is an internal edge and $e$ and $e'$ are not incident, then the degrees of the endpoints of $e'$ are unchanged by the contraction of $e$ (contraction is a local operation). So, $e'$ remains an internal edge. If $e'$ is an internal edge and $e$ and $e'$ are incident, then the degree of the endpoint at which $e$ and $e'$ meet does not decrease under either a dot-contraction or a leaf-contraction. Hence, $e'$ is still an internal edge in $G\odot e$ and also in  $(G\odot e)\setminus \ell_e$.  If $e'$ is a leaf, then it remains a leaf in $G\odot e$ and also in  $(G\odot e)\setminus \ell_e$ as the endpoint of degree 1 remains an endpoint of degree 1 after a leaf- or dot-contraction.  Hence the result follows. 
\end{proof}

\begin{proposition}\label{prop:hook-coeff}
   Let $T$ be an $n$-vertex tree and $I(T)$ the set of internal edges of $T$. If 
     $\csft = \sum_{\lambda\vdash n} c_\lambda \mathfrak{st}_\lambda$, then 
    \[
    c_{(n-m,1^m)} = (-1)^m \binom{\#I(T)}{m}~.
    \]
    In particular, $|c_{(n-1,1)}| = \#I(T)$.
\end{proposition}
\begin{proof}
    Notice that $|c_{(n-m,1^m)}|$ is the number of paths in $\mathcal{T}(T)$ from the root $T$ to leaves labeled $ H = St_{n-m}\cup \underbrace{St_1\cup \cdots \cup St_1}_\text{$m$}$. $H$ can only be obtained  through sequences with $m$ $M$s (dot-contractions) and $\#I(T) - m$ $R$s (leaf-contractions). 

    By Corollary \ref{cor:coeff-paths-seq}, the coefficient is $c_{(n-m,1^m)} = (-1)^m|\mathcal{S}_{(n-m,1^m)}|$. Hence, we simply have to count the number of sequences in $\mathcal{S}_{(n-m,1^m)}$.  By Lemma \ref{lemma:internal-edge-contraction}, these sequence have length $\#I(T)$, so it suffices to choose which terms are $M$s.
    There are  $\binom{\#I(T)}{m}$ such sequences. 
\end{proof}

We remark that in the chromatic symmetric function of any $n$-vertex tree $T$ where $n \geq 2$ or $n \geq 3$, respectively, we have
\[c_{(1^n)} =0 \text{ and } c_{(2,1^{n-2})}=0 ~,\]
this is because we only apply the operations to internal edges and every tree has at least two leaves. In addition, for any $n$-vertex tree $T$, we have $c_{(n)} = 1$, this corresponds to the leaf in $\mathcal{T}(T)$ obtained by applying only leaf-contractions, i.e., a sequence of $\#I(T)$ $R$s.

%%%%%%%%%%%%%%%%%%%%% SECTION 4 %%%%%%%%%%%%%%%%%%%%%%%%%%%%%
\section{The leading partition in the star-basis} \label{sec:leading}
%%%%%%%%%%%%%%%%%%%%%%%%%%%%%%%%%%%%%%%%%%%%%%%%%%%%%%%%%%%%%%
In the remainder of the paper, we assume that $\{\mathfrak{st}_\lambda \mid  \lambda \vdash k\}$ is an ordered basis, ordered using the lexicographic order on partitions.  In this section, we restrict ourselves to the study of the chromatic symmetric function of a tree, $T$, or a forest, $F$. For any $n$-vertex tree, $T$, we write 
\[\mathbf{X}_T = \sum_{\lambda\vdash n} c_\lambda \mathfrak{st}_\lambda, \]
where the summands are listed in increasing lexicographic order. 
For example,  if $T= P_5$, the path with 5 vertices, then 
\[ \mathbf{X}_{P_5} = -\mathfrak{st}_{(2^2,1)}+ \mathfrak{st}_{(3,1,1)}+2\mathfrak{st}_{(3,2)}-2\mathfrak{st}_{(4,1)}+\mathfrak{st}_{(5)}.\]
At the end of the last section we observed that in the case of an $n$-vertex tree $T$, the largest indexing partition, in lexicographic order, with nonzero coefficient is $(n)$, in fact $c_{(n)} =1$. In this section, we are interested in the smallest indexing partition $\lambda$ with nonzero coefficient in the expansion of $\csft$ in the star-basis, as well as the value of this coefficient.  

\begin{definition}
    Let $F$ be an $n$-vertex forest with chromatic symmetric function $\mathbf{X}_F=\sum_{\lambda \vdash n}c_\lambda \mathfrak{st}_\lambda$. The \defn{leading partition} of $\mathbf{X}_F$ is the smallest partition $\lambda \vdash n$, in lexicographic order, such that $c_{\lambda} \neq 0$. We then say that $c_{\lambda}$ is the \defn{leading coefficient} and that $c_{\lambda} \mathfrak{st}_{\lambda}$ is the \defn{leading term}.  We denote the leading partition of $F$ by $\lead(\csff)$.
\end{definition}

For example, in $\mathbf{X}_{P_5}$, the leading partition is $(2^2,1)$, the leading coefficient is $c_{(2^2,1)} = -1$ and the leading term is $-\mathfrak{st}_{(2^2,1)}$. 

Recall that $I(F)$ is the set of internal edges of $F$. We refer to the connected components of $F \setminus I(F)$ as the \defn{leaf components} of $F$, and we denote by $\lc(F)$ the partition whose parts are the orders of the leaf components of $F \setminus I(F)$:
\[
\lc(F)\coloneqq\lambda(F\setminus I(F)).
\]
We will call this partition the \defn{leaf component partition} of $F$.
Notice that $F \setminus I(F)$ is a spanning subgraph, hence every vertex in $F$ is also a vertex in $F \setminus I(F)$. Furthermore, $F\setminus I(F)$ is a forest whose connected components are all stars. Therefore, a leaf component is always a star tree. 

\begin{example}
    Consider the following tree $T$ with two internal edges $e_1, e_2$. It has leaf components $St_4$, $St_2$, and  $St_1$ and $\lc(T) = (4,2,1)$
    
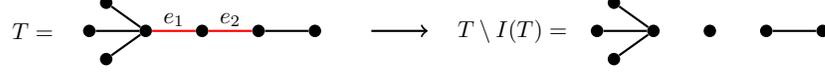
\begin{center}
        \begin{tikzpicture}[auto=center,every node/.style={circle, fill=black, scale=0.5}, scale=0.75]
            \node (a) at (0, 0) {};
            \node (b) at (1, 0) {};
            \node (c) at (.3, .5) {};
            \node (d) at (.3, -.5) {};
            \node (e) at (2, 0) {};
            \node (f) at (3, 0) {};
            \node (g) at (4, 0) {};
                
            \draw[thick] (a) -- (b);
            \draw[thick] (b) -- (c);
            \draw[thick] (b) -- (d);
            \draw[thick][color=red, thick] (b) -- (e);
            \draw[thick][color=red, thick] (e) -- (f);
            \draw[thick] (f) -- (g);
                
            \node (h) at (9, 0) {};
            \node (i) at (10, 0) {};
            \node (j) at (9.3, .5) {};
            \node (k) at (9.3, -.5) {};
            \node (l) at (11, 0) {};
            \node (m) at (12, 0) {};
            \node (n) at (13, 0) {};
                
            \draw[thick] (h) -- (i);
            \draw[thick] (i) -- (j);
            \draw[thick] (i) -- (k);
            \draw[thick] (m) -- (n);
                
            \draw[thick, ->] (5,0)--(6,0);
                
            \node[fill=none, scale=1.75] (p) at (-1,0) {$T = $};
            \node[fill=none, scale=1.75] (q) at (7.5,0) {$T\setminus I(T) =$};
            \node[fill=none, scale=1.75] at (1.5, 0.2) {$e_1$};
            \node[fill=none, scale=1.75] at (2.5, 0.2) {$e_2$};
        \end{tikzpicture}
    \end{center}
\end{example}

Notice that the path $P_5$ has two internal edges, hence $P_5\setminus I(P_5)$ is equal to $St_2 \sqcup St_2 \sqcup St_1$ and $\lc(P_5) =(2^2,1)$ which is the same as the leading partition in $\mathbf{X}_{P_5}$. Note, however, that this is not so obvious. It is not even clear that the partition $\lc(F)$ can be reached in a DNC-tree since the path of all $\#I(F)$ deletions is not always possible. In Example \ref{exa:DNC-tree}, the sequence that achieves the leading partition is $(M,L)$, see the table in Example \ref{exa:DNC-tree-seq}. Notice that $(L, L)$ does not occur as a sequence in this DNC-tree. In fact, no DNC-tree for the tree in this example will contain the sequence $(L,L)$.

The main objective of this section is to prove that for any forest $F$, $\lc(F)$ is equal to the leading partition of $\mathbf{X}_F$. 

\begin{proposition} \label{properties-LCpartition}
Let $T$ be a tree on $n$ vertices with $n\geq3$.
\begin{enumerate}
    \item[(a)] There is a bijection between leaf components of $T$ and internal vertices in $T$. 
    \item[(b)] $\ell(\lc(T)) = \#I(T)+1$.
    \item[(c)] If $T$ has at least one internal edge, then $\lc(T)$ has at least two parts greater than 1.
\end{enumerate}
\end{proposition} 
\begin{proof}
    A leaf component $\mathcal{L}$ in $T$ is a star $St_k$ for some $k\geq 1$ where the center is an internal vertex in $T$ and all other vertices in $\mathcal{L}$  are leaves in $T$, as leaf-edges are the only edges not in $I(T)$. This proves (a). For (b), observe that the number of internal vertices is $\#I(T)+1$, hence the claim follows from (a). For (c), recall that any tree has at least two leaf-vertices $v$ and $v^\prime$. Since $T$ has an internal edge, $v$ and $v^\prime$ are neighbors of internal vertices and, further, we can choose $v$ and $v^\prime$ so that they are neighbors of distinct internal vertices. Let $u$ be the internal vertex in $N(v)$ and $u^\prime$ be the internal vertex in $N(v^\prime)$. Then, the leaf component containing $u$ has at least order 2, and so does the leaf component containing $u^\prime$.
\end{proof}

%%%%%%%%%%%%%%%%%%%%%%%%%%%%%%%%%%%%%%%%%%%%%%%%%%%%%%%%%%%%%%
\subsection{DNC operations and the leaf component partition.}
In this subsection, we consider how $\lc(F)$ relates to $\lc(F\setminus e), \lc((F \odot e)\setminus \ell_e)$ and $\lc(F\odot e)$ for an arbitrary forest $F$ with internal edge $e$.
This will help us prove our main result by allowing us to identify 
the paths (or sequences) in the DNC-tree of $F$ from $F$ to a star forest $H$ which satisfy $\lambda(H)=\lc(F)$.

\begin{definition} Let $F$ be a forest. An internal vertex $u$ is 
    a \defn{deep vertex} if the leaf component containing $u$ is $\mathcal{L}_u= St_1$, the single vertex $u$. In other words, $u$ is an internal vertex without leaves in its neighborhood.
\end{definition}

\begin{example} In the following tree, there are only two deep vertices labeled $u$ and $v_1$. Notice that $u$ is a deep vertex of degree 2 and $v_1$ is a deep vertex of degree 3 and $v_2$ is internal but not a deep vertex. 
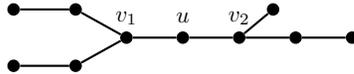
\begin{center}
        \begin{tikzpicture}[auto=center,every node/.style={circle, fill=black, scale=0.5}, scale=0.75]
            \node (b) at (1, 0) {};
            \node (c) at (.1, .5) {};
            \node (d) at (.1, -.5) {};
            \node (e) at (2, 0) {};
            \node (f) at (3, 0) {};
            \node (g) at (4, 0) {};
            \node (h) at (-1,.5) {};
            \node (i) at (-1,-.5) {};
            \node (j) at (5, 0) {};
            \node (k) at (3.6,.5) {};
                
            \draw[thick] (b) -- (c);
            \draw[thick] (b) -- (d);
            \draw[thick] (b) -- (e);
            \draw[thick] (e) -- (f);
            \draw[thick] (f) -- (g);
            \draw[thick] (c) -- (h);
            \draw[thick] (d) -- (i);
            \draw[thick] (g) -- (j);
            \draw[thick] (k) -- (f);
            \node[fill=none, scale=1.75] at (2, .35) {$u$};
            \node[fill=none, scale=1.75] at (1, .35) {$v_1$};
            \node[fill=none, scale=1.75] at (3, .35) {$v_2$};
    \end{tikzpicture}
    \end{center} 
\end{example}

We define a function \defn{$\sort()$} which takes as input a sequence of positive integers and outputs the sequence containing those positive integers in nonincreasing order. Given two sequences $a=(a_1,\ldots,a_s)$ and $b=(b_1,\ldots,b_t)$, let $a \cdot b = (a_1,\ldots, a_s, b_1,\ldots, b_t)$. In addition, we will use the notation $(a_1, \ldots, \widehat{a_i}, \ldots, a_s)$ to denote that the $i$-th term has been omitted. 

The following lemma shows that in general deleting an edge $e$ from a forest $F$ produces a forest $F'=F\setminus e$ such that $\lc(F)\neq \lc(F')$. In addition, the lemma describes how these partitions would differ.  It also describes for which edges $e$ we have $\lc(F)=\lc(F')$. 

\begin{lemma}(Deletion Lemma) \label{lem:deletions-lead}
    Let $F$ be a forest with $\lc(F) = (\kappa_1, \ldots, \kappa_m)$ and assume that $e=uv$ is an internal edge of $F$.
    \begin{enumerate}
        \item[(a)] If both endpoints of $e$ are deep vertices of degree 2 with 
         $N(u)=\{t,v\}$ and $N(v) = \{u,w\}$, then $\lc(F\setminus e)=\sort(\kappa_i+1, \kappa_j+1, \kappa_1,\ldots,\widehat{\kappa_i}, \ldots,\widehat{\kappa_j},\ldots,\kappa_{m-2})$, where $\kappa_i$ and $\kappa_j$ are the orders of the leaf components containing $t$ and $w$, respectively. \label{deletions-lead-2}
        \item[(b)] If only one endpoint of $e$ is a deep vertex of degree 2, without loss of generally assume $u$ is this vertex and $N(u) = \{t,v\}$, then $\lc(F\setminus e)=\sort(\kappa_i+1,\kappa_1,\ldots,\widehat{\kappa_i}, \ldots,\kappa_{m-1})$, where $\kappa_i$ is the order of the leaf component containing the vertex $t$. \label{deletions-lead-1}
        \item[(c)] If neither $u$ nor $v$ are deep vertices of degree $2$, then $\lc(F) = \lc(F \setminus e)$. \label{deletions-lead-3}
    \end{enumerate}
\end{lemma}

\begin{proof}
     A forest $F$ is a finite collection of trees, $T_1\sqcup T_2 \sqcup \ldots \sqcup T_s$ and $\lc(F) = \sort(\lc(T_1)\cdot \lc(T_2) \cdots \lc(T_s))$.  Deleting an internal edge affects only the connected component containing that edge. Hence, we can restrict ourselves to proving the three claims hold for a tree, $T$.  

    If $e= uv\in I(T)$, then $T\setminus e$ has two connected components $T_u$ and $T_v$ containing the vertices $u$ and $v$, respectively. In addition, both $u$ and $v$ have degree greater than 1 in $T$. Then,
    $u$ and $v$ are leaves in $T\setminus e = T_u \sqcup T_v$ if and only if they have degree 2 in $T$. 

    Recall that, by definition, the leaf components of $T$ are the connected components of $T\setminus I(T)$ having orders $\kappa_1\geq \cdots\geq \kappa_m$, where $m$ is the number of leaf components in $T$ which is equal to $\#I(T)+1$ by Proposition \ref{properties-LCpartition} (b).
    
    We begin by proving (b). We assume $u$ is a deep vertex of degree 2 and $v$ is not. Hence $u$ is a leaf component of $T$ of size 1, implying that $\kappa_m = 1$ in $\lc(T)$.  Since $v$ is not a deep vertex of degree 2, there are two cases we need to consider, either $v$ has degree at least 3 or $v$ has degree 2, in which case it is adjacent to a leaf.  If $v$ has degree 2 and adjacent to a leaf, then $T_v$ is a leaf component of size 2, $St_2$, and this is also a leaf component of $T$. If $v$ has degree greater or equal to 3 in $T$, then $v$ is an internal vertex in $T_v$ since it has degree greater or equal to 2.  No other vertex changes degree in $T_v$ when we delete $e$; hence, all leaf components in $T_v$ are leaf components in $T$ of the same orders.     

    Since $u$ has degree 2 and $N(u) =\{v,t\}$ in $T$, $u$ is a leaf in $T_u$ adjacent to $t$. Therefore, the leaf component of $T$ containing $t$, which had order $\kappa_i$ in $T$, for some $i\leq m$, will have order $\kappa_i +1$ in $T_u$.  In addition, $u$ is not a leaf component of $T\setminus e$ as the edge 
    $tu\notin I(T\setminus e)$, hence $\kappa_m =0$ in $\lc(T\setminus e)$. 
    Deleting $e$ from $T$ only decreases the degree of $u$ in $T_u$, hence all other vertices in $T_u$ have the same degrees as in $T$. In particular the only leaf in $T_u$ that is not a leaf in $T$ is $u$.  Therefore, all leaf components not containing $t$ in $T_u$ have the same order as in $T$. Hence, the conclusion of the claim follows since $\lc(T\setminus e) = \sort(\lc(T_u)\cdot \lc(T_v))$. 

    Proving (a) is similar. Here we assume that both $u$ and $v$ are deep vertices of degree 2. This means $\kappa_m=\kappa_{m-1}=1$ since $u$ and $v$ are each leaf components of order 1 in $T$. Since $N(u) = \{t, v\}$ and $N(v) = \{ u, w\}$, then both $t$ and $w$ are internal vertices.  If the leaf components of $t$ and $w$ of $T$ had orders $\kappa_i$ and $\kappa_j$, respectively. Then, in $T\setminus e$ the leaf component of $t$ has one additional leaf, $u$, hence it has order $\kappa_i + 1$ and $v$ is an additional leaf of the leaf component of $w$, hence it has order $\kappa_j+1$ and deleting $e$ only changes the degrees of $v$ and $u$. Since the edges $tu, vw \notin I(T\setminus e)$, the vertices $u$ and $v$ are not leaf components of $T\setminus e$, hence $\kappa_m=0$ and $\kappa_{m-1}=0$ in $\lc(T\setminus e)$. It also follows from the same argument as for (b) that all other leaf components have the same order in $T$ as in $T\setminus e$, so (a) follows.

    For (c), assume that $u$ and $v$ are both not deep of degree 2. Then there are three cases to consider: (1) both $u$ and $v$ have degree greater or equal to 3; (2) one of $u$ or $v$ has degree 2; and (3) both have degree 2.  In the first case, if $u$ and $v$ have degree greater or equal to 3, then $T\setminus e$ has the same internal vertices and same leaves as $T$. Together with Proposition \ref{properties-LCpartition}(a), this implies $\lc(T) = \lc(T\setminus e)$. In case (2), we assume that $v$ has degree 2 and $u$ has degree greater or equal to 3. Since $v$ is not deep, then 
    $T\setminus e = T_u \sqcup T_v$ and $T_v=St_2$, which is also a leaf component of $T$. Since $u$ is an internal vertex in $T\setminus e$, then $T$ and $T\setminus e$ have the same leaf components.  For case (3), if both $u$ and $v$ have degree 2, then they are not deep and each is adjacent to a leaf.  Hence, $T = P_4$, the path with four vertices with the edge $e=uv$ in the middle. Then $T$ and $T\setminus e$  both have two leaf components equal to $St_2$. 
\end{proof}

\begin{example} Figure \ref{fig:deletion-lemma} illustrates Lemma \ref{lem:deletions-lead} parts (a),(b), and (c), respectively.
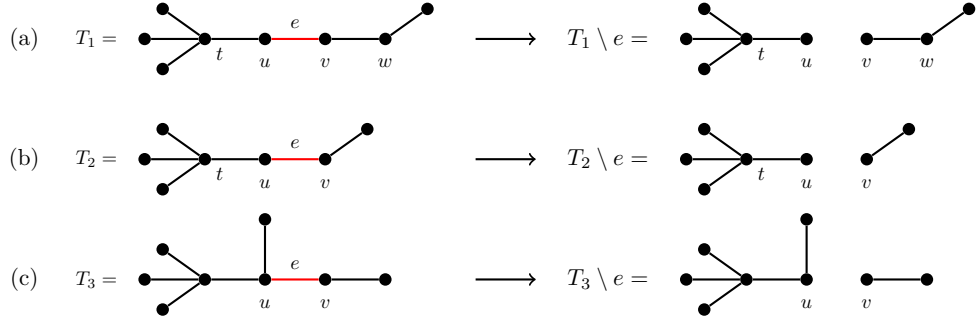
\begin{figure}[ht]
    \centering
    % Part (a)
    \begin{tikzpicture}[auto=center,every node/.style={circle, fill=black, scale=0.5}, scale=0.8]
        \node (a) at (0, 0) {};
        \node (b) at (1, 0) {};
        \node (c) at (.3, .5) {};
        \node (d) at (.3, -.5) {};
        \node (e) at (2, 0) {};
        \node (f) at (3, 0) {};
        \node (g) at (4, 0) {};
        \node (h) at (4.7, .5) {};
        \node[fill = none, scale=1.5] (t) at (1.25, -0.25) {$t$};
        \node[fill = none, scale=1.5] (u) at (2, -0.4) {$u$};
        \node[fill = none, scale=1.5] (v) at (3, -0.4) {$v$};
        \node[fill = none, scale=1.5] (w) at (4, -0.4) {$w$};
        \node[fill = none, scale=1.5] (t) at (10.25, -0.25) {$t$};
        \node[fill = none, scale=1.5] (u) at (11, -0.4) {$u$};
        \node[fill = none, scale=1.5] (v) at (12, -0.4) {$v$};
        \node[fill = none, scale=1.5] (w) at (13, -0.4) {$w$};
        
        \draw[thick] (a) -- (b);
        \draw[thick] (b) -- (c);
        \draw[thick] (b) -- (d);
        \draw[thick] (b) -- (e);
        \draw[thick][color=red, thick] (e) -- (f);
        \draw[thick] (f) -- (g);
        \draw[thick] (g) -- (h);
        
        \node[fill = none, scale=1.5] (p) at (2.5,0.25) {$e$};
        \node[fill = none, scale=1.5] (q) at (-.8, 0) {$T_1=$};
        
        \draw[->, thick] (5.5,0) -- (6.5,0);
        
        \node (a1) at (9, 0) {};
        \node (b1) at (10, 0) {};
        \node (c1) at (9.3, .5) {};
        \node (d1) at (9.3, -.5) {};
        \node (e1) at (11, 0) {};
        \node (f1) at (12, 0) {};
        \node (g1) at (13, 0) {};
        \node (h1) at (13.7, .5) {};
        
        \draw[thick] (a1) -- (b1);
        \draw[thick] (b1) -- (c1);
        \draw[thick] (b1) -- (d1);
        \draw[thick] (b1) -- (e1);
        \draw[thick] (f1) -- (g1);
        \draw[thick] (g1) -- (h1);

        \node[fill = none, scale=1.75] at (-2,0) {(a)};
         \node[fill = none, scale=1.75] (q1) at (7.7,0) {$T_1\setminus e =$};
        % (b)
        \node (a2) at (0, -2) {};
        \node (b2) at (1, -2) {};
        \node (c2) at (.3, -1.5) {};
        \node (d2) at (.3, -2.5) {};
        \node (e2) at (2, -2) {};
        \node (f2) at (3, -2) {};
        \node (g2) at (3.7, -1.5) {};
        \node[fill = none, scale=1.5] (t) at (1.25, -2.25) {$t$};
        \node[fill = none, scale=1.5] (u) at (2, -2.4) {$u$};
        \node[fill = none, scale=1.5] (v) at (3, -2.4) {$v$};
        \node[fill = none, scale=1.5] (t) at (10.25, -2.25) {$t$};
        \node[fill = none, scale=1.5] (u) at (11, -2.4) {$u$};
        \node[fill = none, scale=1.5] (v) at (12, -2.4) {$v$};
        
        \draw[thick] (a2) -- (b2);
        \draw[thick] (b2) -- (c2);
        \draw[thick] (b2) -- (d2);
        \draw[thick] (b2) -- (e2);
        \draw[thick][color=red, thick] (e2) -- (f2);
        \draw[thick] (f2) -- (g2);
        
        \node[fill = none, scale=1.5] (p2) at (2.5,-1.75) {$e$};
        \node[fill = none, scale=1.5] (q2) at (-.8, -2) {$T_2=$};
        
        \draw[->, thick] (5.5,-2) -- (6.5,-2);
        
        \node (a3) at (9, -2) {};
        \node (b3) at (10, -2) {};
        \node (c3) at (9.3, -1.5) {};
        \node (d3) at (9.3, -2.5) {};
        \node (e3) at (11, -2) {};
        \node (f3) at (12, -2) {};
        \node (g3) at (12.7, -1.5) {};
        
        \draw[thick] (a3) -- (b3);
        \draw[thick] (b3) -- (c3);
        \draw[thick] (b3) -- (d3);
        \draw[thick] (b3) -- (e3);
        \draw[thick] (f3) -- (g3);

        \node[fill = none, scale=1.75] at (-2,-2) {(b)};
         \node[fill = none, scale=1.75] (q3) at (7.7,-2) {$T_2\setminus e =$};
         
        \node (a4) at (0, -4) {};
        \node (b4) at (1, -4) {};
        \node (c4) at (.3, -3.5) {};
        \node (d4) at (.3, -4.5) {};
        \node (e4) at (2, -4) {};
        \node (f4) at (3, -4) {};
        \node (g4) at (4, -4) {};
        \node (h4) at (2, -3) {};
        \node[fill = none, scale=1.5] (u) at (2, -4.4) {$u$};
        \node[fill = none, scale=1.5] (v) at (3, -4.4) {$v$};
        \node[fill = none, scale=1.5] (u) at (11, -4.4) {$u$};
        \node[fill = none, scale=1.5] (v) at (12, -4.4) {$v$};
        
        \draw[thick] (a4) -- (b4);
        \draw[thick] (b4) -- (c4);
        \draw[thick] (b4) -- (d4);
        \draw[thick] (b4) -- (e4);
        \draw[thick][color=red, thick] (e4) -- (f4);
        \draw[thick] (f4) -- (g4);
        \draw[thick] (e4) -- (h4);
        
        \node[fill = none, scale=1.5] (p) at (2.5,-3.75) {$e$};
        \node[fill = none, scale=1.5] (q) at (-.8, -4) {$T_3=$};
        
        \draw[->, thick] (5.5,-4) -- (6.5,-4);
        
        \node (a5) at (9, -4) {};
        \node (b5) at (10, -4) {};
        \node (c5) at (9.3, -3.5) {};
        \node (d5) at (9.3, -4.5) {};
        \node (e5) at (11, -4) {};
        \node (f5) at (12, -4) {};
        \node (g5) at (13, -4) {};
        \node (h5) at (11, -3) {};
        
        \draw[thick] (a5) -- (b5);
        \draw[thick] (b5) -- (c5);
        \draw[thick] (b5) -- (d5);
        \draw[thick] (b5) -- (e5);
        \draw[thick] (f5) -- (g5);
        \draw[thick] (e5) -- (h5);

        \node[fill = none, scale=1.75] at (-2,-4) {(c)};
         \node[fill = none, scale=1.75] (q5) at (7.7,-4) {$T_3\setminus e =$};
         
    \end{tikzpicture}

    \caption{The leaf component partitions are (a) $\lc(T_1) = (4,2,1,1)$ and $\lc(T_1\setminus e) = (5,3)$; (b) $\lc(T_2) = (4,2,1)$ and $\lc(T_2\setminus e) = (5,2)$; (c) $\lc(T_3) = \lc(T_3\setminus e) = (4,2,2)$. }
    \label{fig:deletion-lemma}
\end{figure}
\end{example}

\begin{lemma}(Dot-contraction lemma)\label{lemma:dot-contraction-lead}    
    Let $F$ be a forest and suppose that $e=uv$ is an internal edge in $F$. If $\lc(F)=(\kappa_1,\ldots,\kappa_i,\ldots,\kappa_j,\ldots,\kappa_m)$, where $\kappa_i$ and $\kappa_j$ are the orders of the leaf components of $F$ that contain $u$ and $v$, respectively, then $\lc((F\odot e)\setminus \ell_e)=\sort(\kappa_i+\kappa_j-1, \kappa_1,\ldots,\widehat{\kappa_i},\ldots,\widehat{\kappa_j},\ldots,\kappa_{m}, 1)$.
\end{lemma}
\begin{proof} 
    If $F$ is a forest with connected components $T_1,\ldots, T_s$, we have 
    \[\lc(F) = \sort(\lc(T_1)\cdot \lc(T_2)\cdots \lc(T_s)).\] 
    Dot-contracting an edge in $F$ will change the leaf component partition of only one of the connected components in $F$. Therefore, we restrict ourselves to the case of a tree, $T$.  

    If $\lc(T) = (\kappa_1, \ldots, \kappa_m)$ and $e=uv$ is an internal edge with $u$ having $\kappa_i-1$ incident leaves and $v$ having $\kappa_j -1$ incident leaves in $T$, then $(T\odot e)\setminus \ell_e$ consists of two connected components, one is a single vertex and the other component has one fewer vertex than $T$ since the endpoints of $e$ contracted to a single vertex $w$ with $\kappa_i+\kappa_j -2$ incident leaves.  Therefore, the leaf component containing $w$ has order $\kappa_i + \kappa_j -1$. All other leaf components in $T$ are unchanged by the dot-contraction operation.  Hence, $\lc((F\odot e)\setminus \ell_e)=\sort(\kappa_i+\kappa_j-1, \kappa_1,\ldots,\widehat{\kappa_i},\ldots,\widehat{\kappa_j},\ldots,\kappa_{m}, 1)$.
\end{proof}
\begin{example}  Figure \ref{fig:dot-contraction-lead} illustrates Lemma 4.8. 
\begin{figure}[ht]
    \centering
    \begin{tikzpicture}[auto=center,every node/.style={circle, fill=black, scale=0.5}, scale=0.8]
        \node (a) at (0, 0) {};
        \node (b) at (1, 0) {};
        \node (c) at (.5, .5) {};
        \node (d) at (.5, -.5) {};
        \node (e) at (2, 0) {};
        \node (f) at (3, 0) {};
        \node (g) at (4, 0) {};
        \node (h) at (2, 1) {};
        \node (h11) at (1.5,.8) {};
        \node (h22) at (2.5,.8) {};
        \node (g11) at (3.7,.7) {};
        
        \node[fill = none, scale=1.5] (u) at (2, -0.4) {$u$};
        \node[fill = none, scale=1.5] (v) at (3, -0.4) {$v$};
        
        \draw[thick] (a) -- (b);
        \draw[thick] (b) -- (c);
        \draw[thick] (b) -- (d);
        \draw[thick] (b) -- (e);
        \draw[thick][color=red, thick] (e) -- (f);
        \draw[thick] (f) -- (g);
        \draw[thick] (e) -- (h);
        \draw[thick] (e) -- (h11);
        \draw[thick] (e) -- (h22);
        \draw[thick] (f) -- (g11);
        
        \node[fill = none, scale=1.5] (p) at (2.5,0.25) {$e$};
        \node[fill = none, scale=1.5] (q) at (-.7,0) {$T =$};
        
        \draw[->, thick] (4.5,0)--(5.5,0);

        \node (a1) at (8, 0) {};
        \node (b1) at (9, 0) {};
        \node (c1) at (8.5, .5) {};
        \node (d1) at (8.5, -.5) {};
        \node (e1) at (10, 0) {};
        \node (f1) at (11, 0) {};
        \node (h1) at (10, 1) {};
        \node (i1) at (9.5,.8) {};
        \node (j1) at (10.5,.8) {};
       \node (k1) at (10.9,.5) {};
        \node (g1) at (9.7, -.7) {};
         
        \draw[thick] (a1) -- (b1);
        \draw[thick] (b1) -- (c1);
        \draw[thick] (b1) -- (d1);
        \draw[thick] (b1) -- (e1);
        \draw[thick] (e1) -- (f1);

        \draw[thick] (e1)--(i1);
        \draw[thick] (e1)--(j1);
        \draw[thick] (e1)--(k1);
        \draw[thick] (e1)--(h1);
        
        \node[fill = none, scale=1.5] (q1) at (6.7,0) {$(T\odot e)\setminus \ell_e=$};
    \end{tikzpicture}
    \caption{ We have $\lc(T)=(4,4,3)$, and $\lc((T\odot)\setminus\ell_e)=(6,4,1)$}
    \label{fig:dot-contraction-lead}
\end{figure}
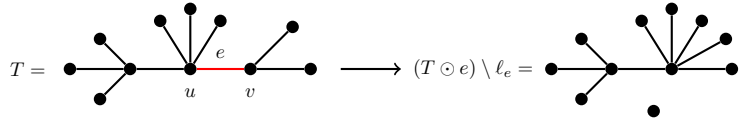
\end{example}
  As we will see later in Theorem \ref{thm:lead-inequalities}, it will be possible to have $\lc(F) = \lc((F\odot e)\setminus \ell_e)$. We now consider leaf-contractions. 

\begin{lemma}(Leaf-contraction lemma)\label{lem:leaf-contraction-lead}
    Suppose $F$ has an internal edge $e=uv$. If $\lc(F)=(\kappa_1,\ldots,\kappa_i,\ldots,\kappa_j,\ldots,\kappa_m)$, where $\kappa_i$ and $\kappa_j$ are the orders of leaf components that contain $u$ and $v$, then $\lc(F\odot e)=\sort( \kappa_i+\kappa_j, \kappa_1,\ldots,\widehat{\kappa_i},\ldots \widehat{\kappa_j},\ldots, \kappa_m)$.
\end{lemma}
\begin{proof}
    As in the proofs of Lemma \ref{lem:deletions-lead} and Lemma \ref{lemma:dot-contraction-lead}, we observe that leaf-contraction only affects the connected component containing the edge $e=uv$ and all other connected components of $F$ are unchanged. Assume $\kappa_i$ and $\kappa_j$ are the orders of the leaf components containing $u$ and $v$ respectively. 
    
    The leaf-contraction operation on the edge $e$ merges the leaf component containing $u$ with the leaf component containing $v$ into a single leaf component and adds one new leaf to this new leaf component.  Hence, the order of the newly created leaf component is $(\kappa_i + \kappa_j-1)+1=\kappa_i + \kappa_j$, where the minus 1 results from $u$ and $v$ merging into one vertex and the plus 1 is the contribution of the new leaf.  Hence, the claim follows. 
\end{proof}
\begin{example} Illustration of Lemma \ref{lem:leaf-contraction-lead}. 
\begin{figure}[ht]
    \centering
    \begin{tikzpicture}[auto=center,every node/.style={circle, fill=black, scale=0.5}, scale=0.8]
        \node (a) at (0, 0) {};
        \node (b) at (1, 0) {};
        \node (c) at (.5, .5) {};
        \node (d) at (.5, -.5) {};
        \node (e) at (2, 0) {};
        \node (f) at (3, 0) {};
        \node (g) at (4, 0) {};
        \node (h) at (2, 1) {};
        \node (h11) at (1.5,.8) {};
        \node (h22) at (2.5,.8) {};
        \node (g11) at (3.7,.7) {};
        
        \node[fill = none, scale=1.5] (u) at (2, -0.4) {$u$};
        \node[fill = none, scale=1.5] (v) at (3, -0.4) {$v$};
        
        \draw[thick] (a) -- (b);
        \draw[thick] (b) -- (c);
        \draw[thick] (b) -- (d);
        \draw[thick] (b) -- (e);
        \draw[thick][color=red, thick] (e) -- (f);
        \draw[thick] (f) -- (g);
        \draw[thick] (e) -- (h);
        \draw[thick] (e) -- (h11);
        \draw[thick] (e) -- (h22);
        \draw[thick] (f) -- (g11);
        
        \node[fill = none, scale=1.5] (p) at (2.5,0.25) {$e$};
        \node[fill = none, scale=1.5] (q) at (-.7,0) {$T =$};
        
        \draw[->, thick] (4.5,0)--(5.5,0);

        \node (a1) at (8, 0) {};
        \node (b1) at (9, 0) {};
        \node (c1) at (8.5, .5) {};
        \node (d1) at (8.5, -.5) {};
        \node (e1) at (10, 0) {};
        \node (f1) at (11, 0) {};
        \node (h1) at (10, 1) {};
        \node (i1) at (9.5,.8) {};
        \node (j1) at (10.5,.8) {};
       \node (k1) at (10.9,.5) {};
        \node (g1) at (9.7, -.7) {};

        \draw[thick] (a1) -- (b1);
        \draw[thick] (b1) -- (c1);
        \draw[thick] (b1) -- (d1);
        \draw[thick] (b1) -- (e1);
        \draw[thick] (e1) -- (f1);

        \draw[thick] (e1)--(i1);
        \draw[thick] (e1)--(j1);
        \draw[thick] (e1)--(k1);
        \draw[thick] (e1)--(h1);
        \draw[thick] (e1) -- (g1);
        
        \node[fill = none, scale=1.5] (q1) at (6.7,0) {$(T\odot e)=$};
    \end{tikzpicture}
    % \vspace{-1cm}
    \caption{ We have $\lc(T)=(4,4,3)$, and $\lc(T\odot e)=(7,4)$}
    \label{fig:leaf-contraction-lead}
\end{figure}
\end{example}

In Lemmas \ref{lem:deletions-lead}, \ref{lem:leaf-contraction-lead}, and \ref{lemma:dot-contraction-lead}, we have computed the leaf component partitions after applying the three DNC operations to a forest $F$.  We now use these lemmas to characterize when $\lc(F)$ is equal to $\lc(F')$, where $F'$ results from applying a DNC operation to $F$.  In the next theorem, $<$ denotes lexicographic order of partitions.

\begin{theorem}\label{thm:lead-inequalities}
Let $F$ be a forest with $\lc(F)=(\kappa_1,\ldots,\kappa_i,\ldots,\kappa_j,\ldots,\kappa_m)$ and at least one internal edge $e=uv$. Then:
\begin{enumerate}
    \item[(a)] $\lambda_{LC}(F) = \lambda_{LC}(F\setminus e)$ if and only if neither $u$ nor $v$ is a degree-$2$ deep vertex. 
    \item[(b)] $\lc(F) = \lc((F\odot e)\setminus \ell_e)$ if and only if at least one of $u$ or $v$ is a deep vertex.  
    \item[(c)] $\lc(F) < \lc(F\odot e)$.
\end{enumerate}
Furthermore, $\lc(F)\leq \lc(F\setminus e)$, $\lc(F)\leq \lc((F\odot e)\setminus \ell_e)$, and $\lc(F)\leq \lc(F\odot e)$.
\end{theorem}

\begin{proof} We  prove (a)-(c) and then the weak inequality follows immediately. The lexicographic order is a complete order on partitions of the same weight. Since $F'$ is the forest resulting from either deletion, dot-contraction, or leaf-contraction of an edge in $F$, the weights of $\lambda_{LC}(F)$ and $\lambda_{LC}(F')$ are equal to the number of vertices in $F$.

      Part (a) follows immediately from the arguments in Lemma \ref{lem:deletions-lead}. If neither endpoint of $e$ is a deep vertex of degree 2, then by \ref{lem:deletions-lead}(c)  we have $\lambda(F)=\lambda(F\setminus e)$.  If one or both endpoints are deep vertices of degree 2, then by parts (a) and (b) of the lemma, we either have $\lc(F \setminus e) = \sort(\kappa_i+1, \kappa_j+1, \kappa_1,\ldots,\widehat{\kappa_i}, \ldots,\widehat{\kappa_j},\ldots,\kappa_{m-2})$ or $\lc(F \setminus e) = \sort(\kappa_i+1, \kappa_1,\ldots,\widehat{\kappa_i}, \ldots,\kappa_{m-1})$. In both cases, $\lc(F \setminus e) > \lc(F)$.

     Part (b) follows from Lemma \ref{lemma:dot-contraction-lead}. If we dot-contract the edge $e$, then $\lambda_{LC}((F\odot e))\setminus 
     \ell_e) = \sort(\kappa_i+\kappa_j -1, \kappa_1, \ldots, 
     \widehat{\kappa_i}, \ldots, \widehat{\kappa_j}, \kappa_m, 1)$, where $\kappa_i$ and $\kappa_j$ are the orders of the leaf components in $F$ containing $u$ and $v$, respectively, and $1\leq i, j\leq m$.  
     Without loss of generality, assume $u$ is a deep vertex in $F$. Then 
     $\kappa_i =1$ in $F$ and $\kappa_i+\kappa_j -1 = \kappa_j$ in $(F\odot e)\setminus \ell_e$. In addition, we gain an extra 1 from the single vertex added in the dot-operation. It then follows that $\lambda_{LC}(F) = \lambda((F\odot e)\setminus \ell_e)$.  Now if neither $u$ nor $v$ are deep vertices, then $\kappa_i,\kappa_j\geq 2$, which implies $\kappa_i+\kappa_j-1>\kappa_i$ and $\kappa_i+\kappa_j-1> \kappa_j$. Hence, $\lc((F\odot e)\setminus \ell_e)>\lc(F)$.

    Part (c) follows from Lemma \ref{lem:leaf-contraction-lead}. We have
    $\lc(F \odot e) = \sort(\kappa_i+\kappa_j, \kappa_1, \ldots, \widehat{\kappa_i}, \ldots, \widehat{\kappa_j}, \ldots, \kappa_m)$. Since $\kappa_i+\kappa_j > \kappa_i$ and $\kappa_i+\kappa_j >\kappa_j$, it follows that $\lc(F\odot e) > \lc(F)$.
\end{proof}

We remark that $\lc(F)\neq \lc(F\odot e)$ for any $F$ and any internal edge $e$.  Thus, all sequences containing $R$s in a DNC-tree of $F$ cannot end in star forests $H$ with $\lambda(H) = \lead(\mathbf{X}_F)$.

%%%%%%%%%%%%%%%%%%%%%%%%%%%%%%%%%%%%%%%%%%%%%%%%%%%%%%%%%%%%%%
\subsection{The Leading Term and $\lc$} In Theorem \ref{thm:lead-inequalities}, we observed that applying DNC operations on internal edges of a forest, $F$, does not produce a partition smaller in lexicographic order than $\lambda_{LC}(F)$.  Since we compute $\mathbf{X}_F$ in the star-basis by applying DNC operations, the following lemma holds. 

\begin{lemma}\label{lem:lower-bound-leading}
    Let $F$ be a forest on $n$ vertices. Then, $ \lead(\csff) \geq \lc(F)$. 
\end{lemma}

We have just established a lower bound on the leading partition of $\mathbf{X}_F$. To conclude that $\lc(F)$ is, in fact, the leading partition of $\mathbf{X}_F$, we must show that there exists a path in any DNC-tree of $F$ from the root $F$ to the star-forest $F\setminus I(F)$. This follows almost immediately from the following lemma.

\begin{lemma} \label{lem:deg2-deep-vert-either-or}
For any internal edge $e=uv$ in a forest $F$, the endpoints of $e$ satisfy at least one of the following:
    \begin{enumerate}
        \item[(a)] Neither $u$ nor $v$ is a deep vertex of degree 2.
        \item[(b)] At least one of $u$ and $v$ is a deep vertex. 
    \end{enumerate}
\end{lemma}
\begin{proof}
If (b) holds, there is nothing to show. If (b) does not hold, then neither $u$ nor $v$ are deep vertices. In particular, neither $u$ nor $v$ are deep vertices of degree 2, so (a) holds
\end{proof}

We now present one of the main results of this section, which follows from all the previous results. 

\begin{theorem} \label{thm:leading-partition}
    Let $F$ be a forest with $n$ vertices. Then $ \lead(\csff) = \lc(F)$.
\end{theorem}
\begin{proof}
    For all $\lambda \vdash n$ with $\lambda<\lc(F)$, it follows by Lemma \ref{lem:lower-bound-leading} that $c_\lambda = 0$.  Let  $\mathcal{D}$ denote a fixed DNC-tree of $F$. At any level $k\geq 1$ of $\mathcal{D}$, if $F_k$ is a forest at this level containing at least one internal edge, then in $\mathcal{D}$ there exists at least one edge from $F_k$ to a forest $F_{k+1}$ at level $k+1$, for which $\lc(F_k)=\lc(F_{k+1})$. To see this, notice that by Lemma \ref{lem:deg2-deep-vert-either-or}, if $e$ is an internal edge of $F_k$, its endpoints satisfy at least one of properties (a) and/or (b).  If $e$ satisfies (a), i.e., both endpoints of $e$ are not degree-2 deep vertices, then deleting $e$ results in a forest in level $k+1$ of $\mathcal{D}$ with the same leaf component partition as $F_k$ by Theorem \ref{thm:lead-inequalities}(a). And if $e$ satisfies (b), i.e., at least one endpoint of $e$ is a deep vertex, then dot-contracting $e$ results in a forest with the same leaf component partitions by Theorem \ref{thm:lead-inequalities}(b). It follows inductively that there exists a path in $\mathcal{D}$ from $F$ to a star-forest $\widetilde{F}$ such that $\lc(\widetilde{F})=\lc(F)$. That is $c_{\lc(F)}\neq 0$. It then follows by Lemma \ref{lem:lower-bound-leading} that $\lc(F)$ is the leading partition of $\csff$.
\end{proof}

\begin{example} Let $T$ be the tree below.
    
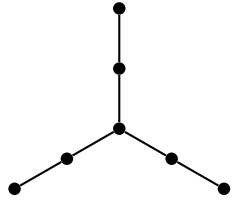
\begin{center}
        \begin{tikzpicture}[auto=center,every node/.style={circle, fill=black, scale=0.5}, scale=0.8]
            \node (1a) at (0, 2) {};
            \node (1b) at (0, 1) {};
            \node (1c) at (0, 0) {};
            \node (1d) at (-0.87, -0.5) {};
            \node (1e) at (-1.74, -1) {};
            \node (1f) at (0.87, -0.5) {};
            \node (1g) at (1.74, -1) {};
            
            \draw[thick] (1a) -- (1b) -- (1c);
            \draw[thick] (1c) -- (1d) -- (1e);
            \draw[thick] (1c) -- (1f) -- (1g);
            
        \end{tikzpicture}
    \end{center}
    
    \noindent By definition, we have that $\lc(T)=(2,2,2,1)$. We also have
    \[
    \csft = -2\mathfrak{st}_{(2^3,1)} +3\mathfrak{st}_{(3,2,1^2)}   +3\mathfrak{st}_{(3,2^2)}
    -\mathfrak{st}_{(4,1^3)}
    -6\mathfrak{st}_{(4,2,1)}+ 3\mathfrak{st}_{(5,1,1)} + 3\mathfrak{st}_{(5,2)} - 3\mathfrak{st}_{(6,1)} +
    \mathfrak{st}_{(7)} 
    \]
    We can see that $\lead(\csft)=(2,2,2,1) = \lc(T)$, as predicted by Theorem \ref{thm:leading-partition}.
\end{example}

We have shown that given any tree $T$, we can now determine the leading partition of $\csft$ based on properties of the tree $T$ itself. This allows us to make some progress towards a positive answer of Stanley's isomorphism conjecture. In particular, we immediately obtain the following corollary:

\begin{corollary} \label{cor:different-leading}
    If $T_1$ and $T_2$ are trees with $T_1 \setminus I(T_1)\not\cong T_2\setminus I(T_2)$, then $\mathbf{X}_{T_1}\neq \mathbf{X}_{T_2}$. Equivalently, if $\lc(T_1)\neq \lc(T_2)$, then $\mathbf{X}_{T_1}\neq \mathbf{X}_{T_2}$. \qed
\end{corollary}

Another consequence of Theorem \ref{thm:leading-partition} is that the number of deep vertices in a tree can be recovered from its leading partition. Notice that a tree with three or fewer vertices are stars. Therefore, they do not have deep vertices.

\begin{corollary} \label{cor:num-deep-vertices}
    For any forest $F$ that has no isolated vertices, the multiplicity of $1$ in $\lead(\csff)$ is the number of deep vertices of $F$.
\end{corollary}
\begin{proof}
     By definition, $\lc(F)=\lambda(F\setminus I(F))$, and by Theorem \ref{thm:leading-partition},  $\lc(F) =\lead(\csff)$. Since there are no isolated vertices in $F$, every component consisting of a single vertex in $F\setminus I(F)$ arises from deleting all internal edges incident to a vertex in $F$ that has no leaves in its neighborhood. That is, there is a natural correspondence between the parts of size 1 in $\lc(F)$ and deep vertices in $F$. 
\end{proof}
\begin{example} The tree below has three deep vertices: $u_1$, $u_2$, and $u_3$.
    
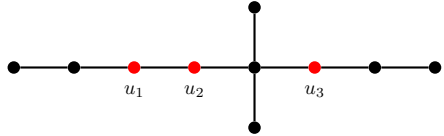
\begin{center}
        \begin{tikzpicture}[auto=center,every node/.style={circle, fill=black, scale=0.5}, scale=0.8]
            \node (1a) at (0, 0) {};
            \node (1b) at (1, 0) {};
            \node (1c)[fill=red] at (2, 0) {};
            \node (1d)[fill=red] at (3, 0) {};
            \node (1e) at (4, 0) {};
            \node (1f)[fill=red] at (5, 0) {};
            \node (1g) at (6, 0) {};
            \node (1h) at (7, 0) {};
            \node (1i) at (4, 1) {};
            \node (1j) at (4, -1) {};

            \node[fill = none, scale=1.5] (u1) at (2, -0.4) {$u_1$};
            \node[fill = none, scale=1.5] (u2) at (3, -0.4) {$u_2$};
            \node[fill = none, scale=1.5] (u3) at (5, -0.4) {$u_3$};
            \draw[thick] (1a) -- (1b) -- (1c) -- (1d) -- (1e) -- (1f) -- (1g) -- (1h);
            \draw[thick] (1i) -- (1e) -- (1j);
        \end{tikzpicture}
    \end{center}
Observe, by computing $\csft$ or by applying Theorem \ref{thm:leading-partition}, that the leading partition of this tree is $(3,2^2,1^3)$. As expected, the number of $1$s equals the number of deep vertices in $T$.
\end{example}

\begin{corollary} \label{cor:prop-trees-deep-vertices}
    For any tree $T$ with at least two vertices, the multiplicity of $1$ in $\lead(\csft)$ is $0$ if and only if $T$ is a proper tree. 
\end{corollary}
\begin{proof}
    By definition, a tree with at least two vertices is proper if and only if each internal vertex has at least one leaf in its neighborhood or, equivalently, if and only if it has no deep vertices. The result then follows immediately by Corollary \ref{cor:num-deep-vertices}.
\end{proof}

We now positively answer Stanley's conjecture for another infinite family of trees.
\begin{definition}
    A \defn{bi-star} is a tree consisting of two star graphs whose centers are joined by an internal edge. An \defn{extended bi-star} is a tree consisting of two star graphs whose centers are connected by a path of one or more deep vertices of degree 2.
\end{definition}

\begin{corollary} \label{cor:bi-stars-leading}
    Let $T$ be a tree with $n$ vertices. Then, $\lead(\csft)=(i,j,1^{n-i-j})$ for some $i,j>1$ if and only if $T$ is a bi-star or extended bi-star with leaf-stars $St_{i}$ and $St_{j}$ separated by $n-i-j$ deep vertices of degree 2.
\end{corollary}
\begin{proof}
    Let $T$ be a tree with $\lead(\csft)=\lc(T)=(i,j,1^{n-i-j})$ for some $i,j>1$. It follows that the connected components of $T \setminus I(T)$ are $St_i$, $St_j$, and $n-i-j$ components of order $1$. Since these order-$1$ components are deep vertices in $T$ and these cannot be leaves by definition, it follows that $T$ consists of the stars $St_i$ and $St_j$ whose centers are connected by a path of $n-i-j$ deep vertices of degree 2. The converse follows simply by an application of Theorem \ref{thm:leading-partition}.
\end{proof}

\begin{example}
    The extended bi-star shown below has leading partition $(6,4,1^4)$
    
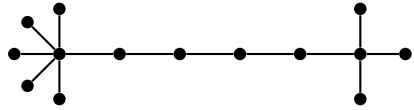
\begin{figure}[ht]
        \centering
        \begin{tikzpicture}[every node/.style={circle, fill=black, scale=0.5}, scale=0.8]
            \node (a) at (0,0) {};
            \node (b) at (1,0) {};
            \node (c) at (2,0) {}; 
            \node (d) at (3,0) {};
            \node (e) at (4,0) {};
            \node (f) at (5,0) {};

            \draw[thick] (a)--(b);
            \draw[thick] (b)--(c);
            \draw[thick] (c)--(d);
            \draw[thick] (d)--(e);
            \draw[thick] (e)--(f);

            \node (a1) at (0,0.75) {};
            \node (a2) at (0,-.75){};
            \node (a3) at (-.75, 0) {};
            \node (a4) at (-.53, .53) {};
            \node (a5) at (-.53, -.53) {};

            \draw[thick] (a)--(a1);
            \draw[thick] (a)--(a2);
            \draw[thick] (a)--(a3);
            \draw[thick] (a)--(a4);
            \draw[thick] (a)--(a5);

            \node (f1) at (5.75,0) {};
            \node (f2) at (5, .75) {};
            \node (f3) at (5, -.75){};

            \draw[thick] (f)--(f1);
            \draw[thick] (f)--(f2);
            \draw[thick] (f)--(f3);
            
        \end{tikzpicture}
    \end{figure}
\end{example}

From the preceding corollary, we immediately obtain the following result:

\begin{corollary} \label{cor:bi-stars-distinguish}
    Bi-stars and extended bi-stars are distinguished by their chromatic symmetric functions.\qed
\end{corollary}

Note that bi-stars and extended bi-stars are particular cases of caterpillars, which are already known to be distinguishable \cite{loebl2018isomorphism}. However, we included the result here since the leading partition provides an almost immediate way to reconstruct them from $\mathbf{X}_T$ directly while in \cite{loebl2018isomorphism,martin2008distinguishing}, the proofs are not reconstructive and use indirect methods.

So far, we have collected a few immediate consequences of Theorem \ref{thm:leading-partition}. In Section \ref{sec:diam5} and \ref{sec:subspace} we will give two bigger applications of Theorem \ref{thm:leading-partition}. 

%%%%%%%%%%%%%%%%%%%%%%%%%%%%%%%%%%%%%%%%%%%%%%%%%%%%%%%%%%%%%%
\subsection{The coefficient of $\lead(\csff)$} The coefficient of $\lead(\csff)$ admits a nice closed formula that only depends on the degrees of the deep vertices of a given forest $F$.

Recall that $\sqcup$ means the disjoint union of disconnected components in a graph.  If $G = H_1 \sqcup H_2$, then $\mathbf{X}_G = \mathbf{X}_{H_1}\mathbf{X}_{H_2}$. 

\begin{lemma} \label{lem:c-lead-forest}
    Let $T_1,\ldots,T_k$ be trees and let $F = \bigsqcup_{i=1}^k T_i.$ Then, $\lead(\csff) = \sort(\lead(\mathbf{X}_{T_1})\cdot \ldots \cdot \lead(\mathbf{X}_{T_k}))$ and $c_{\lead(\csff)}=\prod\limits_{i=1}^k c_{\lead(\mathbf{X}_{T_i})}$.
\end{lemma}

\begin{proof}
 Since we are writing the chromatic symmetric function in the star-basis and $\mathfrak{st}_\lambda\mathfrak{st}_\mu = \mathfrak{st}_{\sort(\lambda\cdot \mu)}$, where $\cdot$ denotes concatenation of two sequences, then the claim follows by the multiplicative property of the chromatic symmetric function. 
\end{proof}

\begin{proposition}\label{prop:no-deep-coeff}
    If $F$ is a proper forest, i.e., $F$ has no deep vertices, then $c_{\lead(\csff)}=1$. 
\end{proposition}
\begin{proof}  
    Since $F$ has no deep vertices, every internal vertex has at least one leaf incident to it.  Let $e=uv$ be an internal edge. Deleting $e$ decreases the degree of $u$ and $v$ by exactly 1. We will show that deleting $e$ reduces the number of internal edges by exactly one. First note that deleting $e$ only changes the degrees of its endpoints. Thus, every other internal edge that is not incident to $u$ or $v$ remains internal. 
    
    If $u$ and $v$ both have degrees $\geq 3$, then any internal edge incident to either $u$ or $v$ will remain an internal edge because $u$ and $v$ have degree $\geq 2$ after deleting $e$.  Suppose that either one or both of $u$ and $v$ have degree 2.  If both have degree $2$, then the connected component of $F$ containing $e=uv$ is a path with 4 vertices since $u$ and $v$ are each incident to $e$ and a leaf-edge. In this case, $e$ is the only internal edge in this connected component and so every other internal edge in $F$ remains internal in $F\setminus e$. If only one endpoint of $e$ has degree 2, we can assume without loss of generality that $u$ has degree 2. Since $u$ is not deep, the two edges incident to $u$ are $e$ and a leaf-edge.  Since $v$ has degree $\geq 3$, any internal edge incident to $v$ remains an internal edge since $v$ has degree $\geq 2$ after deleting $e$. In all cases, $F\setminus e$ has exactly one fewer internal edge than $F$.

    By Lemma \ref{lem:c-lead-forest}, it suffices to show the proposition for a tree $T$. Fix a DNC-tree $\mathcal{D}$ for $T$. We will show that there is exactly one path from $T$ to $T\setminus I(T)$ in $\mathcal{D}$. Let $e_1, \ldots, e_s$ be the internal edges in $T$ listed in the order in which they are operated on in $\mathcal{D}$. We have shown that $e_i$ remains internal in $T \setminus \{e_1, \ldots, e_{i-1}\}$ for each $2 \leq i \leq s$. Theorem \ref{thm:lead-inequalities} implies that any path in $\mathcal{D}$  from $T$ to $T \setminus I(T)$ must begin with a deletion of $e_1$ since any other operation results in a forest with a leading partition greater in lexicographic order than $\lead(\csft)$. Let $k\geq 1$ and assume we have deleted $e_1,\ldots,e_{k}$. Since $e_{k+1}$ remains internal, Theorem \ref{thm:lead-inequalities} implies that $e_{k+1}$ must be deleted as well. It follows by induction that the only path from $F$ to $T\setminus I(T)$ is the path of repeated deletions. Thus, by Corollary \ref{cor:coeff-paths-seq} $c_{\lead(\csft)} = 1$.
\end{proof}

\begin{lemma}\label{lem:deep-vertices-neighborhood}
    If $F$ is a forest containing deep vertices, then there exists a deep vertex in $F$ with at most one deep vertex in its neighborhood.
\end{lemma}
\begin{proof}
    Assume, to the contrary, that each deep vertex in $F$ has at least two deep vertices in its neighborhood. Since $F$ is finite, it follows that $F$ has a cycle consisting of deep vertices, a contradiction.
\end{proof}

We can finally prove the following combinatorial description for the leading coefficient of a forest.

\begin{theorem}\label{thm:leading-coefficient}
    Let $F$ be a forest with $n$ vertices and with deep vertices $u_1, \ldots, u_m$ and leading partition $\lead = \lead(\csff)$. Then:
    \[
    c_{\lead} = (-1)^m \prod_{i=1}^m (\deg(u_i)-1)
    \]
\end{theorem}
\begin{proof}
    We prove the statement by induction on $m$.   If $F$ is a forest with no deep vertices, Proposition \ref{prop:no-deep-coeff} implies that $c_{\lead(\csff)}=1$, so the base case holds.
        
    Assume that the claim is true for all forests with at most $m$ deep vertices, where $m \geq 0$. Let $F$ be a forest with exactly $m+1$ deep vertices and let $u_1,\ldots,u_{m+1}$ be the deep vertices in $F$. Since $m+1\geq 1$, it follows by Lemma \ref{lem:deep-vertices-neighborhood} that there exist a deep vertex in $F$, say $u_1$, such that $u_1$ has at most one deep vertex in its neighborhood. Suppose we fix a DNC-tree $\mathcal{D}$ for $F$ that starts by performing the DNC relation on $\deg(u_1)-1$ internal edges of the form $u_1v_j$, where for all $1 \leq j \leq \deg(u_1)-1$, the vertex $v_j$ is internal but not deep.
    
    In Corollary \ref{cor:coeff-paths-seq}, we showed that for any partition $\lambda$, $c_\lambda = (-1)^r|\mathcal{S}_\lambda|$, 
    where $r$ is the number of 1s in $\lambda$ and $\mathcal{S}_\lambda$ is the set of sequences with terms in $\{L, M, R\}$ that correspond to paths from $F$ to a star forest $H$ such that $\lambda(H) = \lambda$.  Recall that $L$ is deletion, $M$ is dot-contraction and $R$ is leaf-contraction.
    
    Fix an arbitrary path from $F$ to $F\setminus I(F)$ in $\mathcal{D}$, and consider the first $\deg(u_1)-1$ steps along this path. Let $x_1,x_2,\ldots,x_{\deg(u_1)-1}$ where $x_i\in \{L,M,R\}$ be the terms in the sequence that encode the first $\deg(u_1)-1$ operations. Since the only two operations that preserve the leading partition are deletions or dot-contractions by Theorem \ref{thm:lead-inequalities}, then $x_i\in \{L, M\}$.  Furthermore, if all of these steps were deletions, then the final deletion would occur on an edge $u_1v$, where $u_1$ is a degree-2 deep vertex. By Theorem \ref{thm:lead-inequalities}(a), the resulting forest has a greater leading partition than $F$. Therefore, at least one of these labels $x_i$ must be an $M$. Applying a dot-contraction on an edge $wv_j$ where the endpoints are both not deep will not preserve the leading partition by Theorem \ref{thm:lead-inequalities}(b). Hence, only one of the labels can be an $M$, because dot-contracting an edge $u_1v_i$ where only $u_1$ is deep, results in a vertex $w=u_1$ that is no longer a deep vertex (we continue to refer to $w $ as $u_1$).  Note that in particular if $\deg(u_1) = 2$, then only dot-contraction preserves the leading partition.  
        
    Thus, we have shown that any path from $F$ to $F \setminus I(F)$ in $\mathcal{D}$ necessarily has a prefix in the following set
    \[
    S = \{(\underbrace{L, \ldots, L}_{i-1}, M, \underbrace{L, \ldots, L}_{\deg(u_1)-1-i} ) : i \in [\deg(u_1)-1]\}
    \]
    Applying any sequence in $S$, reduces the number of deep vertices by exactly 1, since $w=u_1$ is no longer a deep vertex.  Let $F^\prime$ denote a forest obtained after applying any sequence from $S$. $F^\prime$ has exactly $m$ deep vertices: $u_2, \ldots, u_{m+1}$ and  they have the same degree in $F$ and in $F^\prime$ since no operation was performed on any edge incident to them. Since, $F^\prime$ has exactly $m$ deep vertices by induction hypothesis we have $c_{\lead(\mathbf{X}_{F^\prime})} = (-1)^m \prod\limits_{i=2}^{m+1}(\deg(u_i)-1)$. Therefore, there are $\prod\limits_{i=2}^{m+1}(\deg(u_i)-1)$ paths from $F^\prime$ to $F^\prime \setminus I(F^\prime) = F \setminus I(F)$ in $\mathcal{D}$. 
    
    Now note that there are $\deg(u_1)-1$ distinct paths in $\mathcal{D}$ to obtain such a forest $F^\prime$. Hence, we obtain that the number of paths from $F$ to $F\backslash I(F)$ in $\mathcal{D}$ is 
    \[
    (\deg(u_1)-1)\cdot \prod\limits_{i=2}^{m+1}(\deg(u_i)-1) = \prod_{i=1}^{m+1}(\deg(u_i)-1)
    \]
    Lastly, we have $c_{\lead} = (-1) \cdot (\deg(u_1)-1)\cdot (-1)^m \prod\limits_{i=2}^{m+1}(\deg(u_i)-1) = (-1)^{m+1}\prod\limits_{i=1}^{m+1}(\deg(u_i)-1)$ since there is exactly one dot-contraction ($M$) which contributes a factor $(-1)$ in the first $\deg(u_1)-1$ edges along the path in $\mathcal{D}$.
\end{proof}

%%%%%%%%%%%%%%%%%%%%%%%%%%%%%%%%%%%%%%%%%%%%%%%%%%%%%%%%%%%%%%
\section{Trees with diameter at most 5}\label{sec:diam5}
%%%%%%%%%%%%%%%%%%%%%%%%%%%%%%%%%%%%%%%%%%%%%%%%%%%%%%%%%%%%%%
In this section, we show how to reconstruct trees of diameter at most five from their chromatic symmetric function. We show a reconstruction that relies on two main ideas: (1) the orders of the leaf components in the \emph{internal subgraph} of the tree, which we define below, and (2) the \emph{adjacencies} between the leaf components of the tree. 

In \cite{martin2008distinguishing}, Martin, Morin, and Wagner proved that one can compute the diameter of a tree from its chromatic symmetric function.  Aliste-Prieto, de Mier, Orellana, and Zamora proved in \cite{ADOZ} that \emph{proper} trees of diameter at most five are distinguishable from their chromatic symmetric function using two new graph polynomials. This section improves their result to all trees of diameter at most five, providing a reconstruction algorithm directly from $\mathbf{X}_T$.

Trees of diameter $\leq 2$ are stars and there is only one such tree for a given number of vertices $k$, namely $St_k$. If $T$ has diameter 3, then $T$ is a bi-star with $\lead(\mathbf{X}_T) = (i,j)$, where $i$ and $j$ are the orders of its two leaf components, see Corollary \ref{cor:bi-stars-distinguish}. Therefore, in the remainder of this section we focus on trees with diameter 4 and 5.

%%%%%%%%%%%%%%%%%%%%%%%%%%%%%%%%%%%%%%%%%%%%%%%%%%%%%%%%%%%%%%
\subsection{The internal subgraph}

We begin this section by introducing a special subgraph of a tree using the concept of an internal edge. Recall that a leaf component of a tree $T$ is a connected component of $T \setminus I(T)$. Define the \defn{internal degree} of a vertex $v$ as the number of internal vertices in $N(v)$.

\begin{definition}
    Let $T$ be a tree and let $\{v_1, \ldots, v_l\}$ be the set of vertices of $T$ whose internal degree is strictly greater than 1. Let $L_i$ be the set of leaf-vertices that are neighbours of $v_i$ for $1 \leq i \leq l$. Then, the \defn{internal subgraph} of $T$ is the subgraph of $T$ induced by the set of vertices $\{v_1, \ldots, v_l\} \cup L_1 \cup \cdots \cup L_l$. We denote the internal subgraph of $T$ by $\mathcal{I}_T$.
\end{definition}

\begin{example}
    For the tree on the left-hand side in Figure \ref{fig:internal-subgraph}, the internal subgraph is the tree on the right-hand side. In particular, the vertices whose internal degree is greater than one are labeled $v_1,v_2$ and $v_3$. Their internal degrees are 3, 3, and 5, respectively.
\end{example}

\begin{figure}[ht]
    \centering
    \begin{tikzpicture}[every node/.style={circle, fill=black, scale=0.6}, scale=0.85]
    
        \node (a2) at (2, 3) {};
        \node (a4) at (4, 3) {};
        \node (a5) at (5, 2.75) {};
        \node (a52) at (5.5, 2.5) {};
        \node (a53) at (5.75, 2) {};
    
        \node (b0) at (0.4, 1.5) {};
        \node (b2) at (2,2) {};
        \node[fill=red] (b3) at (2.6, 1.75) {};
        \node[fill=red] (b32) at (3.4, 1.75) {};
        \node (b4) at (4, 2) {};
        \node (b5) at (4.75, 1.75) {};
        
        \node (c0) at (0.25,1) {};
        \node (c1) at (1,1) {};
        \node[fill=red] (c2) at (2,1) {};
        \node[fill=none, scale=1.5] at (1.75, 1.25) {$v_1$};
        \node[fill=red] (c3) at (3,1) {};
        \node[fill=none, scale=1.5] at (3.25, .75) {$v_2$};
        \node[fill=red] (c4) at (4,1) {};
        \node[fill=none, scale=1.5] at (3.75, 1.25) {$v_3$};
        \node (c5) at (5,1) {};
        \node (c6) at (6,1) {};
        
        \node (d0) at (0.4, 0.5) {};
        \node[fill=red] (d2) at (1.6, 0.4) {};
        \node[fill=red] (d22) at (2, 0.25) {};
        \node[fill=red] (d23) at (2.4, 0.4) {};
        \node (d3) at (3, 0) {};
        \node[fill=red] (d4) at (3.6, 0.25) {};
        \node (d5) at (4.75, 0.25) {};
        
        \node (e3) at (2.6, -0.75) {};
        \node (e32) at (3.4, -0.75) {};
        \node (e5) at (5.5, -0.5) {};
        
        \draw[thick] (c1) -- (c2);
        \draw[color=red, thick] (c3) -- (b3);
        \draw[color=red, thick] (c3) -- (b32);
        \draw[color=red, thick] (c2) -- (c3);
        \draw[color=red, thick] (c3) -- (c4);
        \draw[thick] (c4) -- (c5);
        \draw[thick] (c5) -- (c6);
        
        \draw[thick] (c2) -- (b2);
        \draw[thick] (b2) -- (a2);
        \draw[thick] (c4) -- (b4);
        \draw[thick] (b4) -- (a4);
        \draw[thick] (c0) -- (c1);
        \draw[thick] (c1) -- (b0);
        \draw[thick] (c1) -- (d0);
        \draw[color=red, thick] (c2) -- (d2);
        \draw[color=red, thick] (c2) -- (d22);
        \draw[color=red, thick] (c2) -- (d23);
        \draw[thick] (c3) -- (d3);
        \draw[thick] (d3) -- (e3);
        \draw[thick] (d3) -- (e32);
        \draw[color=red, thick] (c4) -- (d4);
        \draw[thick] (c4) -- (b5);
        \draw[thick] (c4) -- (d5);
        \draw[thick] (d5) -- (e5);
        \draw[thick] (b5) -- (a5);
        \draw[thick] (b5) -- (a52);
        \draw[thick] (b5) -- (a53);

    \node[fill=red] (bb3) at (10.6, 1.75) {};
    \node[fill=red] (bb32) at (11.4, 1.75) {};
    \node[fill=red] (cc2) at (10,1) {};
        \node[fill=none, scale=1.5] at (9.75, 1.25) {$v_1$};
        \node[fill=red] (cc3) at (11,1) {};
        \node[fill=none, scale=1.5] at (11.25, .75) {$v_2$};
        \node[fill=red] (cc4) at (12,1) {};
        \node[fill=none, scale=1.5] at (11.75, 1.25) {$v_3$};
        \node[fill=red] (dd2) at (9.6, 0.4) {};
        \node[fill=red] (dd22) at (10, 0.25) {};
        \node[fill=red] (dd23) at (10.4, 0.4) {};
        \node[fill=red] (dd4) at (11.6, 0.25) {};
        \draw[color=red, thick] (cc3) -- (bb3);
        \draw[color=red, thick] (cc3) -- (bb32);
        \draw[color=red, thick] (cc2) -- (cc3);
        \draw[color=red, thick] (cc3) -- (cc4);
        \draw[color=red, thick] (cc2) -- (dd2);
        \draw[color=red, thick] (cc2) -- (dd22);
        \draw[color=red, thick] (cc2) -- (dd23);
        \draw[color=red, thick] (cc4) -- (dd4);
        
    \end{tikzpicture}
    \caption{A tree $T$ on the left and its internal subgraph, $\mathcal{I}_T$, on the right.}
    \label{fig:internal-subgraph}
\end{figure}
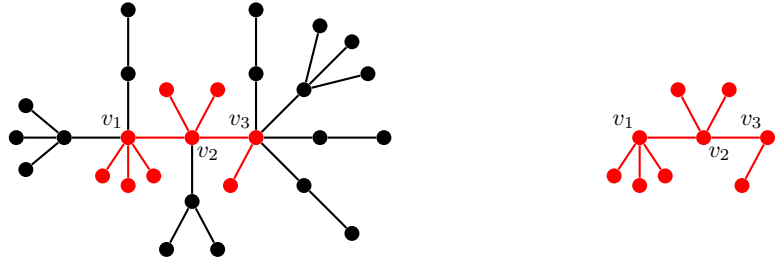

The internal subgraph as defined above plays a crucial role in our reconstruction of trees of small diameter. We now show some basic properties and results about the internal subgraph that we need for our reconstruction. 

\begin{proposition}\label{connectedness internal subgraph}
    The internal subgraph of a tree $T$ is a tree.
\end{proposition}
\begin{proof} The internal subgraph does not contain cycles, because it is a subgraph of a tree. Thus, it suffices to show that $\mathcal{I}_T$, is connected. It then suffices to show that, if $\{v_1, \ldots, v_l\}$ is the set of vertices of $T$ with internal degree strictly greater than 1, then there is a path from $v_i$ to $v_j$ in $\mathcal{I}_T$ for all $1 \leq i,j \leq l$ with $i \neq j$. Since $T$ itself is connected, then there is a path $v_i, u_1, \ldots, u_x, v_j$ from $v_i$ to $v_j$, for some vertices $u_1, \ldots, u_x \in V(T)$. Observe that $u_1, \ldots, u_x$ have two neighbors in the path, so they are not leaves. Therefore, $u_1, \ldots, u_x$ are internal vertices that have internal degree at least two, so $u_1, \ldots, u_x$ are vertices of the internal subgraph. Hence, the path $v_i, u_1, \ldots, u_x, v_j$ is contained in $\mathcal{I}_T$, which implies that $\mathcal{I}_T$ is connected.
\end{proof}

The following lemma is a trivial application of the fact that any longest path in a tree contains two leaves, but we write it here to reference it in future proofs.

\begin{lemma} \label{diam-int}
    If $T$ is a tree with diameter  $d \geq 2$, then any path $P$ in $T$ contains at most $d-2$ internal edges. Furthermore, $T$ must contain at least one path with $d-2$ internal edges. \qed
\end{lemma}

From Lemma \ref{diam-int}, it follows that, if $T$ is a tree with diameter less than or equal to three, then the internal subgraph of $T$ is empty. 

\begin{proposition}\label{prop:numberinternaledges-in-IntSub} Let $d$ be an integer greater than or equal to 4.  If $T$ is a tree with diameter $d$, then any path in $\mathcal{I}_T$ has at most $d-4$ internal edges.
\end{proposition}
\begin{proof}
    % Let $\ell_1, v_2, \ldots, v_{d}, \ell_{d+1}$ be a path of length $d$ in $T$. From basic graph theory, we know that $\ell_1$ and $\ell_{d+1}$ are the only leaves in the path, so $v_2, \ldots, v_d$ are internal vertices. We claim that $v_2$ and $v_d$ have internal degree 1 in $T$. Assume for contradiction that $v_2$ or $v_d$ have internal degree at least two. Without loss of generality, suppose that $v_2$ has internal degree at least two. Since there are no cycles in $T$, there exists an internal vertex $v \in N(v_2)$ that is not in the path $\ell_1, v_2, \ldots, v_d, \ell_{d+1}$. Then, $v, v_2, \ldots, v_d$ is a path in $T$ containing $d-1$ internal edges, contradicting Lemma \ref{diam-int}.  Hence, any path of length $d$ in $T$ contains exactly two leaves, two vertices of internal degree 1 and $d-3$ vertices of internal degree greater than 1. This implies that the maximum distance between two vertices of internal degree $\geq 2$ in $T$ is $d-4$.
     
    Choose a path $P= v_0,\ldots,v_k$ in $\mathcal I_T$ with maximal number of internal edges. If either $v_0$ or $v_k$ is a leaf of $T$ then we may remove them from $P$ without changing the number of internal edges, so we may assume that neither $v_0$ nor $v_k$ is a leaf of $T$. Thus, $P$ contains $k$ internal edges.
    
    Further, $v_0$ and $v_k$ are adjacent to internal vertices $u$ and $w$ of $T$, respectively, and $u$ and $w$ are adjacent to leaves $u'$ and $w'$ of $T$, because if $u$ and $w$ were deep vertices, then we could find a path in $\mathcal{I}_T$ with more internal edges than $P$. Clearly $u$, $w$, $u'$ and $w'$ are pairwise distinct and none of them belong to $P$. So $u'$, $u$, $v_0$, \ldots, $v_k$, $w$, $w'$ is a path in $T$ of length $k+4$. Hence $k+4 \leq d$ and so $k\leq d-4$, as required.
\end{proof}

\begin{proposition}\label{prop:int-edge-incident-I_T} Let $d$ be an integer greater than 4. If $T$ is a tree with diameter $d$, then every internal edge in $T$ is incident to a vertex in $\mathcal{I}_T$.
\end{proposition}
\begin{proof}
    Let $e$ be an internal edge in $T$ and arguing by contraction, assume that $e$ is not incident to a vertex in $\mathcal{I}_T$. This implies that if $e=uv$, then both $u$ and $v$ have internal degree exactly 1. Then $N(u)\setminus \{v\}$ contains only leaf-vertices of internal degree 0, and similarly for $N(v)\setminus \{u\}$. Since $T$ is connected, this implies that $T$ has only one internal edge and hence it is a bi-star. Hence, $T$ has diameter 3, which contradicts that $T$ has diameter $d \geq 4$.
\end{proof}

\begin{corollary}\label{cor:leaf-comp-number-diam4-diam5}
Let $T$ be a tree. Then:
    \begin{enumerate}
         \item[(a)] If $T$ has diameter 4, then $\mathcal{I}_T$ is a star, or equivalently one leaf component of $T$.
        \item[(b)] If $T$ has diameter 5, then $\mathcal{I}_T$ is a star or a bi-star, or equivalently two leaf components of $T$ connected by an edge. 
    \end{enumerate}
\end{corollary}
\begin{proof}
     (a) By Proposition \ref{prop:numberinternaledges-in-IntSub}, any path in $\mathcal{I}_T$ has at most $0$ internal edges since $T$ has diameter 4. Then, $\mathcal{I}_T$ consists of an internal vertex $u$ and the set of leaf-vertices in $T$ incident to $u$. Hence, it follows that $\mathcal{I}_T$ is a leaf component of $T$, or in other words, a star graph. 

     (b) Similarly, Proposition \ref{prop:numberinternaledges-in-IntSub} tells us that any path in $\mathcal{I}_T$ has at most $1$ internal edge, since $T$ has diameter 5.  Then, $\mathcal{I}_T$ consists of two adjacent internal vertices $u$ and $v$ and the set of leaf-vertices adjacent to $u$, $L_u$, and those adjacent to $v$, $L_v$.  In the case, that both $L_u$ and $L_v$ are empty, then $\mathcal{I}_T$ is one edge $uv$ (the star of order 2), if only one of $L_u$ or $L_v$ is empty, then $\mathcal{I}_T$ is a star and if none are empty, then $\mathcal{I}_T$ is a bi-star. 
\end{proof}

\begin{lemma}\label{lem:deep-vertex-in-internal-subgraph}
    Let $T$ be a tree with internal subgraph $\mathcal{I}_T$. If $T$ has a deep vertex $v$, then $v$ is in $\mathcal{I}_T$ and so the leaf components of order 1 in $T$ are contained in $\mathcal{I}_T$.
\end{lemma}
\begin{proof}
    A deep vertex, $v$, in $T$ is by definition an internal vertex without any leaf-vertices in its neighborhood. Since an internal vertex has degree at least 2, this implies that all  neighbors of $v$ have degree at least 2 and there are at least two such neighbors.  Hence, $v$ has internal degree at least 2, so $v \in \mathcal{I}_T$.
\end{proof}

\begin{proposition}\label{prop:num-deep-vert-diam-at-most-5}
    Let $T$ be a tree. Then:
    \begin{enumerate}
        \item[(a)] If $T$ has diameter at most 3, then $T$ has no deep vertices.
        \item[(b)] If $T$ has diameter 4, then $T$ contains at most 1 deep vertex.
        \item[(c)] If $T$ has diameter 5, then $T$ contains at most 2 deep vertices.
    \end{enumerate}
\end{proposition}
\begin{proof}
    For (a), recall that Lemma \ref{diam-int} implies that the internal subgraph of a tree of diameter less than or equal to 3 is empty. Hence, Lemma \ref{lem:deep-vertex-in-internal-subgraph} implies that there cannot be deep vertices in a tree of diameter at most 3. Similarly, (b) and (c) immediately follow from Corollary \ref{cor:leaf-comp-number-diam4-diam5} and Lemma \ref{diam-int}.
\end{proof}

\begin{corollary}\label{cor:1s-leading-diam45}
    Let $\csft$ be the chromatic symmetric function of a tree $T$ and let $\lead(\csft)=(n^{m_n},\ldots,1^{m_1})$ be its leading partition. Then:
    \begin{enumerate}
        \item[(a)] If $T$ has diameter $1 \leq \diam(T) \leq 3$, then $m_1=0$. 
        \item[(b)] If $T$ has diameter 4, then $m_1=0$ or $1$.
        \item[(c)] If $T$ has diameter 5, then $m_1=0,1$ or $2$.
    \end{enumerate}
\end{corollary}
\begin{proof}
    This is an immediate consequence of Corollary \ref{cor:num-deep-vertices} and Proposition \ref{prop:num-deep-vert-diam-at-most-5}.
\end{proof}

%%%%%%%%%%%%%%%%%%%%%%%%%%%%%%%%%%%%%%%%%%%%%%%%%%%%%%%%%%%%%%
\subsection{Leaf components in the internal subgraph} From the leading partition of a tree, $T$, we can recover the orders of the leaf components of $T$. When $T$ has diameter 4, one of the leaf components of $T$ is in $\mathcal{I}_T$ and if $T$ has diameter 5, then two of the leaf components of $T$ will be in $\mathcal{I}_T$.  We now show that in the cases that $T$ is proper or has diameter at most 5, we will be able to recover the orders of the leaf components in $T$ that are also in $\mathcal{I}_T$ from $\mathbf{X}_T$.

\begin{definition}\label{def:lead-component-adjacency}
    Let $T$ be a tree and let $\mathcal{L}_1$ and $\mathcal{L}_2$ be two leaf components with central vertices $v_1$ and $v_2$, respectively. We say that $\mathcal{L}_1$ and $\mathcal{L}_2$ are \defn{adjacent} if $v_1v_2 \in E(T)$. In addition, we will refer to $\mathcal{L}_1$ and $\mathcal{L}_2$ as the \defn{leaf component endpoints} for the internal edge $e=v_1v_2$.
\end{definition}

\begin{example}
    In the tree below, the leaf components $\mathcal{L}_1$ (with center $v_1$) and  $\mathcal{L}_2$ (with center $v_2$), are adjacent and they are the leaf component endpoints of the edge $v_1v_2$.
    
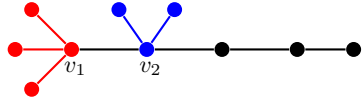
\begin{center}
        \begin{tikzpicture}[every node/.style={circle, fill=black, scale=0.6}, thick]
            \node[fill=red] (a) at (0,0) {};
            \node[fill=red] (a1) at (-.53, -.53) {};
            \node[fill=red] (a2) at (-.75, 0) {};
            \node[fill=red] (a3) at (-.53, .53) {};

            \node[scale=1.5, fill=none] at (0.05, -.25) {$v_1$};
            
            \node[fill=blue] (b) at (1,0) {};
            \node[fill=blue] (b1) at (0.63, 0.53) {};
            \node[fill=blue] (b2) at (1.37, 0.53) {};

            \node[scale=1.5, fill=none] at (1.05, -.25) {$v_2$};

            \node (c) at (2,0) {};
            
            \node (d) at (3,0) {};
            \node (d1) at (3.75, 0) {};

            \draw[red] (a) -- (a1);
            \draw[red] (a) -- (a2);
            \draw[red] (a) -- (a3);

            \draw[blue] (b) -- (b1);
            \draw[blue] (b) -- (b2);

            \draw(d)--(d1);

            \draw (a) -- (b);
            \draw(b) -- (c);
            \draw(c)--(d);
        \end{tikzpicture}
    \end{center}
    
\end{example}

Given a multiset $A$, we denote by $m_A(a)$ the multiplicity of $a$ in $A$. Further, if we have another multiset $B$, then their difference $A-B$ is the multiset where each $a \in A$ has multiplicity $\max(0, m_A(a) - m_B(a))$. For a multiset $A$, we denote by $|A|$ the number of elements (counted with multiplicity) in $A$, or equivalently the sum of the multiplicities of all the elements. We denote multisets using double-brackets, for instance $A = \lmulti 3,3,5 \rmulti$.

\begin{definition}
    Let $\csft=\sum_{\lambda \vdash n}c_\lambda \mathfrak{st}_\lambda$ be the chromatic symmetric function of an $n$-vertex tree. For any partition $\mu \vdash n$, let $A_\mu$ denote the multiset containing the parts of $\mu$ with multiplicity given by the number of occurrences of that part in $\mu$. Further, for any $\mu \vdash n$ such that $c_\mu \neq 0$ in $\csft$, define the \defn{adjacency multiset} $E_\mu \coloneqq A_{\lead(\csft)}-A_\mu$. If $\ell(\mu) = \ell(\lead(\csft))-k$ where $k \geq 1$, we call $E_\mu$ a \defn{k-edge adjacency multiset}.
\end{definition}

In the following proposition, we interpret the definition of $E_\mu$ in terms of edge adjacencies of leaf components of a tree with chromatic symmetric function $\csft$.  For example, in the case of proper trees, we can recover all edge adjacencies between leaf components from the 1-edge adjacency multisets, $E_\mu$. Proposition \ref{prop:k-edge-adjacencies} contains all the results needed about $E_\mu$ for the reconstructions of trees of diameter at most 5. 
Recall that in Theorem \ref{thm:dnc-tree-csf} we used the notation $\lambda(F)$ to denote the partition whose parts are the orders of the connected components in a forest $F$.

\begin{proposition}\label{prop:k-edge-adjacencies}
    Assume that $\csft$ is the chromatic symmetric function of an $n$-vertex tree $T$ with leading partition $\lead(\csft)=(n^{m_n},\ldots, 1^{m_1})$. Let $\mu \vdash n$ such that $c_\mu \neq 0$ in $\csft$ and such that $\mu$ contains no 1s. Then:
    \begin{enumerate}
        \item[(a)] If $m_1=0$ and $\ell(\mu) = \ell(\lead(\csft))-1$, then $E_\mu = \lmulti p,q \rmulti$, where $p$ and $q$ are orders of two adjacent leaf components in $T$. Further, $c_\mu$ is the number of internal edges with leaf component endpoints of order $p$ and $q$.
        \item[(b)] If $m_1=1$ and $\ell(\mu) = \ell(\lead(\csft))-1$, then $E_\mu = \lmulti 1, q \rmulti$, where $q$ is the order of a leaf component adjacent to the deep vertex. Then, $c_\mu$ is the number of leaf components of order $q$ adjacent to the deep vertex.
    \end{enumerate}
\end{proposition}

\begin{proof}
    We begin by noting that since $\mu$ does not contain 1s, then in any DNC-tree for $T$, any sequence of $L,M$, and $R$s encoding a path from $T$ to a star forest $F$ such that $\lambda(F)=\mu$ contains no $M$s (dot-contractions). 

    To prove (a), note that since $\lead(\csft)$ contains no 1s, then $T$ must be a proper tree by Corollary \ref{cor:prop-trees-deep-vertices}. For proper forests, a deletion always results in a forest with one fewer internal edge. Hence, all paths in a DNC-tree from the root $T$ to a leaf $F$ have length $\#I(T)$, the number of internal edges.  If we want $\lambda(F)$ to have no 1s and length $\ell(\lead(\csft))-1$, then $F$ corresponds a path in the DNC-tree with exactly one $R$ (leaf-contraction) and the rest $L$s (deletions). If the internal edge that is leaf-contracted has leaf component endpoints of order $p$ and $q$, then $\lambda(F) = \mu = \sort(p+q, n^{m_n},\ldots, p^{m_p-1}, \ldots, q^{m_q-1}, \ldots, 2^{m_2})$ and so $E_\mu=\lmulti p,q\rmulti$. We get this $\mu$ for every internal edge connecting two leaf components of orders $p$ and $q$. Hence, the result follows.
    
    For (b), observe that Corollary \ref{cor:num-deep-vertices} implies that $T$ has exactly one deep vertex $v$. Then, in any DNC-tree for $T$, if deletions are applied to $\deg(v)-1$ of the internal edges incident to $v$, then after the last of these deletions $v$ is no longer an internal vertex in the resulting forest, it is a leaf in the leaf component of one of its neighbors $u \in N(v)$. This implies that a star forest $F$ with $\lambda(F)$ of length $\ell(\lead(\csft))-1$ and without 1s can correspond to a sequence of only deletions, $L$s.   Note that in this case we must have $\mu=\lambda(F) = (q+1, n^{m_n},\ldots, q^{m_q-1}, \ldots, 2^{m_2})$ where $q$ is the order of the leaf component containing $u$. Hence, $E_\mu = \lmulti 1,q \rmulti$. The other way to obtain a star forest $F$ with $\ell(\lambda(F))=\ell(\lead(\csft))-1$ and such that $\lambda(F)$ contains no ones is by using paths that correspond to a sequence of one $R$ and $\#I(T)-1=\ell(\lead(\csft))-2$ $L$s, where $R$ cannot be the last entry in the sequence and where $R$ is applied to an internal edge incident to the deep vertex $v$. It is then easy to see that we also obtain $E_\mu = \lmulti 1,q\rmulti$ where $q$ is the order of the leaf component containing the endpoint distinct to $v$ of the internal edge that was leaf-contracted. Now let $Q$ be the number of leaf components of order $q$ that are adjacent to the deep vertex $v$. From the argument above, the paths in a DNC-tree for $T$ that end in a star forest, $F$, with $\lambda(F)=\mu$ are $(L,\ldots,L)$ where there are $\#I(T)-1$ $Ls$ and $(L, \ldots, L, R, L, \ldots, L)$ where there are $\#I(T)-1$ $L$s and $R$ is in the $i$-th position for each $1 \leq i \leq Q-1$. Hence, $c_\mu$ is precisely the number of leaf components of order $q$ that are adjacent to the deep vertex $v$.
\end{proof}

For examples of edge adjacency multisets and Proposition \ref{prop:k-edge-adjacencies}, we refer the reader to Examples \ref{ex:proper-distinct-parts} and \ref{ex:diam-4-reconstruction}. The following corollary to the proposition above allows us to reconstruct proper trees whose leading partitions have distinct parts.

\begin{corollary}\label{cor:proper-tree-distinct-parts}
    Let $\csft$ be the chromatic symmetric function of a tree, $T$, such that $\lead=\lead(\csft) = (\lambda_1, \ldots, \lambda_\ell)$ contains no 1s and has all distinct parts. Then, $T$ can be reconstructed from $\csft$. In particular, $T$ can be reconstructed from the 1-edge adjacency multisets.
\end{corollary}
\begin{proof}
    Since $\lead(\csft)$ contains no 1s, then Proposition \ref{prop:k-edge-adjacencies}(a) implies that $\bigsqcup_\mu E_\mu$, where the union is taken over all $\mu \vdash n$ of length $\ell(\lead(\csft))-1$ such that $c_\mu\neq 0$ in $\csft$ and $\mu$ contains no 1s, contains all the adjacencies between leaf components in $T$. Further, each of these adjacencies occurs exactly once in $T$ since all parts in $\lead(\csft)$ have multiplicity 1. Then, we can reconstruct $T$ as follows: (1) draw a leaf component $\mathcal{L}_i$ of order $\lambda_i$ for each $1 \leq i \leq \ell$, and (2) for $1 \leq i < j \leq \ell$, if $\lmulti \lambda_i, \lambda_j\rmulti$ is a 1-edge adjacency multiset, draw an edge connecting the centers of $\mathcal{L}_i$ and $\mathcal{L}_j$.
\end{proof}

\begin{example}\label{ex:proper-distinct-parts}
    We provide an example of how to reconstruct a proper tree $T$ whose leading partition has all distinct parts from $\csft$. For the sake of brevity, we only provide $\lead(\csft)$ and the coefficients of the partitions $\mu$ without 1s such that $\ell(\mu)=\ell(\lead(\csft))-1$. Consider $\csft$ with $\lead(\csft) = (9,7,6,5,4,3,2)$ and with the following coefficients indexed by partitions without 1s with $\ell(\lead(\csft))-1$:  
    \begin{center}
    \begin{tabular}{||c |c |c||} 
         \hline
         $\mu$ & $c_\mu$ & $E_\mu$ \\ [0.5ex] 
         \hline\hline
         $(16,6,5,4,3,2)$ & $1$ & $\lmulti 9,7 \rmulti$ \\ 
         \hline
         $(15,7,5,4,3,2)$ & $1$ & $\lmulti 9,6 \rmulti$ \\
         \hline
         $(11,9,7,4,3,2)$ & $1$ & $\lmulti 6,5 \rmulti$ \\
         \hline
         $(10,9,7,5,3,2)$ & $1$ & $\lmulti 6,4\rmulti$ \\
         \hline
         $(9,7,7,6,5,2)$ & $1$ &  $\lmulti 4,3 \rmulti$\\ 
         \hline
         $(9,7,6,6,5,3)$ & $1$ & $\lmulti 4,2 \rmulti$ \\
         \hline
    \end{tabular}
    \end{center}
    Therefore, following the algorithm outlined in Corollary \ref{cor:proper-tree-distinct-parts}, we reconstruct $T$ and obtain:

    
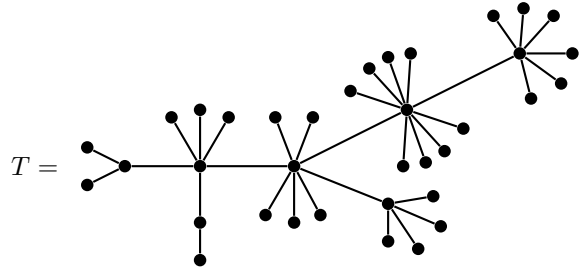
\begin{figure}[ht]
        \centering
        \begin{tikzpicture}[auto=center,thick, every node/.style={circle, fill=black, scale=0.5}]
            \node (c) at (0,0) {};
            \node (c1) at (-.38, 0.65) {};
            \node (c2) at (0, .75) {};
            \node (c3) at (.38, .65) {};

            \node (t) at (-1, 0) {};
            \node (t1) at (-1.5, 0.25) {};
            \node (t2) at (-1.5, -.25) {};

            \node (d) at (0, -.75) {};
            \node (d1) at (0, -1.25) {};
            % \node (d1) at (-0.6,-1.1) {};

            \node (s) at (1.25,0) {};
            \node (s1) at (1, .65) {};
            \node (s2) at (1.5, .65) {};
            \node (s3) at (.87, -.65) {};
            \node (s4) at (1.6, -.65) {};
            \node (s5) at (1.25, -.75) {};

            \node (n) at (2.75, .75) {};
            \node (n1) at (2.8, 1.5) {};
            \node (n2) at (2.7, 0) {};
            \node (n3) at (2.5, 1.45){};
            \node (n4) at (2.25, 1.3) {};
            \node (n5) at (2, 1) {};
            \node (n6) at (3, 0.05) {};
            \node (n7) at (3.25, 0.2) {};
            \node (n8) at (3.5, 0.5){};

            \node (l) at (4.25, 1.5) {};
            \node (l1) at (4.8, 1.1) {};
            \node (l2) at (4.3, 2.1) {};
            \node (l3) at (4.7, 2) {};
            \node (l4) at (4.95, 1.5) {};
            \node (l5) at (3.9, 2) {};
            \node (l6) at (4.4, .9) {};

            \node (f) at (2.5, -.5) {};
            \node (f1) at (3.2, -.8) {};
            \node (f2) at (2.9, -1.1) {};
            \node (f3) at (3.1, -.4) {};
            \node (f4) at (2.5, -1) {};

            \draw (l1)--(l)--(l2);
            \draw (l3)--(l)--(l4);
            \draw (l5)--(l)--(l6);
            \draw (n) -- (l);
            \draw (c1)--(c)--(c2);
            \draw (c)--(c3);
            \draw (t1)--(t)--(t2);
            \draw (d)--(d1);
            \draw (t)--(c)--(d);
            \draw (s1)--(s)--(s2);
            \draw (s3)--(s)--(s4);
            \draw (s)--(s5);
            \draw (c)--(s);
            \draw (f1)--(f)--(f2);
            \draw (f3)--(f)--(f4);
            \draw (f) -- (s);
            \draw (s) -- (n);
            \draw (n1)--(n)--(n2);
            \draw (n3)--(n)--(n4);
            \draw (n6)--(n)--(n5);
            \draw (n7)--(n)--(n8);
            \node[fill=none, scale=2] at (-2.2,0) {$T=$};
        \end{tikzpicture}
    \end{figure}
\end{example}
We now define a number that will help us determine the leaf components of $T$ which are contained in $\mathcal{I}_T$. If $\csft=\sum_\lambda c_\lambda \stfrak_\lambda$ is the chromatic symmetric function of an $n$-vertex tree, $T$, and $p$ is a part in $\lead(\csft)$, then define the quantity:
\begin{equation}\label{eq:N(p)}
    N(p) := \sum m_{E_\lambda}(p) \cdot c_\lambda ~,
\end{equation}   
where the sum runs over all $\lambda\vdash n$ of length $\ell(\lead(\csft))-1$ such that $c_\lambda \neq 0$ in $\csft$ and such that $\lambda$ does not contain 1 as a part. Recall $m_{E_\lambda}(p)$ is the multiplicity of $p$ in $E_\lambda$.

From Equation (\refeq{eq:N(p)}), we observe that $N(p)$ is a quantity that can be recovered from the chromatic symmetric function. In the following remark, we give the combinatorial interpretation for $N(p)$, which reveals information about $T$ when $\lead(\csft)$ contains no 1s.

\begin{remark}\label{remark:n(p)-interpretation}
    If $\lead(\csft)$ contains no 1s (i.e. if $T$ is a proper tree), then Proposition \ref{prop:k-edge-adjacencies}(a) implies that $\bigsqcup_\mu E_\mu$, where the union is taken over all $\mu \vdash n$ of length $\ell(\lead(\csft))-1$ such that $c_\mu\neq 0$ in $\csft$ and $\mu$ contains no 1s, contains the adjacencies between leaf components in $T$. Further, there are $c_\mu$ such adjacencies for each $E_\mu$ in the disjoint union. Hence, $N(p)$ as defined above is exactly the number of times that a leaf component of order $p$ occurs as a leaf component endpoint in $T$. For instance, in Example \ref{ex:diam-4-reconstruction}, the leaf component of order 4 occurs four times as an endpoint in the tree $T$. Notice that there is only one leaf component of order $4$ in the tree in that example but there are four internal edges incident to it and exactly one endpoint of each of these internal edges is contained in a leaf component of order $4$.
\end{remark}

Even though $N(p)$ is defined as a quantity obtained from $\csft$, the remark above shows that $N(p)$ can also be computed from the tree. The following example uses the tree to compute the values of $N(p)$.

\begin{example}
    Consider the tree $T$ below, where the leaf components contained in the internal subgraph have central vertices $v_1$, $v_2$ and $v_3$. We have $\lead(\csft) = (6,5,4^2,3^2,2^4)$. From the leaf component adjacencies in the tree, we obtain $N(6) = 1, N(5)=1, N(4)=4, N(3)=4$ and $N(2)=8$. One can check that these values agree with those obtained when computing $N(p)$ from $\csft$. We omit this computation here for brevity.
    
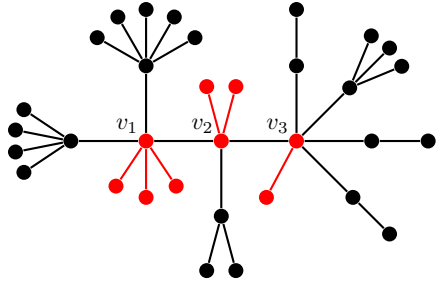
\begin{center}
        \begin{tikzpicture}[every node/.style={circle, fill=black, scale=0.6}, thick]
            \node (b2) at (2,2) {};
            \node (b21) at (1.35, 2.375) {};
            \node (b22) at (1.625, 2.65) {};
            \node (b23) at (2.375, 2.65) {};
            \node (b24) at (2.65, 2.375) {};
            \node (b25) at (2,2.75){};
    
            \draw (b2) -- (b21);
            \draw (b2) -- (b22);
            \draw (b2) -- (b23);
            \draw (b2) -- (b24);
            \draw (b2) -- (b25);
    
            \node (b4) at (4, 2) {};
            \node (b41) at (4, 2.75) {};
    
            \draw (b4)--(b41);
    
            \node (b5) at (4.707, 1.707) {};
            \node (b51) at (4.994, 2.399) {};
            \node (b52) at (5.237, 2.237) {};
            \node (b53) at (5.399, 1.994) {};
    
            \draw (b5)--(b51);
            \draw (b5)--(b52);
            \draw (b5)--(b53);
    
            \node (c1) at (1,1) {};
            \node (c11) at (0.376, 1.417) {};
            \node (c12) at (0.264, 1.146) {};
            \node (c13) at (0.264, 0.854) {};
            \node (c14) at (0.376, 0.583) {};
    
            \draw (c1)--(c11);
            \draw (c1)--(c12);
            \draw (c1)--(c13);
            \draw (c1)--(c14);
        
            \node[red] (c2) at (2,1) {};
            \node[fill=red] (c21) at (1.6, 0.4) {};
            \node[fill=red] (c22) at (2, 0.25) {};
            \node[fill=red] (c23) at (2.4, 0.4) {};
    
            \draw[red] (c2)--(c21);
            \draw[red] (c2)--(c22);
            \draw[red] (c2)--(c23);
            
            \node[red] (c3) at (3,1) {};
            \node[red] (c31) at (2.806, 1.724) {};
            \node[red] (c32) at (3.194, 1.724) {};
    
            \draw[red] (c3)--(c31);
            \draw[red] (c3)--(c32);
            
            \node[red] (c4) at (4,1) {};
            \node[red] (c41) at (3.6, 0.25) {};
    
            \draw[red] (c4)--(c41);
    
            \node (c5) at (5,1) {};
            \node (c51) at (5.75,1) {};
    
            \draw (c5)--(c51);
    
            \node (d3) at (3, 0) {};
            \node (d31) at (2.806, -0.724) {};
            \node (d32) at (3.194, -0.724) {};
    
            \draw (d3)--(d31);
            \draw (d3)--(d32);
            
            \node (d5) at (4.75, 0.25) {};
            \node (d51) at (5.237, -0.237) {};
    
            \draw (d5)--(d51);
            
            \draw (c4)--(b5);
            \draw (c4)--(c5);
            \draw (c4)--(d5);
            \draw (c4)--(b4);
            \draw (c1)--(c2);
            \draw (c2)--(c3);
            \draw (c3)--(c4);
            \draw (c2)--(b2);
            \draw (c3)--(d3);
            \node[fill=none, scale=1.5] at (1.75, 1.2) {$v_1$};
            \node[fill=none, scale=1.5] at (2.75, 1.2) {$v_2$};
            \node[fill=none, scale=1.5] at (3.75, 1.2) {$v_3$};
        \end{tikzpicture}
    \end{center}
    Note that the orders of the leaf components in $\mathcal{I}_T$ are $4, 3,$ and $2$. These are the only parts in $\lead(\csft)$ such that $N(p) > m_p$, where $m_p$ denotes the multiplicity of $p$ in $\lead(\csft)$. The following theorem guarantees that this is always the case for proper trees.
\end{example}

\begin{theorem}\label{thm:leaf-component-order}
    Let $\csft$ be the chromatic symmetric function of an $n$-vertex tree, $T$ and let $\lead(\csft) = (n^{m_n}, \ldots, 2^{m_2})$ be the leading partition. If $p$ is any part of $\lead(\csft)$, then a leaf component of order $p$ in $T$ is  contained in $\mathcal{I}_T$ if and only if $N(p) > m_p$.
\end{theorem}

\begin{proof}
    Suppose that a leaf component $\mathcal{L}$ contained in $\mathcal{I}_T$ has order $p$. Let $v$ be the central vertex of $\mathcal{L}$. Since $\mathcal{L} \subseteq \mathcal{I}_T$, then there are at least two internal edges $e, e^\prime$ incident to $v$ in $T$. Then, we know that $\mathcal{L}$, which has order $p$, occurs at least twice as a leaf component endpoint in $T$ from the edges $e$ and $e^\prime$. Let $\mathcal{L}_1, \ldots, \mathcal{L}_{m_p-1}$ be the remaining leaf components of order $p$ and let $v_1, \ldots, v_{m_p-1}$ be their central vertices, respectively. Since $v_1, \ldots, v_{m_p-1}$ are internal vertices, there is at least one internal edge incident to each of them. Let $e_1, \ldots, e_{m_p-1}$ be internal edges incident to $v_1, \ldots, v_{m_p-1}$, respectively. Note that it is possible that some of these edges are equal, in the case that an edge has as endpoints the centers of two leaf components both of order $p$. Then, we know that each $\mathcal{L}_i$, which has order $p$ is a distinct leaf component endpoint. Hence, $N(p) \geq 2+m_p-1=m_p+1>m_p$.

    We show the converse by contrapositive. 
    Suppose that no leaf component of $\mathcal{I}_T$ has order $p$. If $\diam(T) \leq 3$, then $\mathcal{I}_T$ is empty. In this case, each leaf component of order $p$ is a leaf component endpoint of at most one internal edge, and hence $N(p)\leq m_p$. In particular, $N(p)=0$ if $\diam(T)=1$ or $2$ and $N(p)=m_p$ if $\diam(T)=3$. Thus assume that $\diam(T) \geq 4$ and so $\mathcal{I}_T$ is not empty. Then, by Proposition \ref{prop:int-edge-incident-I_T}, it follows that every leaf component of $T$ that has order $p$ is adjacent to exactly one other leaf component with order not equal to $p$. Let $e_1, \ldots, e_{m_p}$ be the internal edges that connect, respectively, the centers of the $m_p$ leaf components of order $p$ to the centers of the leaf components of orders $q_1, \ldots, q_{m_p}$. Note that $q_1, \ldots, q_{m_p}$ are not necessarily distinct. 
    Since $q_i \neq p$ for all $1 \leq i \leq m_p$, then a leaf component of order $p$ 
    occurs as an endpoint in $T$ exactly $m_p$ times. Thus, by Remark \ref{remark:n(p)-interpretation}, $N(p)=m_p\leq m_p$, finishing the proof.
\end{proof}

%%%%%%%%%%%%%%%%%%%%%%%%%%%%%%%%%%%%%%%%%%%%%%%%%%%%%%%%%%%%%%
\subsection{Reconstruction of trees of diameter 4.} \label{subsec:diameter 4}
In Martin, Morin and Wagner in \cite{martin2008distinguishing}, the authors show that the diameter of a tree can be recovered from the $\csft$. Given a tree $T$, we use their result to check that $T$ has diameter 4, then Theorem \ref{thm:leading-partition} which recovers the leading partition gives us the orders of all the leaf components of $T$; Corollary \ref{cor:leaf-comp-number-diam4-diam5} which says that $\mathcal{I}_T$ is a single leaf component; then Corollary \ref{cor:num-deep-vertices} tells us when  $\mathcal{I}_T$ is just a single vertex, and Theorem \ref{thm:leaf-component-order}, which helps us find the order of $\mathcal{I}_T$ in the case that the leading partition does not have a 1. Finally, Proposition \ref{prop:int-edge-incident-I_T} tells us that every other leaf component not in $\mathcal{I}_T$ is incident to the central vertex of the leaf component in $\mathcal{I}_T$.
\begin{theorem}\label{thm:diam-4-reconstruction}
    Trees of diameter four can be reconstructed from their chromatic symmetric function. In particular, these trees can be reconstructed from their leading partition, $\lead(\csft)$ and the coefficients of the partitions $\mu$ such that $\ell (\mu) = \ell(\lead(\csft))-1$ and $\mu$ has no parts of size 1 such that $c_\mu\neq 0$. 
\end{theorem}

 In fact, we can give an algorithm to reconstruct trees of diameter 4. Given $\csft$ in the star-basis, the leading partition $\lead(\csft)$ tells us the orders of the leaf-components of $T$. Since $T$ has diameter 4, it can have at most 1 deep vertex. This means that $\lead(\csft)$ can have at most one 1.  We need to determine the internal subtree $\mathcal{I}_T$, which is a single leaf component (a star) when $\diam(T) =4$. 
\begin{itemize}
    \item If $\lead(\csft)=(\lambda_1, \ldots, \lambda_\ell, 1)$ has one 1, then $\mathcal{I}_T$ is a single vertex, $v$, then $T$ is the tree obtained by adding an edge from $v$ the center of a star of order $\lambda_i$ for all $1\leq i \leq \ell$.
    \item If $\lead(\csft)=(\lambda_1, \ldots, \lambda_\ell)$ does not have a 1, then:
    \begin{itemize}
        \item Use Theorem \ref{thm:leaf-component-order} to determine $\mathcal{I}_T$, which in this case is a single leaf component of order $\lambda_j$ for some $1\leq j \leq \ell$. 
        \item Then add an edge from the central vertex of a star of order $\lambda_i$, for $i\neq j$, to the central vertex of $\mathcal{I}_T$.
    \end{itemize}
\end{itemize}

\begin{example}\label{ex:diam-4-reconstruction}
We now give an example applying the algorithm outlined above. Consider the following chromatic symmetric function on a tree with 17 vertices.
\begin{eqnarray*}
    \csft &=&  \stfrak_{(5, 4, 3, 3, 2)}- \stfrak_{(5, 5, 3, 3, 1)} - 2\stfrak_{(6, 5, 3, 2, 1)} + \stfrak_{(6, 5, 3, 3)} + 2\stfrak_{(7, 5, 3, 1, 1)} + 2\stfrak_{(7, 5, 3, 2)}  - \stfrak_{(8, 3, 3, 2, 1)} +
    \\
    && \stfrak_{(8, 5, 2, 1, 1)} - 4\stfrak_{(8, 5, 3, 1)} + \stfrak_{(9, 3, 3, 1, 1)} + \stfrak_{(9, 3, 3, 2)} - \stfrak_{(9, 5, 1, 1, 1)} - 2\stfrak_{(9, 5, 2, 1)} + 2\stfrak_{(9, 5, 3)}+
    \\
    && 2\stfrak_{(10, 3, 2, 1, 1)} - 2\stfrak_{(10, 3, 3, 1)} + 3\stfrak_{(10, 5, 1, 1)} + \stfrak_{(10, 5, 2)} - 2\stfrak_{(11, 3, 1, 1, 1)} - 4\stfrak_{(11, 3, 2, 1)} + \stfrak_{(11, 3, 3)}
    \\
    && - 3\stfrak_{(11, 5, 1)} - \stfrak_{(12, 2, 1, 1, 1)} + 6\stfrak_{(12, 3, 1, 1)}+ 2\stfrak_{(12, 3, 2)} + \stfrak_{(12, 5)}+ \stfrak_{(13, 1, 1, 1, 1)}+ 3\stfrak_{(13, 2, 1, 1)}
    \\
     &&   - 6\stfrak_{(13, 3, 1)}- 4\stfrak_{(14, 1, 1, 1)}
     - 3\stfrak_{(14, 2, 1)} + 2\stfrak_{(14, 3)}+ 6\stfrak_{(15, 1, 1)}  + \stfrak_{(15, 2)} - 4\stfrak_{(16, 1)} +\stfrak_{(17)} 
\end{eqnarray*}

As shown in \cite{martin2008distinguishing}, we know that $T$ has diameter four from its chromatic symmetric function. By Theorem \ref{thm:leading-partition}, we have $\lead = \lc(T) = (5,4,3,3,2)$. The coefficients of the partitions of length $\ell(\lead(\csft))-1=4$ that don't contain 1s are: $c_{(9,3,3,2)}=1$,  $c_{(7,5,3,2)}=2$, and $c_{(6,5,3,3)}=1$, whose partitions induce the 1-edge-adjacency multisets $\lmulti 5,4\rmulti, \lmulti 4, 3\rmulti$ (twice) and $\lmulti 4, 2\rmulti$, respectively. We have $N(4)=4>1=m_4$, which by Theorem \ref{thm:leaf-component-order} implies that the leaf component in $T$ contained in $\mathcal{I}_T$ has order $4$. Then, we draw all the remaining leaf components and connect them to the leaf component of order 4 in $\mathcal{I}_T$. This is captured in the figure below.

\begin{figure}[ht]
    \centering
    \begin{tikzpicture}[auto=center,every node/.style={circle, fill=black,scale=0.6}, thick, scale=0.9]
        
        \node (A1) at (7,0) {};
        \node (A2) at (7,0.75) {};
        \node (A3) at (7,-0.75) {};
        \node (A4) at (7.75,0) {};
        \node (A5) at (6.25,0) {};
        
        \draw(A1) -- (A2);
        \draw(A1) -- (A3);
        \draw(A1) -- (A4);
        \draw(A1) -- (A5);
        
        \node[fill=red] (B1) at (8.5, 0) {};
        \node[fill=red] (B2) at (8.5, 0.75) {};
        \node[fill=red] (B3) at (8.5, -0.75) {};
        \node[fill=red] (B4) at (9.25, 0) {};
        
        \draw[red](B1) -- (B2);
        \draw[red](B1) -- (B3);
        \draw[red](B1) -- (B4);
        
        \node (C1) at (10,0) {};
        \node (C2) at (10,0.75) {};
        \node (C3) at (10,-0.75) {};
        
        \draw(C1) -- (C2);
        \draw(C1) -- (C3);
        
        \node (D1) at (10.75,0) {};
        \node (D2) at (10.75,0.75) {};
        \node (D3) at (10.75,-0.75) {};
        
        \draw(D1) -- (D2);
        \draw(D1) -- (D3);
        
        \node (E1) at (11.5,0) {};
        \node (E2) at (11.5,0.75) {};
        
        \draw(E1) -- (E2);
        
        \draw[-stealth, line width=1pt](12.25,0) -- (14.25,0);
        
        \node[fill=red] (A1) at (16,0) {};
        \node[fill=red] (A2) at (16.75,0) {};
        \node[fill=red] (A3) at (15.25,0) {};
        \node[fill=red] (A4) at (16,0.75) {};
        
        \draw[red](A1) -- (A2);
        \draw[red](A1) -- (A3);
        \draw[red](A1) -- (A4);
        
        \node (B1) at (17,1) {};
        \node (B2) at (17.25,1.75) {};
        \node (B3) at (17.75,1.25) {};
        
        \draw(B1) -- (B2);
        \draw(B1) -- (B3);

        \node (C1) at (15,1) {};
        \node (C2) at (14.75, 1.75) {};
        \node (C3) at (14.25, 1.25){};

        \draw(C1) -- (C2);
        \draw(C1) -- (C3);

        \node (D1) at (17,-1) {};
        \node (D2) at (17.53, -1.53){};
        
        \draw(D1) -- (D2);
        
        \node (E1) at (15,-1){};
        \node (E2) at (14.75,-1.6) {};
        \node (E3) at (14.4, -1.2) {};
        \node (E4) at (15.25, -1.5) {};
        \node (E5) at (14.5, -0.75) {};
        
        \draw(E1) -- (E2);
        \draw(E1) -- (E3);
        \draw(E1) -- (E4);
        \draw(E1) -- (E5);
        \draw(A1) -- (B1);
        \draw(A1) -- (C1);
        \draw(A1) -- (D1);
        \draw(A1) -- (E1);    
    \end{tikzpicture}
    \caption{On the left, leaf components of the sizes given by $\lead(\csft)=(5,4,3,3,2)$. On the right, the tree $T$. }
\end{figure}
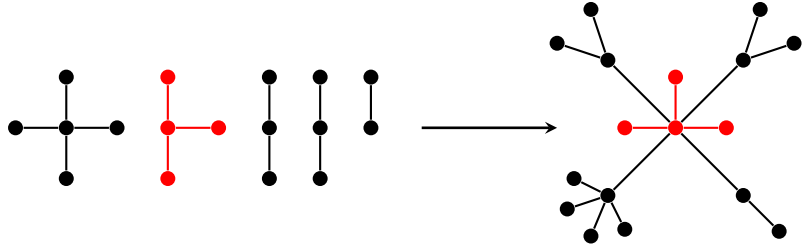
\end{example}

%%%%%%%%%%%%%%%%%%%%%%%%%%%%%%%%%%%%%%%%%%%%%%%%%%%%%%%%%%%%%%
\subsection{Reconstruction of trees of diameter 5}  We now focus on trees of diameter five. As in the case for diameter 4, we first determine which leaf components are in the internal subgraph of the tree, $\mathcal{I}_T$. By Corollary \ref{cor:leaf-comp-number-diam4-diam5} and Corollary \ref{cor:num-deep-vertices}, it follows that $m_1 = 0, 1$ or $2$ in the leading partition of a tree $T$ of diameter five. If $m_1 = 0$, then $T$ is a proper tree so Theorem \ref{thm:leaf-component-order} applies, and if $m_1 = 2$, then $T$ has two deep vertices and by Lemma \ref{lem:deep-vertex-in-internal-subgraph}, both are in $\mathcal{I}_T$, hence $\mathcal{I}_T$ is the tree of order 2. 

If $m_1 = 1$, then we know that one leaf component has order one. Proposition \ref{prop:order leaf component diam 5 one 1} below shows that we can also reconstruct the order of the other leaf component in $\mathcal{I}_T$. Assume that two non-isomorphic trees $T_1, T_2$ of diameter five have equal leading partition, i.e., $\lead=\lead(\mathbf{X}_{T_1})=\lead(\mathbf{X}_{T_2})$, and that the coefficients indexed by partitions of length $\ell(\lead) -1$ not containing ones are also equal in $\mathbf{X}_{T_1}$ and $\mathbf{X}_{T_2}$. Note that this condition is equivalent to saying that they have the same 1-edge-adjacency multisets appearing the same number of times. This means that $T_1$ and $T_2$ have the same leaf components adjacent to the leaf component of order 1. This equivalence follows from the following reasoning. For any part $p$ in $\lead$, Proposition \ref{prop:k-edge-adjacencies}(b) implies that the coefficient of $\stfrak_{\sort(p+1, n^{m_n}, \ldots, p^{{m_p}-1}, \ldots, 2^{m_2})}$ in $\mathbf{X}_{T_1}$ (or $\mathbf{X}_{T_2}$ since they agree on this partition) is equal to the number of leaf components of order $p$ that are adjacent to the leaf component of order 1. We will use this and the following lemma to prove
 Proposition \ref{prop:order leaf component diam 5 one 1}.

\begin{lemma} \label{lem:distinguish diam 5 varying center}
    Let $T_1$ and $T_2$ be two trees of diameter five with equal leading partition $\lead= (n^{m_n}, \ldots, 2^{m_2}, 1)$ and such that $\mathbf{X}_{T_1}$ and $\mathbf{X}_{T_2}$ agree on the coefficients indexed by partitions of length $\ell(\lead)-1$ containing no ones. Let $p_1, p_2 \neq 1$ be the orders of the non-singleton leaf components of the internal subgraphs of $T_1$ and $T_2$, respectively. Then, $\mathbf{X}_{T_1} = \mathbf{X}_{T_2}$ if and only if $p_1 = p_2$.
\end{lemma}

\begin{proof}
    Let $\mathcal{P}_1, \ldots, \mathcal{P}_k$ be the leaf components, of orders $p_1, \ldots, p_k$ respectively, which are adjacent to the singleton leaf component in $T_1$ and in $T_2$. Note that the fact that these leaf components are the same in both trees follows from the discussion in the paragraph preceding this lemma. Consequently, by Theorem \ref{thm:leading-partition},  these trees also agree on the remaining non-singleton leaf components $\mathcal{Q}_1, \ldots, \mathcal{Q}_l$, of orders $q_1, \ldots, q_l$ respectively, which are adjacent to the other leaf component in the internal subgraph of the respective trees. We can assume, without loss of generality, that $\mathcal{P}_1$ and $\mathcal{P}_2$ are the non-singleton leaf components in the internal subgraph of $T_1$ and $T_2$, respectively. A sketch of $T_1$ and $T_2$ can be seen in the figure below, where we have represented the leaf components by their orders to make it more clear. 
    
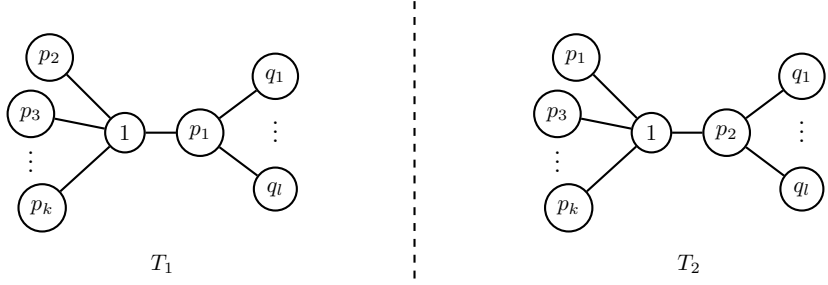
\begin{center}
    \begin{tikzpicture}[auto=center,every node/.style={circle, fill=none, draw=black, scale=0.85}, thick]
        
        \node (a1) at (0,0) {$1$};
        \node (a2) at (1,0) {$p_1$};
        \node (b1) at (-1, 1) {$p_2$};
        \node (a0) at (-1.25,0.25) {$p_3$};
        \node[draw=none] at (-1.25, -.3) {\Large{$\vdots$}};
        \node (c1) at (-1.1, -1) {$p_k$};
        \node (b2) at (2, 0.75) {$q_1$};
        \node[draw=none] at (2,0.1) {\Large{$\vdots$}};
        \node (c2) at (2, -0.75) {$q_l$};
        
        \draw (a0) -- (a1);
        \draw (a1) -- (a2);
        \draw (a1) -- (b1);
        \draw (a1) -- (c1);        
        \draw (a2) -- (b2);
        \draw (a2) -- (c2);
        
        \node[draw=none] at (0.5, -1.75) {$T_1$};
        
        \draw[dashed] (3.85, 1.75) -- (3.85, -2);
        
        % ------------------
        
        \node (a1) at (7,0) {$1$};
        \node (a2) at (8,0) {$p_2$};
        \node (b1) at (6, 1) {$p_1$};
        \node (a0) at (5.75,0.25) {$p_3$};
        \node[draw=none] at (5.75, -.3) {\Large{$\vdots$}};
        \node (c1) at (5.9, -1) {$p_k$};
        \node (b2) at (9, 0.75) {$q_1$};
        \node[draw=none] at (9,0.1) {\Large{$\vdots$}};
        \node (c2) at (9, -0.75) {$q_l$};
        
        \draw (a0) -- (a1);
        \draw (a1) -- (a2);
        \draw (a1) -- (b1);
        \draw (a1) -- (c1);        
        \draw (a2) -- (b2);
        \draw (a2) -- (c2);
        
        \node[draw=none] at (7.5, -1.75) {$T_2$};   
    \end{tikzpicture}
\end{center}
We remark that the $p_i$s and $q_j$s could be equal. 

    We can now prove the statement. We show the forward direction by contrapositive. Assume that $p_1 \neq p_2$.  We will show that some coefficients $c_{(\mu_1,\mu_2)}$ where $\mu_2\geq 2$ will differ in $\mathbf{X}_{T_1}$ and $\mathbf{X}_{T_2}$. The coefficient of such $c_\mu$ is equal to the number of paths in a DNC-tree of $T$ from the root $T$ to a star forest which corresponds to a sequence that uses exactly one deletion and leaf-contractions elsewhere.
    
    Note that in $T_1$, there are $m_{p_1} - 1$ leaf components of order $p_1$ that are not in the internal subgraph of $T_1$. The sequences that delete an internal edge connecting a leaf component of order $p_1$ not contained in $\mathcal{I_T}$ to one of the leaf components in $\mathcal{I}_T$ and leaf-contract all other edges, contribute to the coefficient of $c_{(n-p_1, p_1)}$, and if  $1 + p_2 + \ldots + p_k \neq p_1$, these are the only ones that contribute, hence $c_{(n-p_1,p_1)} = m_{p_1}-1$ in $\mathbf{X}_{T_1}$. In $\mathbf{X}_{T_2}$ there are $m_{p_1}$ leaf components of order $p_1$ not in $\mathcal{I}_T$, hence by a similar argument $c_{(n-p_1,p_1)} \geq m_{p_1}$, so $\mathbf{X}_{T_1} \neq \mathbf{X}_{T_2}$. 
     
    If $1 + p_2 + \ldots + p_k = p_1$, then $p_1 > p_2$. Thus, $1 + p_1 + p_3 + \ldots + p_k \neq p_2$ which by a similar argument as above implies $c_{(n-p_2,p_2)} = m_{p_2}-1$ in $\mathbf{X}_{T_2}$. In $T_1$, there are $m_{p_2}$ leaf components of order $p_2$ not contained in $\mathcal{I}_{T_1}$. Hence, $c_{(n-p_2,p_2)}\geq m_{p_2}$ in $\mathbf{X}_{T_1}$ and so $\mathbf{X}_{T_1} \neq \mathbf{X}_{T_2}$. This finishes the contrapositive argument for the forward direction.
    
    For the converse, it is clear from the discussion at the beginning of the proof that if $p_1 = p_2$, then $T_1 \cong T_2$, and so $\mathbf{X}_{T_1} = \mathbf{X}_{T_2}$.
\end{proof}

With this lemma, we can reconstruct the order of the other leaf component in the internal subgraph for the case $m_1=1$.
\begin{proposition} \label{prop:order leaf component diam 5 one 1}
    Let $\csft$ be the chromatic symmetric function of a tree, $T$, of diameter five with leading partition $\lead = (n^{m_n}, \ldots, 2^{m_2},1)$. Then, the internal subgraph, $\mathcal{I}_T$, is determined by $\csft$.
\end{proposition}

\begin{proof} 
    The internal subgraph of any tree of diameter 5 consists of two adjacent leaf components of orders given by two parts of the leading partition. Since 1 is a part in the leading partition, by Lemma \ref{lem:deep-vertex-in-internal-subgraph}, $\mathcal{I}_T$ contains the leaf component of order 1.  We now recover the order, $p\neq 1$, of the second leaf component in $\mathcal{I}_T$.  
    Let $\csft = \sum c_\lambda \mathfrak{st}_\lambda$. From the partitions $\lambda$ of length $\ell(\lead)-1$ that do not have 1s and such that $c_\lambda \neq 0$, we obtain the 1-edge-adjacency multisets $\lmulti p_1, 1\rmulti, \ldots, \lmulti p_k, 1\rmulti$. By Proposition \ref{prop:k-edge-adjacencies}(b), the numbers $p_1, \ldots, p_k$ are the orders of the leaf components that are adjacent to the one vertex leaf component in $\mathcal{I}_T$. Hence, one of the $p_i$'s is the order of the other leaf component in $\mathcal{I}_T$. For each $1\leq i \leq k$, let $T_i$ be the tree of diameter five such that the orders of the leaf components in the internal subgraph of $T_i$ are 1 and $p_i$ and such that $\csft$ and $\mathbf{X}_{T_i}$ agree on the coefficient of partitions of length $\ell(\lead(\csft))-1$ which do not contain 1s. By the discussion preceding Lemma \ref{lem:distinguish diam 5 varying center} and the lemma itself, we know that $T_i$ is unique up to isomorphism and that $\mathbf{X}_{T_i} = \csft$ if and only if $p = p_i$, so we can reconstruct $p$. 
    \end{proof}

Theorem \ref{thm:leaf-component-order}, Corollary \ref{cor:num-deep-vertices} and Proposition \ref{prop:order leaf component diam 5 one 1} show that we can reconstruct the orders of the leaf components of the internal subgraph for any tree of diameter five. We now show how to finish the reconstruction of a tree of diameter five by proving that we can also determine the adjacencies of all other leaf components.

\begin{proposition}\label{reconstruction diam 5 different orders}
    Let $T$ be a tree of diameter five with internal subgraph $\mathcal{I}_T$ and let $p, q$ be the orders of the leaf components in $\mathcal{I}_T$. If $p \neq q$, then $T$ can be reconstructed from its chromatic symmetric function $\csft$. In particular these trees can be reconstructed from $\lead(\csft)$ and the coefficients of partitions $\mu$ without 1s such that $\ell(\mu)=\ell(\lead(\csft))-1$ and $c_\mu \neq 0$.
\end{proposition}

\begin{proof}
    Let $\lead(\csft)=(n^{m_n}, \ldots, 1^{m_1})$. To reconstruct $T$, Corollary \ref{cor:leaf-comp-number-diam4-diam5} implies that it suffices to know how many leaf components of each order are adjacent to each of the leaf components in $\mathcal{I}_T$. Lemma \ref{lem:deep-vertex-in-internal-subgraph}, Corollary \ref{cor:1s-leading-diam45} and $p \neq q$ imply that either $m_1 = 0$ or $m_1 = 1$. Suppose that $m_1 = 0$, i.e. $T$ is proper. Let $p_1, \ldots, p_k$ be all the distinct parts in $\lead(\csft)$ that are not equal to $p$ or $q$. For each $1 \leq i \leq k$, the coefficient $c_{\lambda^{(i)}}$ of the partition $\lambda^{(i)} = \sort(p + p_i, n^{m_n}, \ldots, p^{m_p-1}, \ldots, p_i^{m_{p_i}-1}, \ldots, 2^{m_2})$ is equal to the number of leaf components of order $p_i$ that are adjacent to the leaf component of order $p$ in $\mathcal{I}_T$ by Proposition \ref{prop:k-edge-adjacencies}(a).
    Hence, the number of leaf components of order $p_i$ adjacent to the leaf component of order $q$ in $\mathcal{I}_T$ is equal to $m_{p_i}-c_{\lambda^{(i)}}$. Now consider the partition $\mu = \sort(q + q, n^{m_n}, \ldots, q^{m_q-2}, \ldots, 2^{m_2})$. Proposition \ref{prop:k-edge-adjacencies}(a) shows that its coefficient $c_\mu$ in $\csft$ equals the number of leaf components of order $q$ adjacent to the leaf component of order $q$ in $\mathcal{I}_T$. Therefore, the number of leaf components of order $q$ adjacent to the leaf component of order $p$ in $\mathcal{I}_T$ equals $m_q - c_\mu - 1$. Similarly, the number of leaf components of order $p$ adjacent to the leaf component of order $p$ in $\mathcal{I}_T$ is equal to $c_\nu$, where $\nu = (p +p, n^{m_n}, \ldots, p^{m_p-2}, \ldots, 2^{m_2})$, so the the number of leaf components of order $p$ adjacent to the leaf component of order $q$ in $\mathcal{I}_T$ equals $m_p - c_\nu - 1$. 

    Now consider the case where $m_1 = 1$ and assume, without loss of generality, that $p = 1$. Let $p_1, \ldots, p_k$ be all the distinct parts in $\lead(\csft)$ not equal to $1$ or $q$. For each $1 \leq i \leq k$, the coefficient $c_{\lambda_i}$ of the partition $\lambda^{(i)} = \sort(p_i + 1, n^{m_n}, \ldots, p_i^{m_{p_i}-1}, \ldots, 2^{m_2})$ is equal to the number of leaf components of order $p_i$ adjacent to the leaf component of order 1 in $\mathcal{I}_T$. Hence, the number of leaf components of order $p_i$ adjacent to the leaf component of order $q$ in $\mathcal{I}_T$ is equal to $m_{p_i}-c_{\lambda^{(i)}}$. Similarly, the coefficient of the partition $\sort(q+1,  n^{m_n}, \ldots, q^{m_q-1}, \ldots, 2^{m_2})$ is the number of leaf components of order $q$ adjacent to the singleton leaf component in $\mathcal{I}_T$, so we can reconstruct the number of those that are adjacent to the leaf component of order $q$ in $\mathcal{I}_T$.
\end{proof}

We now focus on the reconstruction of trees of diameter five whose internal subgraphs have two leaf components of equal order. The proof that we give for this case is very different in nature for the other cases. It relies on applying the DNC relation on the central edge and arguing on the forests that we obtain after applying the DNC operations. In what follows, we say that a tree $T$ is \defn{$p$-balanced} if all the leaf components in the internal subgraph of $T$ have order $p$. 

Observe that we can tell whether a tree of diameter 5 is $p$-balanced or not from its chromatic symmetric function. If a tree $T$ has diameter 5, then it is $1$-balanced if and only if the multiplicity of $1$ in $\lead(\csft)$ is $m_1=2$. Further, by Theorem \ref{thm:leaf-component-order}, $T$ is $p$-balanced for $p \geq 2$ if and only if $p$ is the only part in $\lead(\csft)$ such that $N(p) > m_p$. 

\begin{lemma}\label{lemma: diameter of T1 and T2}
    Let $T$ be a tree of diameter five with internal subgraph $\mathcal{I}_T$. Let $e=u_1u_2$ be the edge between the two internal vertices $u_1$  and  $u_2$ contained in $\mathcal{I}_T$. Let $T\setminus e = T_1 \sqcup T_2$ where $T_1$ and  $T_2$ are the trees containing $u_1$ and $u_2$, respectively. Then:
    \begin{enumerate}
        \item[(a)] $T_i$ has diameter 2 if and only if $u_i$ is a degree-2 deep vertex.
        \item[(b)] $T_i$ has diameter 3 if and only if $u_i$ has internal degree two but is not deep.
        \item[(c)] $T_i$ has diameter 4 if and only if the internal degree of $u_i$ is greater than or equal to three.
    \end{enumerate}
\end{lemma}
\begin{proof}
    Note that it suffices to prove the forward direction of each (a), (b) and (c) and that it also suffices to show the statements for $i=1$. (a) Suppose that $u_1$ is a degree-2 deep vertex with $N(u)=\{t, u_1\}$ where $t$ is an internal vertex. As shown in the proof of Lemma \ref{lem:deletions-lead}, after deleting $e$, $u$ becomes a leaf in the leaf component containing $t$. Hence, in this case $T_1$ is a star, which has diameter two. (b) Since $u_1$ is not deep and has internal degree two, then $\{t, \ell, u_2\} \subseteq N(u_1)$ where $t$ is internal and $\ell$ is a leaf. Then, $u_1$ remains internal in $T \setminus e$ and so $T_1$ has one internal edge, so it is a bi-star. Hence, $\diam(T_1)=3$. (c) If the internal degree of $u_1$ is at least three, then $u_1$ has degree at least 2 in $T \setminus e$, so it remains an internal vertex. Further, there are at least two internal edges in $T_1 \setminus e$ and all of them are incident to $u_1$, so $T_1$ is a tree of diameter 4 with $u_1$ as the center.
\end{proof}

\begin{proposition}\label{prop:xt1 and xt2}
    Let $\csft$ be the chromatic symmetric function of a tree $T$ with diameter five and $\lead = \lead(\csft)=(n^{m_n}, \ldots, 1^{m_1})$. Let $e$ be the internal edge in $\mathcal{I}_T$, the internal subgraph of $T$ and let $T\setminus e = T_1 \sqcup T_2$. Then, $\mathbf{X}_{T_1}\mathbf{X}_{T_2}$ can be recovered from $\csft$. Further, $\#V(T_1)$ and $\#V(T_2)$ can be recovered from $\csft$.
\end{proposition}
\begin{proof}
    Since $T$ has diameter five, then by Corollary \ref{cor:leaf-comp-number-diam4-diam5}(c), $\mathcal{I}_T$ consists of two leaf components of order $p,q$ whose centers are connected by the edge $e$. By the DNC relation and since $T\setminus e = T_1 \sqcup T_2$:
    \[
    \csft = \mathbf{X}_{T_1 \sqcup T_2} - \mathbf{X}_{(T \odot e)\setminus \ell_e} + \mathbf{X}_{T \odot e}.
    \] 
    Note that $(T \odot e)\setminus \ell_e = H\sqcup v$, where $v$ is an isolated vertex and $H$ is a tree of diameter four. If $m_1=0$, then $\lead(\mathbf{X}_H)=\sort(p+q-1,n^{m_n},\ldots, p^{m_p-1},q^{m_q-1}, \ldots, 2^{m_2})$. If $m_1=1$, then we can assume that $p=1$ and so $\lead(\mathbf{X}_H) = \sort(n^{m_n},\ldots,q^{m_q}, \ldots, 2^{m_2})$. Lastly, if $m_1=2$, then $p=q=1$ and so $\lead(\mathbf{X}_H) = \sort(n^{m_n},\ldots, 2^{m_2},1)$. Further, we know that the leaf component in the internal subgraph of $H$ has order $p+q-1$. Therefore, by the results obtained in Section \ref{subsec:diameter 4}, we can reconstruct $H$ from its leading partition and the degree of the central vertex, and therefore compute $\mathbf{X}_{(T \odot e)\setminus \ell_e}=\stfrak_{(1)}\mathbf{X}_H$. Similarly, $T \odot e$ is the unique tree of diameter four with leading partition $\sort(p+q,n^{m_n},\ldots, p^{m_p-1},q^{m_q-1} \ldots, 2^{m_2})$ and whose internal subgraph is a star of order $p+q$. Note that there are no 1s in $\lead(\mathbf{X}_{T \odot e})$. Hence, we can reconstruct $T \odot e$ and therefore compute $\mathbf{X}_{T \odot e}$. Therefore, we can solve for $\mathbf{X}_{T_1\sqcup T_2}$ from $\csft$:
    \[
    \mathbf{X}_{T_1\sqcup T_2} = \csft + \mathbf{X}_{(T \odot e)\setminus \ell_e} - \mathbf{X}_{T \odot e}
    \]
    This implies that we can recover $\mathbf{X}_{T_1 \sqcup T_2} = \mathbf{X}_{T_1}\mathbf{X}_{T_2}$ from $\csft$. In the star-basis, the only term whose partition has length one in $\mathbf{X}_{T_1}$ is $\stfrak_{\#V(T_1)}$, which occurs with coefficient 1, and the same applies to $\mathbf{X}_{T_2}$. Since the star-basis is multiplicative, this implies that the only partition $\mu$ of length two such that $c_\mu \neq 0$ in $\mathbf{X}_{T_1}\mathbf{X}_{T_2}$ is $\mu = \sort(\#V(T_1), \#V(T_2))$. Hence, we can recover $\#V(T_1)$ and $\#V(T_2)$ from $\csft$, finishing the proof.
\end{proof}

\begin{theorem}\label{thm:p-balanced diameter 5 reconstruction}
    Let $\csft$ be the chromatic symmetric function of a $p$-balanced tree, $T$, with diameter five and $\lead = \lead(\csft)=(n^{m_n}, \ldots, 1^{m_1})$. Then, $T$ can be reconstructed from $\csft$.
\end{theorem}
\begin{proof}
    Since $T$ is $p$-balanced, then $\mathcal{I}_T$ consists of two leaf components $\mathcal{L}_1$ and $\mathcal{L}_2$ of order $p$ whose respective centers $u_1$ and  $u_2$ are connected by an edge $e=u_1u_2$. Let $T\setminus e = T_1 \sqcup T_2$, where $T_1$ contains $u_1$ and $T_2$ contains $u_2$. By Proposition \ref{prop:xt1 and xt2}, we can recover $\#V(T_1)$ and  $\#V(T_2)$ from $\csft$. Let $N_1 = \#V(T_1)$ and $N_2 = \#V(T_2)$ and assume without loss of generality that $N_1\geq N_2$. By the same proposition, we can recover $\mathbf{X}_{T_1}\mathbf{X}_{T_2}$ from $\csft$. Let $\mathbf{X}_{T_1}\mathbf{X}_{T_2} = \sum_{\mu \vdash N_1 + N_2} c_\mu \stfrak_\mu$.
    
    Consider the case where $N_1 > N_2$. This implies that no indexing partition in $\mathbf{X}_{T_2}$ contains $N_1$ as a part. Therefore, $\sum_{\nu \vdash N_2}c_{(N_1,\nu)}\stfrak_{(N_1,\nu)}=\stfrak_{(N_1)}\mathbf{X}_{T_2}$ where $(N_1, \nu)$ is the partition obtained by adding a part $N_1$ to $\nu$. Hence, we can recover $\mathbf{X}_{T_2}$ from $\mathbf{X}_{T_1}\mathbf{X}_{T_2}$ by taking all partitions that have largest part $N_1$ in $\mathbf{X}_{T_1}\mathbf{X}_{T_2}$ and removing this largest part.
    In addition, $T_2$ has diameter at most 4 by Lemma \ref{lemma: diameter of T1 and T2}, so we can reconstruct $T_2$ by the results obtained previously in this section since all trees of diameter of diameter 4 or less can be reconstructed from their chromatic symmetric function. For each part $q$ of $\lead$, since we can reconstruct $T_2$, we know the number $A_q$ of leaf components of order $q$ that are adjacent to $\mathcal{L}_2$ for each part $q$ of $\lead$. Then, we know that there are $m_q - A_q$ leaf components of order $q$ adjacent to $\mathcal{L}_1$, so we can reconstruct $T_1$ and, therefore, we can reconstruct $T$.

    Suppose then that $N_1 = N_2$. Without loss of generality, we may assume $\lead(\mathbf{X}_{T_1}) \leq \lead(\mathbf{X}_{T_2})$ where $\leq$ is lexicographic order. Then, the smallest partition $\alpha$ in lexicographic order such that $c_{(N_1,\alpha)} \neq 0$ must be exactly $\alpha = \lead(\mathbf{X}_{T_1})$. Hence, we can recover $\lead(\mathbf{X}_{T_1})$ from $\mathbf{X}_{T_1}\mathbf{X}_{T_2}$. If $\ell(\lead(\mathbf{X}_{T_1}))=1$, then $T_1$ is a star, so it has diameter 2. By Lemma \ref{lemma: diameter of T1 and T2}(a), this implies that $p=1$ and that the only leaf component adjacent to $\mathcal{L}_1$ other than $\mathcal{L}_2$ has order $|\lead(\mathbf{X}_{T_1})|-1$.  Then, $T_1$ can be reconstructed. If $\ell(\lead(\mathbf{X}_{T_1})) = 2$, then $T_1$ is a bi-star so it has diameter 3. We showed in Corollaries \ref{cor:bi-stars-leading} and \ref{cor:bi-stars-distinguish} that bi-stars can be reconstructed from their leading partitions. Hence, $T_1$ can be reconstructed. Lastly, consider the case where $\ell(\lead(\mathbf{X}_{T_1})) \geq 3$. Then, $T_1$ has at least two internal edges and since $\diam(T_1) \leq 4$ by Lemma \ref{lemma: diameter of T1 and T2}, it follows that $\diam(T_1)=4$. Since we know that the order of the leaf component in the internal subgraph of $T_1$ is $p$ and we know $\lead(\mathbf{X}_{T_1})$, then by the results in Section \ref{subsec:diameter 4}, we can reconstruct $T_1$. Hence, we can reconstruct $T_1$ in this case also. By the same argument given above, this implies that we can reconstruct $T_2$ as well and since we know the orders of the leaf components in $\mathcal{I}_T$, we can also reconstruct $T$. 
\end{proof}

The results that we have shown in this section allow us to reconstruct trees of diameter at most five from their chromatic symmetric function. In fact, we can give a reconstruction algorithm. Given $\csft$ in the star-basis of a tree $T$ of diameter 5, the following algorithm reconstructs $T$:

\begin{itemize}
    \item If $\lead(\csft)$ has no ones, then apply Theorem \ref{thm:leaf-component-order} to determine the orders of the leaf components of $\mathcal{I}_T$, which is a bi-star. Then:
    \begin{itemize}
        \item If $T$ is not $p$-balanced, then the algorithm outlined in the proof of Proposition \ref{reconstruction diam 5 different orders} reconstructs $T$ using the adjacencies between leaf components.
        \item If $T$ is $p$-balanced, then by considering the edge $e$ contained in $\mathcal{I}_T$ and $T \setminus e = T_1 \sqcup T_2$, recover $T_1$ and $T_2$ as shown in the proof of Theorem \ref{thm:p-balanced diameter 5 reconstruction}, and then reconstruct $T$ using $T_1$ and $T_2$ and $\lead(\csft)$.
    \end{itemize}
    \item If $\lead(\csft)$ has one 1, then apply Lemma \ref{lem:distinguish diam 5 varying center} to obtain the order of the other leaf component in $\mathcal{I}_T$. Then, apply Proposition \ref{reconstruction diam 5 different orders} to obtain the adjacencies, reconstructing $T$.
    \item If $\lead(\csft)$ has two 1s, then $\mathcal{I}_T$ is an edge. Then, use Theorem \ref{thm:p-balanced diameter 5 reconstruction} to reconstruct $T$.
\end{itemize}

\begin{example}\label{ex:diam 5 reconstruction}
    Consider the following chromatic symmetric function of a tree, $T$ of diameter five.
    \begin{eqnarray*}
\csft &=& 
            3\stfrak_{(4,3,2^2,1^2)} + \stfrak_{(4, 3, 2^3)}- \stfrak_{(4, 3^2, 1^3)}- 4\stfrak_{(4, 3^2, 2, 1)}- 2\stfrak_{(4^2, 2, 1^3)}- 2\stfrak_{(4^2, 2^2, 1)}+3\stfrak_{(4^2, 3, 1^2)}+ \\
            && 2\stfrak_{(4^2, 3, 2)}- 2\stfrak_{(5, 3, 2, 1^3)}- 4\stfrak_{(5, 3, 2^2, 1)}+ \stfrak_{(5, 3^2, 1^2)}+ 2\stfrak_{(5, 3^2, 2)}+ \stfrak_{(5, 4, 1^4)}+ 8\stfrak_{(5, 4, 2, 1^2)}+\\
            &&2\stfrak_{(5, 4, 2^2)}- 5\stfrak_{(5, 4, 3, 1)}- \stfrak_{(5^2, 1^3)}- 4\stfrak_{(5^2, 2, 1)} + \stfrak_{(5^2, 3)}- \stfrak_{(6, 2^2, 1^3)}+ \stfrak_{(6, 3, 1^4)}+\\
            &&6\stfrak_{(6, 3, 2, 1^2)}+ \stfrak_{(6, 3, 2^2)}- 4\stfrak_{(6, 4, 1^3)} - 6\stfrak_{(6, 4, 2, 1)}+ \stfrak_{(6, 4, 3)}+ 3\stfrak_{(6, 5, 1^2)}+2\stfrak_{(6, 5, 2)}+\\
            &&   2\stfrak_{(7, 2, 1^4)}+ 3\stfrak_{(7, 2^2, 1^2)} - 4\stfrak_{(7, 3, 1^3)}- 6\stfrak_{(7, 3, 2, 1)}+ 6\stfrak_{(7, 4, 1^2)}+ 2\stfrak_{(7, 4, 2)}- 3\stfrak_{(7, 5, 1)}\\
            &&- \stfrak_{(8, 1^5)} - 8\stfrak_{(8, 2, 1^3)}- 3\stfrak_{(8, 2^2, 1)}+ 6\stfrak_{(8, 3, 1^2)}+ 2\stfrak_{(8, 3, 2)}- 4\stfrak_{(8, 4, 1)}+ \stfrak_{(8, 5)} + 5\stfrak_{(9, 1^4)}+\\
            &&  12\stfrak_{(9, 2, 1^2)}+ \stfrak_{(9, 2, 2)}- 4\stfrak_{(9, 3, 1)}+ \stfrak_{(9, 4)}- 10\stfrak_{(10, 1^3)}- 8\stfrak_{(10, 2, 1)}+ \stfrak_{(10, 3)}+ 10\stfrak_{(11, 1^2)}+ \\
            &&  2\stfrak_{(11, 2)}- 5\stfrak_{(12, 1)}+\stfrak_{(13)}
    \end{eqnarray*}
    Note that $\lead(\csft)=(4,3,2^2,1^2)$, so $T$ is 1-balanced. Let $e=u_1u_2$ be the edge in $\mathcal{I}_T$ and let $T \setminus e = T_1 \sqcup T_2$, where without loss of generality $T_1$ contains $u_1$ and $T_2$ contains $u_2$. From the proof of Proposition \ref{prop:xt1 and xt2}, we have:
    \[
    (T \odot e)\setminus \ell_e =
        \tikzmath{
        \filldraw (0,0) node {} circle (.06cm);
        
        \filldraw (0,-.5) node {} circle (.06cm);
        \filldraw (0,-1) node {} circle (.06cm);
        
        \filldraw (-1,0) node {} circle (.06cm);
        \filldraw (-.5,0) node {} circle (.06cm);
        
        \filldraw (.5,0) node {} circle (.06cm);
        \filldraw (1,0.2) node {} circle (.06cm);
        \filldraw (1,-.2) node {} circle (.06cm);
        
        \filldraw (0,0.5) node {} circle (.06cm);
        
        \filldraw (0,1) node {} circle (.06cm);
        \filldraw (0.25,0.93) node {} circle (.06cm);
        \filldraw (-0.25,0.93) node {} circle (.06cm);
        
        \filldraw (-.75,0.75) node {} circle (.06cm);

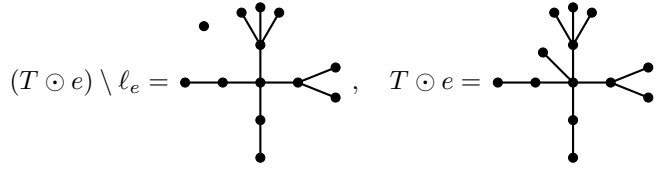
        \draw[thick] (-1,0)--(.5,0);
        \draw[thick] (0,1)--(0,-1);
        \draw[thick] (0.25,0.93)--(0,0.5)--(-.25,.93);
        \draw[thick] (1,0.2)--(.5,0)--(1,-0.2);
        }\;, \quad
        T \odot e =
            \tikzmath{
            \filldraw (0,0) node {} circle (.06cm);
        
        \filldraw (0,-.5) node {} circle (.06cm);
        \filldraw (0,-1) node {} circle (.06cm);
        
        \filldraw (-1,0) node {} circle (.06cm);
        \filldraw (-.5,0) node {} circle (.06cm);
        
        \filldraw (.5,0) node {} circle (.06cm);
        \filldraw (1,0.2) node {} circle (.06cm);
        \filldraw (1,-.2) node {} circle (.06cm);
        
        \filldraw (0,0.5) node {} circle (.06cm);
        
        \filldraw (0,1) node {} circle (.06cm);
        \filldraw (0.25,0.93) node {} circle (.06cm);
        \filldraw (-0.25,0.93) node {} circle (.06cm);
        
        \filldraw (-.4,0.4) node {} circle (.06cm);

        \draw[thick] (-1,0)--(.5,0);
        \draw[thick] (0,1)--(0,-1);
        \draw[thick] (0.25,0.93)--(0,0.5)--(-.25,.93);
        \draw[thick] (1,0.2)--(.5,0)--(1,-0.2);
        \draw[thick] (-.4,0.4)--(0,0);
        }
    \]
    Then, computing $\mathbf{X}_{(T \odot e)\setminus \ell_e}, \mathbf{X}_{T \odot e}$ and solving for $\mathbf{X}_{T_1}\mathbf{X}_{T_2}$ in the DNC relation yields:
    \begin{eqnarray*}
        \mathbf{X}_{T_1}\mathbf{X}_{T_2} &=& 
        - 2\stfrak_{(5,3, 2^2, 1)} + \stfrak_{(5,3^2, 1^2)} + 2\stfrak_{(5,3^2, 2)} + 2\stfrak_{(5,4, 2, 1^2)} + \stfrak_{(5,4, 2^2)} - 2\stfrak_{(5,4, 3, 1)} \\
        && - \stfrak_{(5^2, 1^3)} - 4\stfrak_{(5^2, 2, 1)}+ \stfrak_{(5^2, 3)}  + 3\stfrak_{(6,5, 1^2)} + 2\stfrak_{(6,5, 2)} - 3\stfrak_{(7,5, 1)}+ \stfrak_{(8,5)} 
\end{eqnarray*}
The only two-part partition whose term has nonzero coefficient in $\mathbf{X}_{T_1}\mathbf{X}_{T_2}$ is $(8,5)$. Hence, without loss of generality, $\#V(T_1)=8$ and $\#V(T_2)=5$. Since $\#V(T_1)>\#V(T_2)$, we can recover $\mathbf{X}_{T_2}$ from $\mathbf{X}_{T_1}\mathbf{X}_{T_2}$. In particular, we obtain $\mathbf{X}_{T_2} = \stfrak_{(5)}$ as the only term in $\mathbf{X}_{T_1}\mathbf{X}_{T_2}$ whose indexing partition contains $\#V(T_1)=8$ is $\stfrak_{(8,5)}$. This implies that $T_2 = St_5$. 
By Lemma \ref{lemma: diameter of T1 and T2}(a), $u_2$ is a deep vertex of degree two. Hence, the only leaf component not contained in $\mathcal{I}_T$ that is adjacent to $u_2$ has order 4. Then, we know that the remaining leaf components, which have orders $3,2,2$ are adjacent to $u_1$. Therefore:
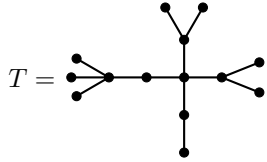
\[
T = \tikzmath{
        \filldraw (0,0) node {} circle (.06cm);
        \filldraw (0,-.5) node {} circle (.06cm);
        \filldraw (0,-1) node {} 
        circle (.06cm);
        
        \filldraw (-1.5,0) node {} circle (.06cm);
        \filldraw (-1.43, 0.25) node {} circle (.06cm);
        \filldraw (-1.43, -.25) node {} circle (.06cm);
        
        \filldraw (-1,0) node {} circle (.06cm);
        \filldraw (-.5,0) node {} circle (.06cm);
        
        \filldraw (.5,0) node {} circle (.06cm);
        \filldraw (1,0) node {} circle (.06cm);
        %\filldraw (1,-.2) node {} circle (.06cm);
        
        \filldraw (0,0.5) node {} circle (.06cm);
        
        \filldraw (0.25,0.93) node {} circle (.06cm);
        \filldraw (-0.25,0.93) node {} circle (.06cm);
        
        \draw[thick] (-1.43,-.25)--(-1,0)--(-1.43,.25);
        \draw[thick] (-1.5,0)--(.5,0);
        \draw[thick] (0,.5)--(0,-1);
        \draw[thick] (0.25,0.93)--(0,0.5)--(-.25,.93);
        \draw[thick] (1,0)--(.5,0);
}
\]
\end{example}
%%%%%%%%%%%%%%%%%%%%%%%%%%%%%%%%%%%%%%%%%%%%%%%%%%%%%%%%%%%%%%
\section{The symmetric tree chromatic subspace}
\label{sec:subspace}
%%%%%%%%%%%%%%%%%%%%%%%%%%%%%%%%%%%%%%%%%%%%%%%%%%%%%%%%%%%%%%
In this section, we provide another application of the results obtained in Section \ref{sec:leading}. In particular, for any positive integer $n$, we consider the $\C$-vector space spanned by the set 
\[ \{ \csft : \text{$T$ is a tree with $n$ vertices}\},\] 
and we prove that it has dimension $p(n)-n+1$, where $p(n)$ is the number of partitions of $n$.  In addition, we give a construction of a basis for this vector space which we call the caterpillar basis.

The lexicographic order for partitions allows us to consider $\{\stfrak_\lambda : \lambda \vdash n \}$ as an ordered basis for $\Lambda^n$ for any integer $n$. Let $\lambda^{(1)}<\lambda^{(2)}<\cdots< \lambda^{(p(n))}$ be the partitions of $n$ in increasing lexicographic order. We use this ordering to obtain the ordered basis $\{\stfrak_{\lambda^{(1)}}, \ldots, \stfrak_{\lambda^{(p(n))}}\}$. 

\begin{definition}
    Let $T$ be a tree on $n$ vertices with $\csft=\sum_{j=1}^{p(n)} c_{\lambda^{(j)}}\mathfrak{st}_{\lambda^{(j)}}$. We define the \defn{star-vector} of $T$ with respect to the ordered basis $\{\stfrak_{\lambda^{(1)}}, \ldots, \stfrak_{\lambda^{(p(n))}}\}$ to be the coordinate vector $[c_{\lambda^{(1)}} \; c_{\lambda^{(2)}} \; \cdots \; c_{\lambda^{(p(n))}}]$. Note that this allows us to write 
    \[
        \csft =
        \begin{bmatrix}
        c_{\lambda^{(1)}} & c_{\lambda^{(2)}} & \cdots & c_{\lambda^{(p(n))}}
        \end{bmatrix}
        \begin{bmatrix}
        \mathfrak{st}_{\lambda^{(1)}} \\
        \mathfrak{st}_{\lambda^{(2)}} \\
        \vdots \\
        \mathfrak{st}_{\lambda^{(p(n))}}
        \end{bmatrix}.
    \]
\end{definition}

\begin{example}
    Let $P_5$ be the path on 5 vertices. Then, we have $\mathbf{X}_{P_5} = -\stfrak_{(2,2,1)}+\stfrak_{(3,1,1)}+ 2\stfrak_{(3,2)}- 2\stfrak_{(4,1)} +\stfrak_{(5)}$, so the star-vector of $P_5$ is $\begin{bmatrix}
        0 & 0 & -1 & 1 &2 & -2 & 1
    \end{bmatrix}$.
\end{example}

\begin{definition}
    Let $T_1,T_2,\ldots,T_q$ be all the trees with $n$ vertices. For each $1 \leq i \leq q$, we write $\mathbf{X}_{T_i}=\sum_{j=1}^{p(n)} c^i_{\lambda^{(j)}} \mathfrak{st}_{\lambda^{(j)}}$ for some integers $c^i_{\lambda^{(1)}},\ldots,c^i_{\lambda^{(p(n))}}$. Up to reordering of the trees, we define an \defn{$n$-CSF matrix} to be the $q \times p(n)$ matrix $(c^i_{\lambda^{(j)}})_{i,j}$. In matrix notation:
    \[
        \begin{bmatrix}
            \mathbf{X}_{T_1} \\
            \mathbf{X}_{T_2} \\
            \vdots \\
            \mathbf{X}_{T_q}
        \end{bmatrix} 
        =
        \begin{bmatrix}
            c^1_{\lambda^{1}} & c^1_{\lambda^{2}} & \cdots & c^1_{\lambda^{p(n)}}\\
            c^2_{\lambda^{1}} & c^2_{\lambda^{2}} & \cdots & c^2_{\lambda^{p(n)}}\\
            \vdots & \vdots & \ddots & \vdots \\
            c^q_{\lambda^{1}} & c^q_{\lambda^{2}} & \cdots & c^q_{\lambda^{p(n)}}\\
        \end{bmatrix}
        \begin{bmatrix}
            \mathfrak{st}_{\lambda^1} \\
            \mathfrak{st}_{\lambda^2} \\
            \vdots \\
            \mathfrak{st}_{\lambda^{p(n)}}
        \end{bmatrix}.
    \]
\end{definition}

\begin{example}
    We give an example of the $6$-CSF matrix. We enumerate the trees on 6 vertices as follows:
    \[
    \quad T_1 = \tikzmath{
        \filldraw (0,0) node {} circle (.06cm);
        \filldraw (0.6,0) node {} circle (.06cm);
        \filldraw (1.2,0) node {} circle (.06cm);
        \filldraw (1.8,0) node {} circle (.06cm);
        \filldraw (2.4,0) node {} circle (.06cm);
        \filldraw (3,0) node {} circle (.06cm);

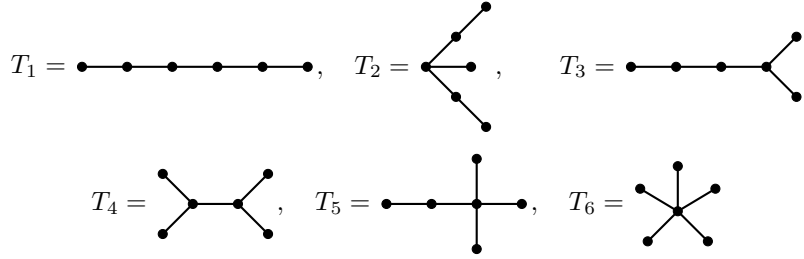
        \draw[thick] (0,0)--(3,0);
    }, \quad
    T_2 = \tikzmath{
        \filldraw (0,0) node {} circle (.06cm);
        \filldraw (0.4,.4) node {} circle (.06cm);
        \filldraw (.4,-.4) node {} circle (.06cm);
        \filldraw (0.8,.8) node {} circle (.06cm);
        \filldraw (0.8,-.8) node {} circle (.06cm);
        \filldraw (0.6, 0) node {} circle (.06cm);

        \draw[thick] (0,0)--(.8,.8);
        \draw[thick] (0,0)--(.8,-.8);
        \draw[thick] (0,0)--(.6,0);
    }, \quad
    \quad T_3=
    \tikzmath{
        \filldraw (0,0) node {} circle (.06cm);
        \filldraw (0.6,0) node {} circle (.06cm);
        \filldraw (1.2,0) node {} circle (.06cm);
        \filldraw (1.8,0) node {} circle (.06cm);
        \filldraw (2.2,-.4) node {} circle (.06cm);
        \filldraw (2.2,.4) node {} circle (.06cm);

        \draw[thick] (0,0)--(.6,0);
        \draw[thick] (1.2,0)--(.6,0);
        \draw[thick] (1.2,0)--(1.8,0);
        \draw[thick] (2.2,0.4)--(1.8,0);
        \draw[thick] (2.2,-.4)--(1.8,0);
    }   
    \]
    \[\quad T_4 = 
    \tikzmath{
        \filldraw (0,0) node {} circle (.06cm);
        \filldraw (0.6,0) node {} circle (.06cm);
        \filldraw (1,.4) node {} circle (.06cm);
        \filldraw (1,-.4) node {} circle (.06cm);
        \filldraw (-.4,.4) node {} circle (.06cm);
        \filldraw (-.4,-.4) node {} circle (.06cm);

        \draw[thick] (0,0)--(.6,0);
        \draw[thick] (0,0)--(-.4,.4);
        \draw[thick] (0,0)--(-.4,-.4);
        \draw[thick] (.6,0)--(1,-.4);
        \draw[thick] (.6,0)--(1,.4);
    }, \quad
    T_5 = 
    \tikzmath{
        \filldraw (0,0) node {} circle (.06cm);
        \filldraw (0.6,0) node {} circle (.06cm);
        \filldraw (1.2,0) node {} circle (.06cm);
        \filldraw (1.2,.6) node {} circle (.06cm);
        \filldraw (1.2,-.6) node {} circle (.06cm);
        \filldraw (1.8,0) node {} circle (.06cm);

        \draw[thick] (0,0)--(.6,0);
        \draw[thick] (1.2,0)--(.6,0);
        \draw[thick] (1.2,0)--(1.2,.6);
        \draw[thick] (1.2,0)--(1.2,-.6);
        \draw[thick] (1.2,0)--(1.8,0);
    }, \quad
    T_6 = 
    \tikzmath{
        \filldraw (0,0) node {} circle (.06cm);
        \filldraw (0,.6) node {} circle (.06cm);
        \filldraw (.5,0.3) node {} circle (.06cm);
        \filldraw (-.5,0.3) node {} circle (.06cm);
        \filldraw (-.4,-0.4) node {} circle (.06cm);
        \filldraw (.4,-0.4) node {} circle (.06cm);
    
        \draw[thick] (0,0)--(0,.6);
        \draw[thick] (0,0)--(.5,.3);
        \draw[thick] (0,0)--(-.5,.3);
        \draw[thick] (0,0)--(-.4,-.4);
        \draw[thick] (0,0)--(.4,-.4);
    }
    \]
    Then, we have:
    \[
        \begin{bmatrix}
            \mathbf{X}_{T_1} \\
            \mathbf{X}_{T_2} \\
            \mathbf{X}_{T_3} \\
            \mathbf{X}_{T_4} \\ 
            \mathbf{X}_{T_5} \\
            \mathbf{X}_{T_6}
        \end{bmatrix} 
        =\begin{bmatrix}
            0 & 0 & 1 & 1 & -1 & -4 & 1 & 3 & 2 & -3 & 1\\
            0 & 0 & 0 & 1 & 0 & -2 & 0 & 1 & 2 & -2 & 1\\
            0 & 0 & 0 & 0 & 0 & -1 & 1 & 1 & 1 & -2 & 1\\
            0 & 0 & 0 & 0 & 0 & 0 & 1 & 0 & 0 & -1 & 1\\
            0 & 0 & 0 & 0 & 0 & 0 & 0 & 0 & 1 & -1 & 1\\
            0 & 0 & 0 & 0 & 0 & 0 & 0 & 0 & 0 & 0 & 1
        \end{bmatrix}\cdot
        \begin{bmatrix}
            \mathfrak{st}_{(1^6)} \\
            \mathfrak{st}_{(2,1^4)} \\
            \mathfrak{st}_{(2^2,1^2)}  \\
            \mathfrak{st}_{(2^3)} \\
            \mathfrak{st}_{(3,1^3)}\\
            \mathfrak{st}_{(3,2,1)}\\
            \mathfrak{st}_{(3,3)}\\
            \mathfrak{st}_{(4,1^2)}\\
            \mathfrak{st}_{(4,2)}\\
            \mathfrak{st}_{(5,1)}\\
            \mathfrak{st}_{(6)}
        \end{bmatrix}
    \]
\end{example}

    If $\lambda = (\lambda_1, \ldots, \lambda_\ell) \vdash n$, we say that $\lambda$ is a \defn{hook partition}, or simply a hook, if $\ell > 1$ and $\lambda_2 = \cdots = \lambda_\ell = 1$. We say that $\lambda$ is a \defn{non-hook partition} if $\ell = 1$ or $\lambda_2 > 1$.

\begin{definition}
    A tree $T$ is a \defn{caterpillar} if all its internal edges form a path, which is called the \defn{spine} of the caterpillar. If $v_1,v_2, \ldots, v_{k-1},v_k$ is the spine of the caterpillar, we define the \defn{leaf-component sequence} of $T$, up to reversal, as $(L_1, \ldots, L_k)$ where $L_i$ is the order of the unique leaf component containing $v_i$ as its central vertex.
\end{definition}

\begin{example}
    The following tree is a caterpillar with leaf-component sequence $(5,3,1,2,4)$.
    
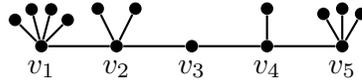
\begin{center}
        \begin{tikzpicture}[auto=center,every node/.style={circle, fill=black, scale=0.5}, thick]
            \node (c0) at (0,0) {};
            \node (c1) at (1,0) {};
            \node (c2) at (2,0) {};
            \node (c3) at (3,0) {};
            \node (c4) at (4,0) {};

            \node[fill=none,scale=2.25] at (0, -.3) {$v_1$};
            \node[fill=none,scale=2.25] at (1, -.3) {$v_2$};
            \node[fill=none,scale=2.25] at (2, -.3) {$v_3$};
            \node[fill=none,scale=2.25] at (3, -.3) {$v_4$};
            \node[fill=none,scale=2.25] at (4, -.3) {$v_5$};

            \node (c01) at (-.35, .35) {};
            \node (c02) at (.13, .49) {};
            \node (c03) at (-.13, .49) {};
            \node (c04) at (.35, .35) {};

            \node (c11) at (.75, .5) {};
            \node (c12) at (1.25, .5) {};

            \node (c31) at (3, .5) {};

            \node (c41) at (3.75, .43) {};
            \node (c42) at (4.25, .43) {};
            \node (c43) at (4, .5) {};

            \draw (c0) -- (c1) -- (c2) -- (c3) -- (c4);

            \draw (c0) -- (c01);
            \draw (c0) -- (c02);
            \draw (c0) -- (c03);
            \draw (c0) -- (c04);

            \draw (c1) -- (c11);
            \draw (c1) -- (c12);

            \draw (c3) -- (c31);
            
            \draw (c4) -- (c41);
            \draw (c4) -- (c42);
            \draw (c4) -- (c43);
        \end{tikzpicture}
    \end{center}
\end{example}

\begin{remark}
    Given a sequence of positive integers $\alpha = (\alpha_1, \ldots, \alpha_k)$ where $\alpha_1,\alpha_k>1$, there is a unique caterpillar up to isomorphism with $\alpha$ as its leaf-component sequence. This caterpillar can be constructed by creating a path with vertices $v_1, \ldots, v_k$ and adding $\alpha_i-1$ leaves to $v_i$ for each $1 \leq i \leq k$. We denote this caterpillar by $C[\alpha_1, \ldots, \alpha_k]$ or simply $C[\alpha]$. For example, given the sequence $(4,2,1,1,2,1,1,6)$, the caterpillar $C[4,2,1,1,2,1,1,6]$ up to isomorphism is:
    
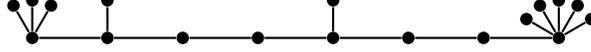
\begin{center}
        \begin{tikzpicture}[auto=center,every node/.style={circle, fill=black, scale=0.5}, thick]
            \node (c0) at (0,0) {};
            \node (c1) at (1,0) {};
            \node (c2) at (2,0) {};
            \node (c3) at (3,0) {};
            \node (c4) at (4,0) {};
            \node (c5) at (5,0) {};
            \node (c6) at (6,0) {};
            \node (c7) at (7,0) {};

            \node (c01) at (-.25, .43) {};
            \node (c02) at (0, .5) {};
            \node (c03) at (0.25, .43) {};

            \node (c11) at (1, .5) {};

            \node (c41) at (4, .5) {};

            \node (c71) at (6.75, .43) {};
            \node (c72) at (7.25, .43) {};
            \node (c73) at (7, .5) {};
            \node (c74) at (7.43, .25) {};
            \node (c75) at (6.57, .25) {};
            
            \draw (c0) -- (c1) -- (c2) -- (c3) -- (c4) -- (c5) -- (c6) -- (c7);

            \draw (c0) -- (c01);
            \draw (c0) -- (c02);
            \draw (c0) -- (c03);

            \draw (c1) -- (c11);

            \draw (c4) -- (c41);
            
            \draw (c7) -- (c71);
            \draw (c7) -- (c72);
            \draw (c7) -- (c73);
            \draw (c7) -- (c74);
            \draw (c7) -- (c75);
        \end{tikzpicture}
    \end{center}
\end{remark}

We can use these concepts to obtain the following result.

\begin{lemma}\label{lemma:non-hook-building}
    Let $\lambda \vdash n$ be a non-hook partition. Then, there exists a tree $T$ with $n$ vertices such that $\lead(\csft) = \lambda$.
\end{lemma}

\begin{proof}
    If $\ell(\lambda) = 1$, then $T = St_n$ has $\lambda=(n)$ as its leading partition. Assume then that $\ell(\lambda) = \ell > 1$ and write $\lambda = (\lambda_1, \ldots, \lambda_\ell)$. Since $\lambda$ is not a hook partition and it has length greater than 1, it follows that $\lambda_1, \lambda_2 > 1$. Then the caterpillar $C = C[\lambda_2, \ldots, \lambda_\ell,\lambda_1]$ has $\lead(\mathbf{X}_{C}) = \lambda$.
\end{proof}

We now use Lemma \ref{lemma:non-hook-building} to prove the following lower bound for the rank of an $n$-CSF matrix.

\begin{proposition}\label{prop:lower-bound-rank}
Let $\mathcal{M}$ be an $n$-CSF matrix for some $n\geq 1$. Then, $\rank(\mathcal{M})\geq p(n)-n+1$.
\end{proposition}

\begin{proof}
Fix an $n$-CSF matrix $\mathcal{M}$. There are exactly $n-1$ hook partitions of $n$. For each of the $p(n)-n+1$ non-hook partitions $\lambda \vdash n$, it follows from Lemma \ref{lemma:non-hook-building} that there is a tree $T$ with $\lead(\mathbf{X}_{T})=\lambda$. Then, let us fix a set of $p(n) - n+1$ trees on $n$ vertices with mutually distinct leading partitions. The rows corresponding to the CSF vectors of these trees are linearly independent because the first nonzero entry of a row corresponding to a tree with leading partition $\lambda$ is strictly to the left of the first nonzero entry of a row corresponding to a tree with leading partition $\mu$ if $\lambda < \mu$ in lexicographic order and the trees are ordered by increasing leading partition. 
\end{proof}

We have just shown a lower bound on the rank of any $n$-CSF matrix. We will shortly prove that $p(n)-n+1$ is also an upper bound for $\rank(\mathcal{M})$. However, we first need the following proposition.

\begin{lemma}\label{lem:sum-coeff-0}
    Let $F$ be a forest on $n$ vertices and let $cc(F)$ denote the number of connected components of $F$. If $ \csff = \sum_{\lambda \vdash n}c_{\lambda}\mathfrak{st}_\lambda$, then for all $m>cc(F)$, the following holds
    $$
    \sum_{\substack{\lambda\vdash n \\ \ell(\lambda)=m}} c_\lambda = 0.
    $$
\end{lemma}

\begin{proof}
    We will proceed by induction on the number of internal edges $k$ of $F$. If $k=0$, then $F$ is a star forest, and $\csff = \stfrak_{(\lambda_1, \ldots, \lambda_\ell)}$ where $\lambda_1, \ldots, \lambda_\ell$ are the orders of the star graphs. Thus, the claim holds vacuously since $c_\lambda = 0$ for all $\lambda$ such that $\ell(\lambda) > cc(F)=\ell$. 
    
    Now suppose that the claim holds for all forests with $n$ vertices with at most $k$ internal edges. Assume that $F$ has $k+1$ internal edges and choose an internal edge $e$. By the deletion-near-contraction relation, $\csff=\mathbf{X}_{F \setminus e}-\mathbf{X}_{(F\odot e)\setminus \ell_e}+\mathbf{X}_{F\odot e}$. Note that $F\setminus e$, $(F\odot e)\setminus \ell_e$, and $F\odot e$ each have $n$ vertices and at most $k$ internal edges. In addition, $F\setminus e$, $(F\odot e)\setminus \ell_e$ have one more connected component than $F$, while $F\odot e$ has the same number of connected components than $F$.
    
    The coefficients indexed by partitions with more than $cc(F)+1$ parts in $\mathbf{X}_{F\setminus e}$, $\mathbf{X}_{(F\odot e)\setminus \ell_e}$, and $\mathbf{X}_{F\odot e}$ are all $0$ by induction hypothesis. Hence, if $m > cc(F)+1$, we have $\sum c_\lambda = 0$, where the sum runs over partitions of length $m$.

    It remains to show that the claim holds for partitions of length 
    $m=cc(F)+1$. In this case, the induction hypothesis still applies for $F\odot e$ since $cc(F\odot e) = cc(F)$. Suppose $\mathbf{X}_{F\setminus e}=\sum_{\lambda \vdash n} a_\lambda\mathfrak{st}_\lambda$ and $ \mathbf{X}_{(F\odot e)\setminus \ell_e}=\sum_{\lambda \vdash n}b_\lambda\mathfrak{st}_\lambda$.
        
    Note that $cc(F\setminus e)=cc((F\odot e)\setminus \ell_e)=m$. Since deletions and dot-contractions increase the number of connected components, the only path in the DNC-tree from the root (either $F\setminus e$ or $(F\odot e)\setminus \ell_e)$) to a star forest with $m$ connected components is the sequence of repeated leaf-contractions in both. Therefore, we obtain:
    \[
        \sum_{\substack{\lambda\vdash n, \\ \ell(\lambda)=m}}a_\lambda=\sum_{\substack{\lambda\vdash n, \\ \ell(\lambda)=m}}b_\lambda=1.
    \]
    Since $\mathbf{X}_{F\setminus e}$ and $\mathbf{X}_{(F\odot e)\setminus \ell_e}$ have opposite parities in the DNC relation, the claim follows.
\end{proof}

Lemma \ref{lem:sum-coeff-0} above is generalized to forests, not just trees. However, we shall only use the result for the chromatic symmetric function of trees. In particular, we use it to prove an upper bound for the rank of an $n$-CSF matrix.

\begin{proposition}\label{prop:upper-bound-rank}
    Let $\mathcal{M}$ be an $n$-CSF matrix for some $n\geq 1$. Then, $\rank(\mathcal{M}) \leq p(n)-n+1$
\end{proposition}
\begin{proof}
    Recall that $\mathcal{M}$ has $p(n)$, number of partitions of $n$, columns. Let $k\in [n-1]$ be arbitrary, and let $h_k$ denote the unique hook partition with biggest part $k$, i.e. $h_k=(k, 1^{n-k})$. By Lemma \ref{lem:sum-coeff-0}, we have that the sum of the column vectors in $\mathcal{M}$ corresponding to partitions of length $n-k+1$ is the zero vector. That is, the column indexed by $h_k$ can be expressed as a linear combination of the column vectors corresponding to non-hook partitions of length $n-k+1$. Thus, we can express $n-1$ columns of $\mathcal{M}$ as linear combinations of other columns in $\mathcal{M}$, so the claim follows immediately.
\end{proof}

From the previous results, we obtain the following theorem.

\begin{theorem}\label{thm:dimension}
     For each positive integer $n$, the $\C$-vector space $\mathcal{V}_n = \spn_\C\{\csft : |V(T)|=n\}$ has dimension $p(n)-n+1$. Furthermore, there exists a basis of caterpillars for this space which we call the \defn{\textit{caterpillar basis}}.
\end{theorem}
\begin{proof}
    The first statement of the theorem is obvious from Propositions \ref{prop:lower-bound-rank} and \ref{prop:upper-bound-rank}. For the second part, note that we just need to find $p(n)-n+1$ linearly independent vectors in $\mathcal{V}_n$. Suppose $\mu=(\mu_1, \ldots, \mu_\ell)$ is a non-hook partition. Construct a caterpillar $C[\mu]$ with leaf component sequence $(\mu_2, \ldots, \mu_\ell, \mu_1)$. Now, let $\mu^{(1)}, \ldots, \mu^{(p(n)-n+1)}$ be the non-hook partitions of $n$, then the set $\{\mathbf{X}_{C[\mu^{(1)}]}, \ldots, \mathbf{X}_{C[\mu^{(p(n)-n+1)}]}\}$ is certainly linearly independent since each term has a different leading partition. Thus, this set is indeed a basis for $\mathcal{V}_n$ and we refer to it as the \defn{caterpillar basis} for $\mathcal{V}_n$.
\end{proof}

Note that one of the main consequences of Theorem \ref{thm:dimension} is that there are linear relations among the chromatic symmetric functions of trees with $n$ vertices. We provide an example below.

\begin{example}
    Letting $\sim$ represent equality when passing to the chromatic symmetric function, we have the following linear combination:
    \begin{figure}[ht]
        \centering
        \begin{tikzpicture}[auto=center,every node/.style={circle, fill=black, scale=0.5}, thick]
            \node (a) at (0,0) {};
            \node (b) at (1,0) {};
            \node (c) at (2,0) {};

            \node (a1) at (0,.5) {};

            \node (b1) at (1.25, .43) {};
            \node (b2) at (1,.5) {};
            \node (b3) at (.75, .43) {};

            \node (c1) at (2, .5) {};

            \draw (a)--(b)--(c);
            
            \draw (a)--(a1);
            \draw (b)--(b1);
            \draw (b)--(b2);
            \draw (b)--(b3);
            \draw (c)--(c1);

            \node[fill=none, scale=2] at (2.75,0.25) {$\sim$};

            \node (a) at (3.5, 0) {};
            \node (b) at (4.5, 0) {};
            \node (c) at (5.5, 0) {};

            \node (a1) at (3.5, 0.5) {};
            \node (b1) at (4.5, .5) {};
            \node (c1) at (5.25, .43) {};
            \node (c2) at (5.5, .5) {};
            \node (c3) at (5.75, .43) {};

            \draw(a)--(b)--(c);
            \draw(a)--(a1);
            \draw(b)--(b1);
            \draw(c)--(c1);
            \draw(c)--(c2);
            \draw(c)--(c3);

            \node[fill=none,scale=2] at (6.25, .25) {$-$};

            \node (a) at (7, 0) {};
            \node (b) at (8, 0) {};
            \node (c) at (9, 0) {};

            \node (a1) at (6.8, 0.5) {};
            \node (a2) at (7.2, .5) {};
            
            \node (c1) at (9.25, .43) {};
            \node (c2) at (9, .5) {};
            \node (c3) at (8.75, .43) {};

            \draw(a)--(b)--(c);
            \draw(a)--(a1);
            \draw(a)--(a2);
            \draw(c)--(c1);
            \draw(c)--(c2);
            \draw(c)--(c3);
            
            \node[fill=none, scale=2] at (9.75,.25) {$+$};

            \node (a) at (10.5, 0) {};
            \node (b) at (11.5, 0) {};
            \node (c) at (12.5, 0) {};

            \node (a1) at (10.5, 0.5) {};

            \node (c1) at (12.15, .35) {};
            \node (c2) at (12.37, .49) {};
            \node (c3) at (12.63, .49) {};
            \node(c4) at (12.85, .35) {};

            \draw(a)--(b)--(c);
            \draw(a)--(a1);
            \draw(c)--(c4);
            \draw(c)--(c1);
            \draw(c)--(c2);
            \draw(c)--(c3);
        \end{tikzpicture}
        \caption{The linear relation $\mathbf{X}_{C[2,4,2]} = \mathbf{X}_{C[2,2,4]} - \mathbf{X}_{C[3,1,4]} + \mathbf{X}_{C[2,1,5]}$.}
        \label{fig:linear-combin-2-int-edges}
    \end{figure}
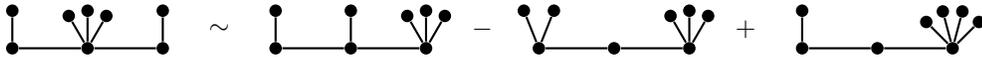
\end{example}

We can give a complete description of the linear combinations that occur in trees with 2 internal edges. Note that a tree with 2 internal edges is a caterpillar with three leaf components. Hence given a partition with three parts $(a,b,c)$, there are at most three non-isomorphic trees with $(a,b,c)$ as their leading partition. In particular, these are $C[a,b,c], C[b,c,a], C[b,a,c]$. Note that these trees may be isomorphic if $(a,b,c)$ has repeated parts. Further, $\mathbf{X}_{C[b,c,a]}$ is in the caterpillar basis by definition, so we only need to find the linear combinations that $C[a,b,c]$ and $C[b,a,c]$ satisfy.

\begin{proposition} Let $T$ be a tree with two internal edges, with leading partition $\lead = (a, b, c)$, then $T = C[a,b,c], C[b,c,a]$ or $C[b,a,c]$, we have the following linear combinations:
    \begin{align}
        \mathbf{X}_{C[a,b,c]} &= \mathbf{X}_{C[b,c,a]} - \mathbf{X}_{C[b,1,a+c-1]} + \mathbf{X}_{C[b+1,a+c-1]} + \mathbf{X}_{C[c,1,a+b-1]} - \mathbf{X}_{C[c+1,a+b-1]}
        \\
        \mathbf{X}_{C[b,a,c]} &= \mathbf{X}_{C[b,c,a]} - \mathbf{X}_{C[a,1,b+c-1]} + \mathbf{X}_{C[a+1,b+c-1]} + \mathbf{X}_{C[c,1,a+b-1]} - \mathbf{X}_{C[c+1,a+b-1]}
    \end{align}
\end{proposition}

\begin{proof}
    Both relations are obtained by applying the DNC relation on each of the caterpillars on the right-hand side and simplifying the result to obtain the chromatic symmetric function in the left-hand side.   
\end{proof}

\printbibliography

\end{document}